\documentclass[a4paper,10pt]{article}

\usepackage{amsmath, amssymb, amsthm} 

\usepackage{fancybox, framed}
\usepackage{tikz-cd}

\usepackage{braket} 

\usepackage[urlbordercolor={.7 .5 .1}]{hyperref}

\usepackage{pdfpages}

\usepackage{mathtools} 

\usepackage{cite} 

\usepackage{graphicx} 

\usepackage{tabularx} 

\usepackage{multirow} 

\usepackage{enumerate} 

\usepackage{pgf,tikz}  
\usepackage{subcaption}
\usepackage{caption}
\usepackage{mathrsfs}
\usetikzlibrary{arrows} 

\usepackage{gensymb} 

\usepackage{cleveref} 

\theoremstyle{definition}
\newtheorem{definition}{Definition}[section]
\newtheorem{example}[definition]{Example}
\newtheorem{fact}[definition]{Fact}
\theoremstyle{corollary}
\newtheorem{corollary}[definition]{Corollary}
\newtheorem{lemma}[definition]{Lemma} 
\newtheorem{proposition}[definition]{Proposition}
\newtheorem{theorem}[definition]{Theorem}

\usepackage{fancyhdr}
\newcommand{\R}{\boldsymbol{\mathrm{R}}}
\newcommand{\N}{\boldsymbol{\mathrm{N}}}
\newcommand{\I}{\boldsymbol{\mathrm{I}}}
\newcommand{\GL}{\boldsymbol{\mathrm{GL}}}
\newcommand{\id}{\mathrm{id}}
\newcommand{\M}{\mathrm{M}}
\newcommand{\kk}{,}

\title{Geometric Aspects to Diophantine Equations of the Form $ x^2 + zxy + y^2 = M $ and $ z $-Rings}
\date{2024, November}
\author{Chris Busenhart \\ Departement of Mathematics, ETH Zürich}

\begin{document}

\maketitle

\begin{abstract}
\noindent 
In the following we consider Diophantine equations of the form $ x^2+ zxy + y^2 = M $ for given $ M,z \in \mathbb{Z} $ and discuss the number of its (primitive) solutions as well as the construction of them. To reach this goal we introduce $ z $-rings which turn out to be a useful tool to investigate these Diophantine equations. Moreover, we will extend these rings and study the algebraic curves defined by them on a plane by methods inspired by the complex plane. Then we define the so called subbranches which are bounded and connected parts of the algebraic curves containing a representative of each solution of the Diophantine equations with respect to association in $ z $-rings. With the help of them we can easily prove the existence or non-existence of solutions to the above Diophantine equations. Then we divide the integer primes with respect to the different $ z $-rings into two main categories, i.e. the regular and irregular elements. We show that the irregular elements are prime in the corresponding $ z $-rings and we identify that most of the $ z $-rings cannot be unique factorization domains. We determine the number of positive, primitive solutions of the above Diophantine equation if $ M \in \mathbb{N} $ is a product of irregular elements in the corresponding $ z $-ring for $ z \in \mathbb{N} $. We also give an overview how many primitive and non-primitive solutions in a given quadrant we can find for arbitrary $ M,z \in \mathbb{Z} $, especially, if $ M $ is a power of any irregular element. Furthermore, we consider the case $ z = 3 $, determine the regular and irregular elements as well as the number of positive, primitive solutions of the Diophantine equation $ x^2 + 3xy + y^2 = M $ depending on $ M \in \mathbb{N} $.
\end{abstract}

\begin{section}{Introduction and motivation}

The name Diophantine equations goes back to the Greek mathematician Diophantus of Alexandria. He was living in the third century and probably one of the first who examined equations with integer solutions using an advanced algebraic notation for that time. However, he was not the first one who studied Diophantine equations as there exist Babylonian clay tables containing Pythagorean triples which are from around $ 1800 $ BC. Phythagorean triples are integer solutions for the Diophantine equation $ x^2 + y^2 = z^2 $. A more general form of this equation is then the equation $ x^2 + y^2 = M $ where Albert Girard \cite{Girard} was the first who proved that every prime of the form $ 4n+1 $ is the sum of two squares following by a lot of other proofs from Euler \cite{Euler1, Euler2}, Dedekind \cite[p.\,145]{Dedekind_Stillwell} and many others \cite{John_Ewell, Heath-Brown, Zagier, David_Christopher}. Another approach goes back to Minkowski, see \cite{Chapman}[p.\,139-143] or more recent, \cite{Kotschick}. He came up with a theorem named after him which is a useful tool for proving number theoretic statements. In fact, the arguments for its proof are based on purely geometric observations on a lattice in $ \mathbb{R}^n $. This approach is called geometry of numbers \cite{Minkowski, Minkowski2} and it was developed  further, see \cite{Lekerkerker}.

In \cite{Miniatur} we showed how to construct the positive, primitive solutions of the Diophantine equation $ x^2 + y^2 = M $ where $ M $ is a product of primes of the form $ 4n+1 $. Furthermore, if $ M = p_1^{k_1}p_2^{k_2} \dots p_l^{k_l} $ is such a product, then we concluded that there are $ 2^{l-1} $ positive, primitive solutions what $ 4 $ centuries before was deduced by Bernard Frénicle de Bessy \cite{Frenicle} by experimental mathematics, i.e. the study of numerical examples where he recognized that there are exactly $ 2^{l-1} $ primitive right triangles with hypotenuse of length $ M $. Our approach to understand the solutions of the above described equation was to use the Gaussian integers and the fact that they are a unique factorization domain as well as a lot of other tools we know from the complex numbers.

At some point the question came up whether this approach can also be used for other types of Diophantine equations. Indeed, for Diophantine equations of the form $ x^2 + zxy + y^2 = M $ for $ z, M \in \mathbb{Z} $ we can proceed similarly (compare also the more general case \cite[p.\,408-412]{Dickson} and \cite[p.\,387-389]{Lemmermeyer} where Gauss used quadratic forms). In fact, for each $ z \in \mathbb{Z} $ we will define the so called $ z $-ring which have similar properties as the Gaussian integers.
In particular, the geometric model helps to understand the structure of the set of solutions to the above Diophantine equations. Moreover, we will see that there is a strong connection between the geometric and algebraic properties of these $ z $-rings.




\end{section}	
	
\begin{section}{Construction of $ z $-rings}
	For the whole section, let $ \left( a_1, a_2 \right), \left( b_1, b_2 \right), \left( c_1, c_2 \right) \in \mathbb{Z} \times \mathbb{Z} $. Consider the group $ \left( \mathbb{Z} \times \mathbb{Z}, + \right) $ where the addition is defined component wise. We would like to define a product $ * $ which turns $ \R_{z} \coloneqq \left( \mathbb{Z} \times \mathbb{Z}, +, * \right) $ into a ring for all $ z \in \mathbb{Z} $.
	
\begin{definition}
	The {\em $ z $-product} is defined in the following way:
	\begin{displaymath}
		\left( a_1, a_2 \right) * \left( b_1, b_2 \right) \coloneqq \left( a_1 b_1 - a_2 b_2, a_1 b_2 + a_2 b_1 + z a_2 b_2 \right).
	\end{displaymath}
\end{definition}	
	
Note that the $ z $-product depends on $ z $, whereas this is not the case for addition in $ z $-rings. By identifying $ \left( a_1,a_2 \right) $ with $ a_1 + a_2 i $ where $ i $ is the complex unit with $ i^2 = -1 $ we clearly see that $ \R_{0} $ is isomorphic to the Gaussian integers $ \mathbb{Z}[i] $. In fact, $ \R_{z} $ is a commutative ring for all $ z \in \mathbb{Z} $.
	
\begin{proposition} \label{prop_ring}
	$ \R_{z} $ is a commutative and unitary ring for all $ z \in \mathbb{Z} $.
\end{proposition}			
	
\begin{proof}
	Fix $ z \in \mathbb{Z} $. Then $ \left( \mathbb{Z} \times \mathbb{Z}, + \right) $ is an abelian group with neutral element $ \left( 0,0 \right) \in \mathbb{Z} \times \mathbb{Z} $ and for $ \left( a_1,a_2 \right) \in \mathbb{Z} \times \mathbb{Z} $ we have that $ \left( -a_1,-a_2 \right) \in \mathbb{Z} \times \mathbb{Z} $ is the inverse of it.
The $ z $-product is commutative because of its symmetry: If we exchange $ a_j $ by $ b_j $ for $ j = 1,2 $, respectively, then the value of the above product does not change.
Since
	\begin{gather*}
		\left( a_1, a_2 \right) * \big( \left( b_1, b_2 \right)	+ \left( c_1, c_2 \right) \big) = \left( a_1, a_2 \right) *  \left( b_1+c_1, b_2+c_2 \right) \\
		= \big( a_1 \left( b_1 + c_1 \right) - a_2 \left( b_2+c_2 \right),a_1 \left( b_2 + c_2 \right) + a_2 \left( b_1+c_1 \right)+ za_2 \left( b_2+c_2 \right) \big) \\ 
		= \left( a_1b_1-a_2b_2,a_1b_2+a_2b_1+za_2b_2 \right) + \left( a_1c_1-a_2c_2,a_1c_2+a_2c_1+za_2c_2 \right) \\
		= \left( a_1, a_2 \right) * \left( b_1, b_2 \right) + \left( a_1, a_2 \right) * \left( c_1, c_2 \right)
	\end{gather*}
holds, distributivity is satisfied. It remains to show associativity of the $ z $-product. For this we calculate
	\begin{gather*}
		\big( \left( a_1, a_2 \right) * \left( b_1, b_2 \right) \big) * \left( c_1, c_2 \right)
		= \left( a_1 b_1 - a_2 b_2, a_1 b_2 + a_2 b_1 + z a_2 b_2 \right) * \left( c_1, c_2 \right) \\
		= \Big( a_1b_1c_1-a_2b_2c_1-a_1b_2c_2-a_2b_1c_2-za_2b_2c_2, 
		 a_1b_1c_2 + a_1b_2c_1+a_2b_1c_1 \\  za_2b_2c_1+za_1b_2c_2+za_2b_1c_2 + \left( z^2-1 \right) a_2b_2c_2 \Big)
	\end{gather*}
and by commutativity of the $ z $-product, associativity holds if and only if 
	\begin{align*}
		\big( \left( a_1, a_2 \right) * \left( b_1, b_2 \right)	\big) * \left( c_1, c_2 \right)
		&= \big( \left( c_1, c_2 \right) * \left( b_1, b_2 \right)	\big) * \left( a_1, a_2 \right).
	\end{align*} 	
I.e. if we can exchange $ a_j $ and $ c_j $ for $ j = 1,2 $, respectively,  in $ \big( \left( a_1, a_2 \right) * \left( b_1, b_2 \right) \big) * \left( c_1, c_2 \right) $ such that the value of the product does not change, then associativity holds. This symmetry can easily be checked.
\end{proof}	
	
From now on we will call $ \R_{z} $ $ z $-ring for all $ z \in \mathbb{Z} $ and we will identify $ \mathbb{Z} $ with $ \mathbb{Z} \times \{ 0 \} $. This turns $ \mathbb{Z} $ to a subring of $ \R_{z} $. Moreover, if $ k \in \mathbb{Z} $ and $ \alpha \in \mathbb{Z} \times \mathbb{Z} $, then we will interpret
	$$ k \alpha = \left\{ 
\begin{array}{ll}
	\underbrace{\alpha + \dots + \alpha}_{\left\vert k \right\vert \ \textrm{times}} & k \geq 0 \\ [5ex]
	-\underbrace{\left(\alpha + \dots + \alpha\right)}_{\left\vert k \right\vert \ \textrm{times}} & k < 0
\end{array}	
\right. . $$


In the next section we will see that $ \R_{z} $ has similar properties as the Gaussian integers. We will introduce the real and imaginary part, the (mirror) conjugate and the norm. All these definitions are related to what we know from the complex numbers. Moreover, we will prove that $ \R_{z} $ is an integral domain if and only if $ z \notin \left\{ -2,2\right\} $.
\end{section}

\begin{section}{Conjugate, norm, real and imaginary parts}

\begin{definition}
	Let $ \alpha = \left( a_1, a_2 \right) \in \R_{z} $. Then we define the {\em conjugate of $ \alpha $} as
	$$ \overline{\alpha} \coloneqq \left( a_1+za_2,-a_2 \right) .$$
\end{definition}	
	
Observe that the conjugation depends on $ z $, i.e. on the ring we apply it. As for the complex numbers we can define the imaginary and the real part for elements in the $ z $-ring:

\begin{definition}
	Let $ \alpha = \left(a_1,a_2 \right) \in \R_{z} $, then we call $ \mathrm{Re}\left( \alpha \right) \coloneqq a_1 $ the {\em real} and $ \mathrm{Im}\left( \alpha \right) \coloneqq a_2 $ the {\em imaginary part of $ \alpha $}.
\end{definition}	
	
\begin{lemma}
Let $ \alpha, \beta \in \R_{z} $ be arbitrary.
	The conjugation has the following properties:
	\begin{itemize}
		\item[i)] $ \overline{\alpha+\beta} =  \overline{\alpha} +\overline{\beta}$ 
		\item[ii)] $ \overline{\alpha * \beta} = \overline{\alpha} *\overline{\beta}$ 
		\item[iii)] $ \overline{\overline{\alpha}} = \alpha $
		\item[iv)] $ \alpha = \overline{\alpha} $ iff $ \alpha \in \mathbb{Z} $
	\end{itemize}
\end{lemma}
	
\begin{proof}
	Let $ \alpha =  \left( a_1, a_2 \right) \in \R_z $ and $ \beta = \left( b_1, b_2 \right) \in \R_z $, then we have
	\begin{align*}
		\overline{\alpha+\beta} &= \overline{\left(a_1+b_1,a_2+b_2 \right)} \\
		&= \big( a_1+b_1 + z \left( a_2+b_2 \right),-\left(a_2+b_2\right) \big) \\
		&= \big( a_1 + za_2,-a_2 \big) + \big( b_1 + zb_2,-b_2 \big) \\
		&= \overline{\alpha} + \overline{\beta} \\ \\
		\overline{\alpha*\beta} &= \overline{\left( a_1 b_1 - a_2 b_2, a_1 b_2 + a_2 b_1 + z a_2 b_2 \right)}	 \\
		&= \big( a_1 b_1 - a_2 b_2 + z \left( a_1 b_2 + a_2 b_1 + z a_2 b_2 \right),-\left( a_1 b_2 + a_2 b_1 + z a_2 b_2 \right) \big) \\
		&= \big( \left( a_1+za_2 \right) \left( b_1+zb_2 \right) - a_2 b_2, -\left( a_1+za_2 \right)b_2 -a_2\left(b_1+zb_2 \right) + z a_2 b_2  \big) \\	
		&= \left( a_1+za_2,-a_2 \right) * \left( b_1+zb_2,-b_2 \right) \\
		&= \overline{\alpha}*\overline{\beta} \\ \\
\overline{\overline{\alpha}}	&= \overline{\left( a_1+za_2,-a_2 \right)}	\\
		&= \left( a_1+za_2-za_2,a_2 \right)\\
		&=\alpha
	\end{align*}
If $ \alpha = \overline{\alpha} $, i.e. $ \left(a_1,a_2 \right) = \left( a_1+za_2,-a_2 \right) $, then $ a_2 = 0 $ and vice versa.
\end{proof}	
	
\begin{definition}
	The {\em norm of $ \alpha = \left( a_1, a_2 \right) \in \R_{z} $} is defined as
	\begin{displaymath}
		\N\left(\alpha\right) \coloneqq a_1^2 + z a_1 a_2 + a_2^2.
	\end{displaymath}
\end{definition}

Observe that our norm is not a proper norm in a strictly mathematical sense. For example, $ \R_{z} $ contains elements which have negative norm if and only if $ \lvert z \rvert \geq 3 $. If $ z = 0 $, then the norm coincides with the squared standard norm of the complex numbers.
	
\begin{lemma} \label{Nz_lemma}
Let $ \alpha, \beta \in \R_{z} $ be arbitrary. The following holds true:
	\begin{itemize}
		\item[i)] $ \N(\alpha) = 0 $ iff $ \alpha = \left( 0,0 \right) $ or $ z = \pm 2 $ and $ \alpha \in \{ \left( \mp\lambda,  \lambda \right) \mid \lambda \in \mathbb{Z} \} $.
		\item[ii)] $ \N(\alpha * \beta) = \N(\alpha) \N(\beta) $	
		\item[iii)] $ \alpha*\overline{\alpha} =  \N\left(\alpha \right) =  \N\left(\overline{\alpha} \right) $
		\item[iv)] $ \N(\alpha) = \pm 1 $ iff $ \alpha $ is a unit. Moreover, if $ \N(\alpha) = \pm 1 $, then $ \pm \overline{\alpha} $ is the inverse of $ \alpha $.
	\end{itemize}
\end{lemma}	
	
\begin{proof}
	\begin{enumerate}
		\item[i)] Let $ \alpha = \left( a_1,a_2 \right) \in \R_z $ and $ \beta = \left( b_1,b_2 \right) \in \R_z $. Assume $ \N\left(\alpha\right) = 0 $, or equivalently, $ a_1^2 + z a_1 a_2 + a_2^2 = 0 $. If $ a_1 \neq 0 \neq a_2 $ we can write $ a_j = \lambda a_j' $ for $ j = 1,2 $ where $ \lambda > 0 $ is the greatest common divisor of $ a_1, a_2 $. Then $ {a_1'}^2 + z a_1' a_2' + {a_2'}^2 = 0 $ holds true which gives us $ a_1' \mid {a_2'}^2 $ and $ a_2' \mid {a_1'}^2 $. Since $ a_1', a_2' $ are relatively prime and different from zero, we see that $ a_1', a_2' \in \left\{ -1,1 \right\} $ and $ z \in \left\{ -2,2 \right\} $ and so the statement follows. The reverse direction is clear.
		\item[ii)] By calculation we see
	\begin{align*}
		\N(\alpha * \beta) &= \N\big(\left(   a_1 b_1 - a_2 b_2, a_1 b_2 + a_2 b_1 + z a_2 b_2  \right) \big) \\
		&= \left( a_1 b_1 - a_2 b_2 \right)^2 + z \left( a_1 b_1 - a_2 b_2 \right) \left( a_1 b_2 + a_2 b_1 + z a_2 b_2 \right) \\
		&+ \left( a_1 b_2 + a_2 b_1 + z a_2 b_2 \right)^2 \\	
		&= \left( a_1^2+za_1a_2+a_2^2 \right) \left( b_1^2+zb_1b_2+b_2^2 \right) \\
		&= \N(\alpha) \N(\beta) 
\end{align*}
		\item[iii)] Moreover,
		$$\alpha * \overline{\alpha} = \left( a_1^2+za_1a_2+a_2^2 \right)	
		= \N(\alpha) .$$
With this we also deduce
$$ \N(\overline{\alpha}) = \overline{\alpha}* \overline{\overline{\alpha}} = \alpha \overline{\alpha} .$$
		\item[iv)] Assume that $ \N \left( \alpha \right) = \pm 1 $, then 
$ \alpha * \left(\pm \overline{\alpha}\right) = \pm \N\left( \alpha \right) = 1 $ and so $ \alpha $ is a unit with inverse $ \pm \overline{\alpha} $. Conversely, if $ \alpha $ is a unit, its norm must be a unit in $ \mathbb{Z} $ because 
$$ \N \left( \alpha \right) \N \left( \overline{\alpha} \right) = \N \left( \alpha \overline{\alpha} \right) = \N \left( \pm 1 \right) = 1 $$ and so we conclude.
	\end{enumerate}
\end{proof}

That $ \N $ is multiplicative can also be proven in a more “creative” way. We can define $ \iota: \R_{z} \hookrightarrow \GL_2 \left( \mathbb{R} \right) $ by mapping $ \left( a,b\right) $ to $ \begin{pmatrix}
 a & -b  \\
 b & a+zb \\
\end{pmatrix} $
and show that $ \iota $ is an embedding as well as that the following diagram commutes:
	
\begin{center}
\begin{tikzcd}
  \R_z \arrow[r, hook, "\iota"] \arrow["\N",d]
    & \GL_2 \left( \mathbb{R} \right) \arrow[d, "\det"] \\
  \mathbb{Z}  \arrow[r, "\id"]
& \mathbb{Z} 
\end{tikzcd}		
\end{center}			

Let $ \alpha, \beta \in \R_z $, then we have
$$ \N \left( \alpha \beta \right) = \det \left( \iota \left( \alpha \beta \right) \right) = \det \left( \iota \left( \alpha \right) \right) \det \left( \iota \left(\beta \right) \right) = \N \left( \alpha \right)\N \left( \beta \right).$$
Hence, $ \N $ inherits its multiplicativity from $ \iota $ and the determinant defined for $ 2 \times 2 $-matrices.

\begin{example}
$ z $-conjugation and $ z $-norm can also be interpreted geometrically: Let elements of $ z $-rings be points on the $ \mathbb{Z}\times\mathbb{Z} $-grid as in \Cref{z_conjugation}.
Consider $ \left( 1,4 \right) \in \R_1 $. Its norm is $ 1^2+4+4^2= 21 $ and so it is contained on the ellipse defined by the equation $ x^2+xy+y^2 = 21 $ over $ \mathbb{R} \times \mathbb{R} $. We know that the conjugate of $ \left( 1,4 \right) $ has the same norm and so it must also lie on the same ellipse and be a point on the grid. To construct $ \overline{\left( 1,4 \right)} $ we can just reflect $ \left( 1,4 \right) $ on the origin and then find another point with the same imaginary part on the ellipse as the reflected point. Hence, we get that $ \overline{\left( 1,4 \right)} = \left( 5,-4\right) $. Analogously, we can show that $ \left( 3,1 \right) \in \R_3 $ and $ \left( -1,2 \right) \in \R_4 $ have norm $ 19 $ and $ -3 $, respectively. Thus $ \overline{\left( 3,1 \right)} = \left( 6,-1 \right) $ and $ \overline{\left( -1,2 \right)} = \left( 7,-2 \right) $ with respect to the corresponding $ z $-rings.
\end{example}

\vspace{5mm}
\begin{figure}[h]  
\begin{center}
\pagestyle{empty}
\definecolor{ttffqq}{rgb}{0.2,1.,0.}
\definecolor{qqzzqq}{rgb}{0.,0.6,0.}
\definecolor{ffxfqq}{rgb}{1.,0.4980392156862745,0.}
\definecolor{qqzzff}{rgb}{0.,0.6,1.}
\definecolor{qqqqff}{rgb}{0.,0.,1.}
\definecolor{ffqqtt}{rgb}{1.,0.,0.2}
\definecolor{ffqqqq}{rgb}{1.,0.,0.}
\definecolor{cqcqcq}{rgb}{0.7529411764705882,0.7529411764705882,0.7529411764705882}
\begin{tikzpicture}[line cap=round,line join=round,>=triangle 45,x=0.8cm,y=0.8cm]
\draw [color=cqcqcq,, xstep=0.8cm,ystep=0.8cm] (-5.6647697463879565,-5.5681579844159845) grid (7.854440470161663,5.604743020996957);
\draw[->,color=black] (-5.6647697463879565,0.) -- (7.854440470161663,0.);
\foreach \x in {-5.,-4.,-3.,-2.,-1.,1.,2.,3.,4.,5.,6.,7.}
\draw[shift={(\x,0)},color=black] (0pt,2pt) -- (0pt,-2pt) node[below] {\footnotesize $\x$};
\draw[->,color=black] (0.,-5.5681579844159845) -- (0.,5.604743020996957);
\foreach \y in {-5.,-4.,-3.,-2.,-1.,1.,2.,3.,4.,5.}
\draw[shift={(0,\y)},color=black] (2pt,0pt) -- (-2pt,0pt) node[left] {\footnotesize $\y$};
\draw[color=black] (0pt,-10pt) node[right] {\footnotesize $0$};
\clip(-5.6647697463879565,-5.5681579844159845) rectangle (7.854440470161663,5.604743020996957);
\draw [rotate around={-45.:(0.,0.)},line width=2.pt,color=ffqqqq] (0.,0.) ellipse (5.184592558726289cm and 2.9933259094191533cm);
\draw [samples=50,domain=-0.99:0.99,rotate around={-135.:(0.,0.)},xshift=0.cm,yshift=0.cm,line width=2.pt,color=qqqqff] plot ({2.7568097504180447*(1+(\x)^2)/(1-(\x)^2)},{6.164414002968976*2*(\x)/(1-(\x)^2)});
\draw [samples=50,domain=-0.99:0.99,rotate around={-135.:(0.,0.)},xshift=0.cm,yshift=0.cm,line width=2.pt,color=qqqqff] plot ({2.7568097504180447*(-1-(\x)^2)/(1-(\x)^2)},{6.164414002968976*(-2)*(\x)/(1-(\x)^2)});
\draw [line width=2.pt,color=qqzzff,domain=-5.6647697463879565:7.854440470161663] plot(\x,{(-0.-1.*\x)/-3.});
\draw [line width=2.pt,color=ffxfqq,domain=-5.6647697463879565:7.854440470161663] plot(\x,{(-0.-4.*\x)/-1.});
\draw [samples=50,domain=-0.99:0.99,rotate around={-45.:(0.,0.)},xshift=0.cm,yshift=0.cm,line width=2.pt,color=qqzzqq] plot ({1.7320508075688772*(1+(\x)^2)/(1-(\x)^2)},{1.*2*(\x)/(1-(\x)^2)});
\draw [samples=50,domain=-0.99:0.99,rotate around={-45.:(0.,0.)},xshift=0.cm,yshift=0.cm,line width=2.pt,color=qqzzqq] plot ({1.7320508075688772*(-1-(\x)^2)/(1-(\x)^2)},{1.*(-2)*(\x)/(1-(\x)^2)});
\draw [line width=2.pt, color=ttffqq,domain=-5.6647697463879565:7.854440470161663] plot(\x,{(-0.-2.*\x)/1.});
\draw [line width=2.pt, dash pattern=on 4pt,domain=-5.6647697463879565:7.854440470161663] plot(\x,{(-12.-0.*\x)/6.});
\draw [line width=2.pt, dash pattern=on 4pt, domain=-5.6647697463879565:7.854440470161663] plot(\x,{(-24.-0.*\x)/6.});
\draw [line width=2.pt, dash pattern=on 4pt, domain=-5.6647697463879565:7.854440470161663] plot(\x,{(-9.-0.*\x)/9.});
\begin{scriptsize}
\draw[color=ffqqqq] (4.98,4.8) node {\fontsize{10}{0} $x^2+xy+y^2 = 21$};
\draw[color=qqqqff] (4.87,4.2) node {\fontsize{10}{0} $x^2+3xy+y^2 = 19$};
\draw[color=qqzzqq] (4.94,3.6) node {\fontsize{10}{0} $x^2+4xy+y^2 = -3$};
\draw [fill=ffqqtt] (1.,4.) circle (2.5pt);
\draw [fill=qqqqff] (3.,1.) circle (2.5pt);
\draw [fill=black] (0.,0.) circle (2.5pt);
\draw [fill=ffqqqq] (-1.,-4.) circle (2.5pt);
\draw [fill=qqqqff] (-3.,-1.) circle (2.5pt);
\draw [fill=qqqqff] (10.,-3.) circle (2.5pt);
\draw [fill=qqzzqq] (-1.,2.) circle (2.5pt);
\draw [fill=qqzzqq] (1.,-2.) circle (2.5pt);
\draw [fill=qqzzqq] (7.,-2.) circle (2.5pt);
\draw [fill=ffqqqq] (5.,-4.) circle (2.5pt);
\draw [fill=qqqqff] (6.,-1.) circle (2.5pt);


\end{scriptsize}
\end{tikzpicture}
\caption{Geometric interpretation of conjugation}
\label{z_conjugation}
\end{center}
\end{figure}
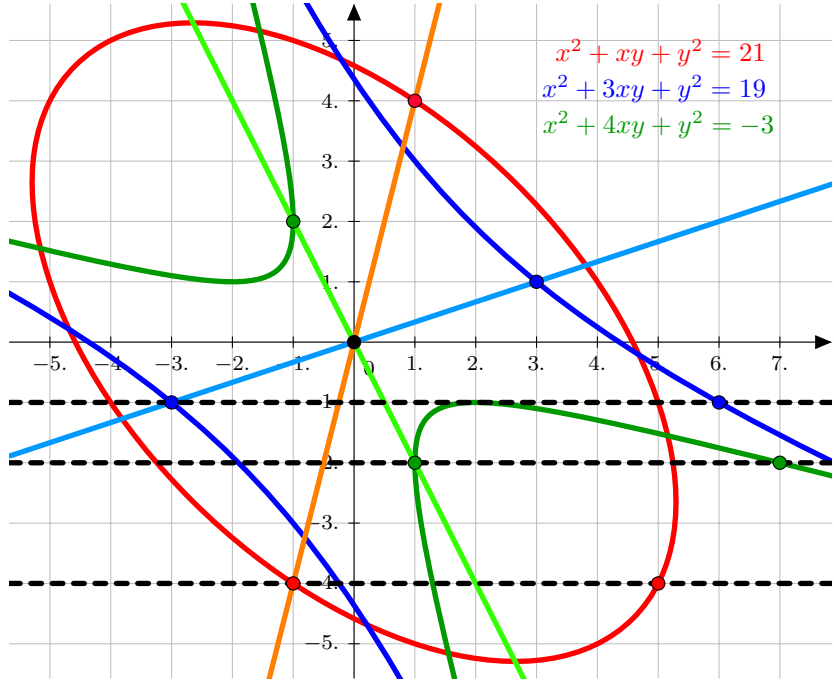
\vspace{5mm}

\begin{corollary}
	$ \textbf{R}_{z} $ is an integral domain iff $ z \notin \left\{ -2,2 \right\} $.
\end{corollary}	
	
\begin{proof}
Let $ \left( 1, \pm 1 \right) \in \R_{\mp 2} $, then
	$$ \left( 1, \pm 1 \right)*\left( 1, \pm 1 \right) = \left( 1-1, \pm 1 \pm 1 \mp 2 \right) = \left( 0,0 \right) .$$	
The rest follows immediately from \Cref{Nz_lemma} and the fact that $ \left( \mathbb{Z},+, \cdot \right) $ is also an integral domain.
\end{proof}

Another useful definition similar to the conjugate is the mirror conjugate which exchanges the real and imaginary parts of an element:

\begin{definition}
	Let 	$ \alpha \in \R_{z} $. Then we call 
$$ \widetilde{\alpha} \coloneqq \left( \mathrm{Im} \left( \alpha \right), \mathrm{Re}\left( \alpha \right) \right) $$
the {\em mirror conjugate of $ \alpha $}.	
\end{definition}	

\begin{lemma} \label{lemma_mirror}
	If 	$ \alpha \in \R_{z} $, then the following identity for the mirror conjugate of $ \alpha $ holds true:
	$$ \widetilde{\alpha} = \left( 0,1 \right) * \overline{ \alpha } $$
\end{lemma}

\begin{proof}
	Let $ \alpha = \left( a_1,a_2 \right) $, then we have
$$ \left( 0,1 \right) * \overline{ \alpha } = \left( 0,1 \right) *  \left( a_1+ za_2,-a_2 \right) = \left(a_2,a_1+za_2-za_2 \right) = \widetilde{\alpha}. $$
\end{proof}

If we consider the elements of $ \textbf{R}_{z} $ as vectors in the $ \mathbb{Z} \times \mathbb{Z} $-plane, then we can calculate the oriented area of the parallelogram which is defined by two such vectors. We will see that this oriented area will play an important role for many results in the following sections.

\begin{definition}
	Consider $ \alpha = \left( a_1, a_2 \right) \in \mathbb{Z} \times \mathbb{Z} $ and $ \beta = \left( b_1,b_2 \right) \in \mathbb{Z} \times \mathbb{Z} $. Then we define
	\begin{align*}
		\big< \alpha , \beta \big> \coloneqq a_1 b_2 - a_2 b_1
	\end{align*}	
and call it the {\em oriented area of $ \alpha, \beta $}.
\end{definition}

Let $ \alpha, \beta \in \mathbb{Z} \times \mathbb{Z} $ be defined as above. 
Then by using “$ \times $” exceptionally as the sign for the cross product we get
\begin{displaymath}
	\begin{pmatrix}
 a_1  \\
 a_2  \\
 0  \\
\end{pmatrix} 
\times
\begin{pmatrix}
 b_1  \\
 b_2  \\
 0  \\
\end{pmatrix}=
\begin{pmatrix}
 0  \\
 0  \\
 a_1 b_2-a_2 b_1  \\
\end{pmatrix} =
\begin{pmatrix}
 0  \\
 0  \\
 \big< \alpha , \beta \big>  \\
\end{pmatrix}.
\end{displaymath}
Hence, the absolute value of $ \big< \alpha , \beta \big> $ is equal to the positive area defined by the parallelogram generated by $ \alpha, \beta $ where $ \alpha, \beta $ are interpreted as vectors in $ \mathbb{Z} \times \mathbb{Z} $. The sign of the oriented area defines the orientation which depends on the order of $  \alpha, \beta $. Therefore the oriented area is anti-commutative and bilinear. 


\begin{lemma} \label{lemma_oriented_area}
	Let $ \alpha, \beta \in \textbf{R}_{z} $, then the following holds true:
	\begin{itemize}
		\item[i)] $ \big< \alpha ,\widetilde{\overline{\alpha}} \big> = \N \left( \alpha \right) $
		\item[ii)] $ \big< \overline{\widetilde{\alpha}}, \alpha \big> = \N \left( \alpha \right) $
		\item[iii)] $ \big< \widetilde{\beta}, \widetilde{\alpha} \big> = \big< \alpha , \beta \big>   $ 
		\item[iv)] $ \big< \overline{\beta}, \overline{\alpha} \big> = \big< \alpha , \beta \big>$ 
	\end{itemize}
\end{lemma}

\begin{proof}
Let $ \alpha = \left( a_1,a_2 \right) $ and $ \beta = \left( b_1,b_2 \right) $, then we have:
	\begin{enumerate}
		\item[i)] \begin{align*}
		\big< \alpha ,\widetilde{\overline{\alpha}} \big>
		&= \big< \left( a_1, a_2 \right) ,\left( -a_2,a_1+za_2 \right) \big> \\
		&= a_1 \left( a_1+za_2 \right) + a_2^2 \\
		&= \N \left( \alpha \right)
		\end{align*}
		\item[ii)] \begin{align*}
		\big< \overline{\widetilde{\alpha}}, \alpha \big>
		&= \big< \left( a_2+za_1,-a_1 \right), \left( a_1, a_2 \right) \big> \\
		&= \left( a_2+za_1 \right)a_2 + a_1^2 \\
		&= \N \left( \alpha \right)
		\end{align*}
		\item[iii)] \begin{align*}
\big< \overline{\beta}, \overline{\alpha} \big> &= \big< \left( b_1+zb_2, -b_2 \right), \left( a_1+za_2, -a_2 \right) \big> \\ 
		&= -b_1a_2 -zb_2a_2 + zb_2a_2 + a_1b_2 \\
		&= \big< \alpha , \beta \big>
	\end{align*}
		\item[iv)] \begin{align*}
\big< \widetilde{\beta}, \widetilde{\alpha} \big> &= \big< \left( b_2,b_1 \right), \left( a_2,a_1 \right) \big> \\ 
		&= b_2a_1 - b_1a_2  \\
		&= \big< \alpha , \beta \big>
	\end{align*}
	\end{enumerate}	
\end{proof}

In the next lemma we would like to find out about the isomorphy classes of these  $ z $-rings. 
	
\begin{lemma} \label{lem_isomorphy}
	Let $ z_1,z_2 \in \mathbb{Z} $. Then $ \R_{z_1} $ and $ \R_{z_2} $ are isomorphic if and only if $ z_1 = z_2 $ or $ z_1 = -z_2 $. Moreover, $ \R_z $ is isomorphic to $ \mathbb{Z}[x]/(x^2\pm zx+1) $.
\end{lemma}
	
\begin{proof}	
	Let $ \R_{z_1} $ and $ \R_{z_2} $ be isomorphic. Hence, we find a ring isomorphism $ \phi: \R_{z_1} \to \R_{z_2} $. Observe that the inverse of $ \phi $, denoted by $ \phi^{-1} $, is also a ring homomorphism. Since $ \phi $ and $ \phi^{-1} $ are ring homomorphisms, they preserve the neutral elements with respect to addition and multiplication. Moreover, $ \phi $ and $ \phi^{-1} $ must be $ \mathbb{Z} $-linear. Define $ \left(a_1,a_2\right) \coloneqq \phi \left( 0,1 \right) $ and $ \left(b_1,b_2\right) \coloneqq \phi^{-1}\left( 0,1 \right) $. Hence, we get
	\begin{align*}
		\phi \big( \left( 0,1 \right) \big) = a_1 \left( 1,0 \right) + a_2 \left( 0,1 \right)
	\end{align*}
and if we apply $ \phi^{-1} $, then
	\begin{align*}
		\left( 0,1 \right) = a_1 \left( 1,0 \right) + a_2 \left( b_1,b_2 \right).
	\end{align*}
So we deduce $ a_2b_2 = 1 $, i.e. $ a_2 = b_2 \in \left\{ -1,1 \right\} $. By definition of the $ z_1 $-product we have
\begin{align*}
	\left( 0,1 \right) * \left( 0,1 \right) = - \left( 1,0 \right) + z_1 \left( 0,1 \right)
\end{align*}
If we apply $ \phi $, then we get 
\begin{align*}
	\left( a_1,a_2 \right) * \left( a_1,a_2 \right) = \left( a_1^2-a_2^2,2a_1a_2+z_2a_2 \right) = - \left( 1,0 \right) + z_1 \left( a_1,a_2 \right)
\end{align*}	
and so we get the equations
	\begin{align*}
		a_1^2-a_2^2 &= -1+z_1a_1 \\
		2a_1+z_2a_2 &= z_1.
	\end{align*}
Hence, $ a_1 = z_1 $ and therefore $ z_1 = z_2 $ or $ z_1 = -z_2 $ by the second equation.


On the other hand, if $ z_1=z_2 $ then the statement holds clearly true. We would like to define an isomorphism for the case $ z_1 = -z_2 $. Define $ \phi: \textbf{R}_{z_1} \to \textbf{R}_{z_2} $ by $ \phi \left(a,b \right) = \left(a,-b \right) $. This is clearly a bijective ring homomorphism which maps the neutral elements onto each other. Thus, $ \textbf{R}_{z_1} $ and  $\textbf{R}_{z_2} $ are isomorphic iff $ z_1 = z_2 $ or $ z_1 = -z_2 $.

Consider $ \mathbb{Z}[x] $ as a ring endowed with its natural addition and multiplication. Then we can define a $ \mathbb{Z} $-linear and surjective ring homomorphism $ \mathbb{Z}[x] \twoheadrightarrow \R_z $ where $ 1 \in \mathbb{Z}[x] $ is mapped to $ \left( 1,0 \right) \in \R_z $ and $ x \in \mathbb{Z}[x] $ (or $ -x \in \mathbb{Z}[x] $) to $ \left( 0,1 \right) \in \R_z $. The kernel of this ring homomorphism is the $ \mathbb{Z} $-ideal generated by $ x^2-zx+1 $ (or $ x^2+zx+1 )$. By the fundamental theorem on homomorphisms we conclude that $ \mathbb{Z}[x]/(x^2\pm zx+1) $ and $ \R_z $ are isomorphic.
\end{proof}

In fact, the isomorphims defined above also respect the norm and the conjugation. Indeed, if $ \phi $ is defined as above for $ z_1 = - z_2 $ and $ \left(a_1, a_2\right) \in \R_{z_1} $, we have
	\begin{align*}
		\phi \left( \overline{\left( a_1, a_2 \right)} \right) &= \phi \big( \left( a_1+z_1a_2,-a_2 \right) \big) \\
		&= \left( a_1 - z_2a_2,a_2 \right) \\
		&= \overline{\left( a_1, -a_2 \right)} \\
		&= \overline{\phi \left( a_1, a_2 \right)}
	\end{align*}
and
	\begin{align*}
		\N \big( \phi \left( a_1, a_2 \right) \big) 
		&= \N \big( \left( a_1, -a_2 \right) \big) \\
		&= a_1^2 -z_2a_1a_2 + a_2^2 \\
		&= a_1^2 +z_1a_1a_2 + a_2^2 \\
		&= \N\big( \left( a_1, a_2 \right) \big) \\
	\end{align*}
In case $ z_1 = 0 = z_2 $, then the conjugation and $ \phi $ are equal. 



\end{section}

\begin{section}{$ z $-rings and their application}

\subsection{Extension of $ z $-rings}

The aim of this section is to investigate properties of the $ z $-rings such that we can finally deal with the question about the number of positive, primitive solutions to the Diophantine equation $ x^2 + zxy + y^2 = M $ and how we can construct these solutions for a given $ z \in \mathbb{N} $ and some $ M \in \mathbb{N} $. First of all we simplify the notation, then extend the $ z $-rings in a similar way as we can extend the Gaussian integers to the complex numbers.

By \Cref{lem_isomorphy} we can interpret $ \R_{z} $ as the ring $ \mathbb{Z}[i_{z}] $ where $ i_{z} $ is the element which satisfy the equation $ i_{z}^2 - zi_{z} + 1 = 0 $. Then the definition of addition and multiplication (we will often omit the sign for the multiplication) in $ \mathbb{Z}[i_z] $ of $ a_1+a_2i_z,b_1+b_2i_z $ is the following:
	\begin{align*}
		\left(a_1 + a_2i_{z} \right) + \left(b_1 + b_2i_{z}\right) &\coloneqq \left( a_1 + b_1 \right) + \left( a_2 + b_2 \right)i_{z} \\
		\left(a_1 + a_2i_{z} \right) \cdot \left(b_1 + b_2i_{z}\right) &\coloneqq \left( a_1b_1-a_2b_2 \right) + \left( a_1b_2 + a_2b_1 + za_2b_2 \right)i_{z}
	\end{align*}

We will also use the tools we developed in the last section for $ \mathbb{Z}[i_z] $: If $ \alpha \coloneqq a_1 + a_2i_{z} \in \mathbb{Z}[i_{z}] $, we call $ a_1 $ its real and $ a_2 $ its imaginary part, similarly, $ \overline{\alpha} = a_1+za_2-a_2i_{z} $ its conjugate and $ \widetilde{\alpha} = a_2 + a_1i_z $ its mirror conjugate. We write $ \N\left(\alpha\right) = {a_1}^2 + za_1a_2 + a_2^2 $ for the norm of $ \alpha $. Sometimes we write $ \left( a_1, a_2 \right)_{\N\left(\alpha\right)} $ for an element $ \alpha $ to indicate also the value of its norm. Moreover, let $ \beta \coloneqq b_1+b_2i_{z} \in \mathbb{Z}[i_{z}] $. We say that $ \alpha $ and $ \beta $ are associated if there is a unit $ \varepsilon \in \mathbb{Z}[i_{z}] $ such that $ \alpha = \varepsilon \beta $. We write $ \big< \alpha, \beta \big> = a_1b_2-a_2b_1 $ for the oriented area of $ \alpha $ and $ \beta $. The advantage of interpreting $ z $-rings in this way is that computation with elements of these rings is simpler. If $ z=0 $, we will write the complex unit $ i_{0} $ just as $ i $.

As mentioned in the last section we can interpret $ \mathbb{Z} \subset \mathbb{Z}[i_z] $ with the above addition and multiplication as a subring. We will see that prime numbers in this subring $ \mathbb{Z} $ are very important for discussing solutions of Diophantine equations in the form $ x^2+zxy+ y^2 = M $. In case we have a prime number $ p \in \mathbb{Z} $, then we allow $ p $ also to be negative. Otherwise we say $ p \in \mathbb{N} $ is prime. 

In this chapter we will consider the extension ring $ \mathbb{Z}[i_z] \subset \mathbb{R}[i_z] $ and the above notions as norm, conjugate etc. are defined on $ \mathbb{R}[i_z] $ analogously. Then we can consider the plane $ \mathbb{R} \times \mathbb{R}i_z $ which we call complex plane as we know it from the complex numbers. Furthermore, the isomorphism $ \phi_{z,-z}: \mathbb{Z}[i_z] \to \mathbb{Z}[i_{-z}] $ defined by $ \phi_{z,-z} \left( a_1+a_2i_z \right) = a_1-a_2i_{-z} $ can be extended to an isomorphism $ \Phi_{z,-z}:\mathbb{R}[i_z] \to \mathbb{R}[i_{-z}] $ in a natural way by the assignment $ r_1+r_2i_z \mapsto r_1-r_2i_{-z} $ for all $ r_1+r_2i_z \in \mathbb{R}[i_z] $. Then $ \Phi_{z,-z} $ still preserves $ \mathbb{Z} $ and respects the corresponding norm and conjugation functions. To simplify the notation we just write $ \Phi $ if there is no ambiguity.

Let $ a_1 < a_2 $ and $ b_1 < b_2 $ be reel numbers, then we consider $ [a_1,a_2] \times [b_1,b_2]i_z \subset \mathbb{R} \times \mathbb{R}i_z $ as the set containing the elements of $ \mathbb{R}[i_z] $ having their real and imaginary part in the intervals $ [a_1,a_2] $ and $ [b_1,b_2] $, respectively. Similarly, we can extend this definition for open and half open intervals. We numerate the quadrants of the complex plane anti-clockwise starting with the first quadrant being equal to $ [0,\infty) \times [0,\infty)i_z $ and so on until the fourth quadrant $ [0,\infty) \times (-\infty,0]i_z $. 




The following definitions and examples will be important for the coming sections.

\begin{definition}
	Let  $ z,M \in \mathbb{Z} $, then we say that $ \alpha \in \mathbb{Z}[i_z] $ {\em solves/satisfies the Diophantine equation} or {\em is a solution to the Diophantine equation $ x^2 + zxy+ y^2 = M $} if $ \left\{ \mathrm{Re}\left( \alpha \right), \mathrm{Im}\left( \alpha \right) \right\} $ is a solution of $ x^2 + zxy+ y^2 = M $ for $ M = \N \left( \alpha \right) $. We call this solution {\em positive} if $ \mathrm{Re}\left( \alpha \right) \geq 0 $ and $ \mathrm{Im}\left( \alpha \right) \geq 0 $. We also say that {\em $ M $ is represented (or representable) by $ x^2 + zxy+ y^2 $} if we can find a solution to the Diophantine equation $ x^2 + zxy+ y^2 = M $.
\end{definition}

\begin{definition}
	Let $ z,M \in \mathbb{Z} $. We call 
	$$ S_M = \{ a + b i_z \in \mathbb{R}[i_z] \mid \N \left( a+bi_z \right) = M \} $$ 
the {\em ($ M $-)level set} and its connected components in the complex plane {\em branches} (connected in the sense of path-connected with respect to the standard topology we have on $ \mathbb{R} \times \mathbb{R} $). 
\end{definition}

\begin{example} \label{ex_non_neg_sol}
The Diophantine equation 
	$$ x^2 + zxy + y^2 = M $$
is not solvable for $ \vert z \vert \leq 2 $ and $ M < 0 $ because we have 
	$$ x^2 + zxy + y^2 \geq x^2 -2 \vert xy \vert + y^2 = \left( \vert x \vert - \vert y \vert \right)^2 \geq 0 $$
which is a contradiction. Observe that the above arguments also hold true for $ x,y \in \mathbb{R} $. This shows that we also have $ S_M = \emptyset $ in this case.
\end{example}

\vspace{5mm}
\begin{figure}[h]  
\begin{center}
\pagestyle{empty}

\definecolor{zzqqtt}{rgb}{0.6,0.,0.2}
\definecolor{yqqqyq}{rgb}{0.5019607843137255,0.,0.5019607843137255}
\definecolor{zzwwff}{rgb}{0.6,0.4,1.}
\definecolor{wwccff}{rgb}{0.4,0.8,1.}
\definecolor{ttffqq}{rgb}{0.2,1.,0.}
\definecolor{ffzztt}{rgb}{1.,0.6,0.2}
\definecolor{ffdxqq}{rgb}{1.,0.8431372549019608,0.}
\definecolor{qqzzff}{rgb}{0.,0.6,1.}
\definecolor{wwccqq}{rgb}{0.4,0.8,0.}
\definecolor{qqqqff}{rgb}{0.,0.,1.}
\definecolor{ffwwqq}{rgb}{1.,0.4,0.}
\definecolor{ffqqqq}{rgb}{1.,0.,0.}
\definecolor{cqcqcq}{rgb}{0.7529411764705882,0.7529411764705882,0.7529411764705882}
\begin{tikzpicture}[line cap=round,line join=round,>=triangle 45,x=0.535cm,y=0.535cm]
\draw [color=cqcqcq,, xstep=0.535cm,ystep=0.535cm] (-8.852120759792612,-6.849976970818427) grid (11.43030182280011,5.765416053070059);
\draw[->,color=black] (-8.852120759792612,0.) -- (11.43030182280011,0.);
\foreach \x in {-8.,-7.,-6.,-5.,-4.,-3.,-2.,-1.,1.,2.,3.,4.,5.,6.,7.,8.,9.,10.,11.}
\draw[shift={(\x,0)},color=black] (0pt,2pt) -- (0pt,-2pt) node[below] {\footnotesize $\x$};
\draw[->,color=black] (0.,-6.849976970818427) -- (0.,5.765416053070059);
\foreach \y in {-6.,-5.,-4.,-3.,-2.,-1.,1.,2.,3.,4.,5.}
\draw[shift={(0,\y)},color=black] (2pt,0pt) -- (-2pt,0pt) node[left] {\footnotesize $\y$};
\draw[color=black] (0pt,-10pt) node[right] {\footnotesize $0$};
\clip(-8.852120759792612,-6.849976970818427) rectangle (11.43030182280011,5.765416053070059);
\draw [samples=50,domain=-0.99:0.99,rotate around={-135.:(0.,0.)},xshift=0.cm,yshift=0.cm,line width=2.pt,color=ffqqqq] plot ({0.5773502691896257*(1+(\x)^2)/(1-(\x)^2)},{1.*2*(\x)/(1-(\x)^2)});
\draw [samples=50,domain=-0.99:0.99,rotate around={-135.:(0.,0.)},xshift=0.cm,yshift=0.cm,line width=2.pt,color=ffqqqq] plot ({0.5773502691896257*(-1-(\x)^2)/(1-(\x)^2)},{1.*(-2)*(\x)/(1-(\x)^2)});
\draw [samples=50,domain=-0.99:0.99,rotate around={-135.:(0.,0.)},xshift=0.cm,yshift=0.cm,line width=2.pt,color=ffwwqq] plot ({1.8257418583505538*(1+(\x)^2)/(1-(\x)^2)},{3.1622776601683795*2*(\x)/(1-(\x)^2)});
\draw [samples=50,domain=-0.99:0.99,rotate around={-135.:(0.,0.)},xshift=0.cm,yshift=0.cm,line width=2.pt,color=ffwwqq] plot ({1.8257418583505538*(-1-(\x)^2)/(1-(\x)^2)},{3.1622776601683795*(-2)*(\x)/(1-(\x)^2)});
\draw [samples=50,domain=-0.99:0.99,rotate around={-45.:(0.,0.)},xshift=0.cm,yshift=0.cm,line width=2.pt,color=qqqqff] plot ({1.*(1+(\x)^2)/(1-(\x)^2)},{0.5773502691896257*2*(\x)/(1-(\x)^2)});
\draw [samples=50,domain=-0.99:0.99,rotate around={-45.:(0.,0.)},xshift=0.cm,yshift=0.cm,line width=2.pt,color=qqqqff] plot ({1.*(-1-(\x)^2)/(1-(\x)^2)},{0.5773502691896257*(-2)*(\x)/(1-(\x)^2)});
\draw [samples=50,domain=-0.99:0.99,rotate around={-135.:(0.,0.)},xshift=0.cm,yshift=0.cm,line width=2.pt,color=wwccqq] plot ({5.597618541248889*(1+(\x)^2)/(1-(\x)^2)},{9.695359714832659*2*(\x)/(1-(\x)^2)});
\draw [samples=50,domain=-0.99:0.99,rotate around={-135.:(0.,0.)},xshift=0.cm,yshift=0.cm,line width=2.pt,color=wwccqq] plot ({5.597618541248889*(-1-(\x)^2)/(1-(\x)^2)},{9.695359714832659*(-2)*(\x)/(1-(\x)^2)});
\draw [samples=50,domain=-0.99:0.99,rotate around={-45.:(0.,0.)},xshift=0.cm,yshift=0.cm,line width=2.pt,color=qqzzff] plot ({1.7320508075688772*(1+(\x)^2)/(1-(\x)^2)},{1.*2*(\x)/(1-(\x)^2)});
\draw [samples=50,domain=-0.99:0.99,rotate around={-45.:(0.,0.)},xshift=0.cm,yshift=0.cm,line width=2.pt,color=qqzzff] plot ({1.7320508075688772*(-1-(\x)^2)/(1-(\x)^2)},{1.*(-2)*(\x)/(1-(\x)^2)});
\draw [samples=50,domain=-0.99:0.99,rotate around={-135.:(0.,0.)},xshift=0.cm,yshift=0.cm,line width=2.pt,color=ffdxqq] plot ({4.08248290463863*(1+(\x)^2)/(1-(\x)^2)},{7.0710678118654755*2*(\x)/(1-(\x)^2)});
\draw [samples=50,domain=-0.99:0.99,rotate around={-135.:(0.,0.)},xshift=0.cm,yshift=0.cm,line width=2.pt,color=ffdxqq] plot ({4.08248290463863*(-1-(\x)^2)/(1-(\x)^2)},{7.0710678118654755*(-2)*(\x)/(1-(\x)^2)});
\draw [samples=50,domain=-0.99:0.99,rotate around={-135.:(0.,0.)},xshift=0.cm,yshift=0.cm,line width=2.pt,color=ffzztt] plot ({2.8284271247461903*(1+(\x)^2)/(1-(\x)^2)},{4.898979485566356*2*(\x)/(1-(\x)^2)});
\draw [samples=50,domain=-0.99:0.99,rotate around={-135.:(0.,0.)},xshift=0.cm,yshift=0.cm,line width=2.pt,color=ffzztt] plot ({2.8284271247461903*(-1-(\x)^2)/(1-(\x)^2)},{4.898979485566356*(-2)*(\x)/(1-(\x)^2)});
\draw [samples=50,domain=-0.99:0.99,rotate around={-45.:(0.,0.)},xshift=0.cm,yshift=0.cm,line width=2.pt,color=wwccff] plot ({3.1622776601683795*(1+(\x)^2)/(1-(\x)^2)},{1.8257418583505538*2*(\x)/(1-(\x)^2)});
\draw [samples=50,domain=-0.99:0.99,rotate around={-45.:(0.,0.)},xshift=0.cm,yshift=0.cm,line width=2.pt,color=wwccff] plot ({3.1622776601683795*(-1-(\x)^2)/(1-(\x)^2)},{1.8257418583505538*(-2)*(\x)/(1-(\x)^2)});
\draw [samples=50,domain=-0.99:0.99,rotate around={-45.:(0.,0.)},xshift=0.cm,yshift=0.cm,line width=2.pt,color=zzwwff] plot ({4.795831523312719*(1+(\x)^2)/(1-(\x)^2)},{2.7688746209726918*2*(\x)/(1-(\x)^2)});
\draw [samples=50,domain=-0.99:0.99,rotate around={-45.:(0.,0.)},xshift=0.cm,yshift=0.cm,line width=2.pt,color=zzwwff] plot ({4.795831523312719*(-1-(\x)^2)/(1-(\x)^2)},{2.7688746209726918*(-2)*(\x)/(1-(\x)^2)});
\draw [samples=50,domain=-0.99:0.99,rotate around={-45.:(0.,0.)},xshift=0.cm,yshift=0.cm,line width=2.pt,color=yqqqyq] plot ({5.656854249492381*(1+(\x)^2)/(1-(\x)^2)},{3.265986323710904*2*(\x)/(1-(\x)^2)});
\draw [samples=50,domain=-0.99:0.99,rotate around={-45.:(0.,0.)},xshift=0.cm,yshift=0.cm,line width=2.pt,color=yqqqyq] plot ({5.656854249492381*(-1-(\x)^2)/(1-(\x)^2)},{3.265986323710904*(-2)*(\x)/(1-(\x)^2)});
\draw [samples=50,domain=-0.99:0.99,rotate around={-45.:(0.,0.)},xshift=0.cm,yshift=0.cm,line width=2.pt,color=zzqqtt] plot ({7.0710678118654755*(1+(\x)^2)/(1-(\x)^2)},{4.08248290463863*2*(\x)/(1-(\x)^2)});
\draw [samples=50,domain=-0.99:0.99,rotate around={-45.:(0.,0.)},xshift=0.cm,yshift=0.cm,line width=2.pt,color=zzqqtt] plot ({7.0710678118654755*(-1-(\x)^2)/(1-(\x)^2)},{4.08248290463863*(-2)*(\x)/(1-(\x)^2)});
\begin{scriptsize}
\draw[color=ffdxqq] (3.2,3.2) node {\fontsize{10}{0} $50$};
\draw[color=wwccff] (-3,2.5) node {\fontsize{10}{0} $-10$};
\draw[color=zzqqtt] (-6.1,5) node {\fontsize{10}{0} $-50$};
\draw [fill=zzqqtt] (-5.,5.) circle (2.0pt);
\draw [fill=zzqqtt] (5.,-5.) circle (2.0pt);
\draw [fill=ffqqqq] (0.,1.) circle (2.0pt);
\draw [fill=ffqqqq] (1.,0) circle (2.0pt);
\draw [fill=ffqqqq] (-1.,0.) circle (2.0pt);
\draw[color=ffqqqq] (0.68,0.68) node {\fontsize{10}{0} $1$};
\draw[color=ffwwqq] (1.55,1.55) node {\fontsize{10}{0} $10$};
\draw [fill=ffqqqq] (-1.,4.) circle (2.0pt);
\draw [fill=ffqqqq] (4.,-1.) circle (2.0pt);
\draw [fill=ffqqqq] (0.,-1.) circle (2.0pt);
\draw [fill=ffqqqq] (1.,-4.) circle (2.0pt);
\draw [fill=ffqqqq] (-4.,1.) circle (2.0pt);
\draw [fill=ffzztt] (2.,2.) circle (2.0pt);
\draw[color=ffzztt] (2.35,2.35) node {\fontsize{10}{0} $24$};
\draw [fill=ffzztt] (-2.,-2.) circle (2.0pt);
\draw [fill=ffzztt] (10.,-2.) circle (2.0pt);
\draw [fill=wwccqq] (5.,3.) circle (2.0pt);
\draw[color=wwccqq] (4.35,4.35) node {\fontsize{10}{0} $94$};
\draw [fill=wwccqq] (3.,5.) circle (2.0pt);
\draw [fill=wwccqq] (-3.,-5.) circle (2.0pt);
\draw [fill=wwccqq] (-5.,-3.) circle (2.0pt);
\draw [fill=qqzzff] (-7.,2.) circle (2.0pt);
\draw [fill=qqzzff] (-2.,1.) circle (2.0pt);
\draw [fill=qqzzff] (-1.,2.) circle (2.0pt);
\draw[color=qqzzff] (-1.88,1.58) node {\fontsize{10}{0} $-3$};
\draw [fill=qqzzff] (2.,-1.) circle (2.0pt);
\draw [fill=qqzzff] (1.,-2.) circle (2.0pt);
\draw[color=qqqqff] (-0.6,0.4) node {\fontsize{10}{0} $-1$};
\draw [fill=qqzzff] (7.,-2.) circle (2.0pt);
\draw [fill=zzwwff] (-4.,3.) circle (2.0pt);
\draw[color=zzwwff] (-4.15,3.55) node {\fontsize{10}{0} $-23$};
\draw [fill=zzwwff] (-3.,4.) circle (2.0pt);
\draw [fill=zzwwff] (4.,-3.) circle (2.0pt);
\draw [fill=zzwwff] (3.,-4.) circle (2.0pt);
\draw [fill=zzwwff] (8.,-3.) circle (2.0pt);
\draw [fill=yqqqyq] (-4.,4.) circle (2.0pt);
\draw[color=yqqqyq] (-5,4.1) node {\fontsize{10}{0} $-32$};
\draw [fill=yqqqyq] (4.,-4.) circle (2.0pt);

\end{scriptsize}
\end{tikzpicture}
\caption{Some level sets in $ \mathbb{R} \times \mathbb{R}i_4 $}
\label{level_sets}
\end{center}
\end{figure}
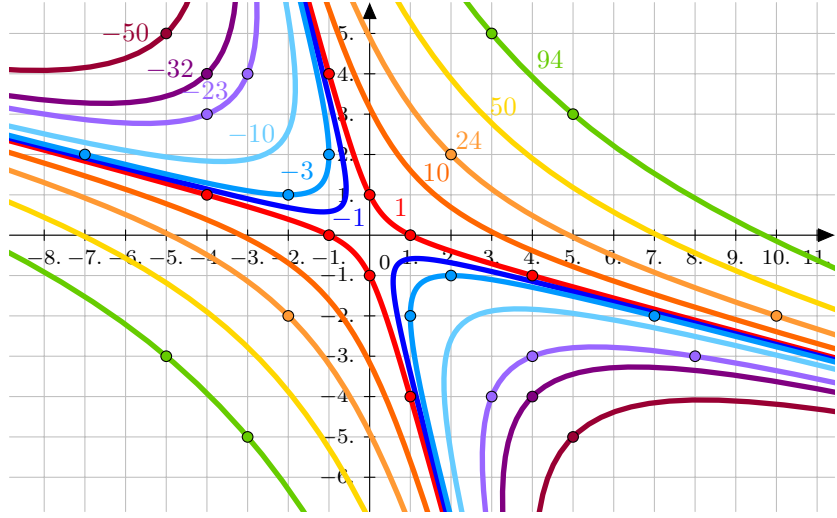
\vspace{5mm}

\begin{example} \label{ex_level_set}
	In \Cref{level_sets} you can see different level sets with respect to the ring $ \mathbb{R}[i_4] $ where each level set consists of two branches (they are in the same color). Some of them intersect the $ \left( \mathbb{Z} \times \mathbb{Z}i_z \right) $-grid (then the points are indicated) and some of them do not. For example, we see that $ -1+4i_4,i_4,1,4-i_4 $ are contained in the $ 1 $-level set, whereas the $ -1 $-level set does not seem to intersect the considered part of the $ \left( \mathbb{Z} \times \mathbb{Z}i_z \right) $-grid. We will see later that from such local considerations we can indeed conclude the non-solvability of the Diophantine equation $ x^2 + 4xy + y^2 = -1 $.
\end{example}

\begin{example} \label{ex_unit_ellipse}
	Let $ z = -1 $, then $ i_{-1}^2+i_{-1}+1=0 $ and so $ \mathbb{Z}[i_{-1}] $ is isomorphic to the Eisenstein (or sometimes also called Eulerian) integers (see \cite[p.\,67f]{Cox}). We would like to determine all units of $ \mathbb{Z}[i_{-1}] $. By \Cref{Nz_lemma} we know that the units in $ \mathbb{Z}[i_{-1}] $ are the elements with norm equal to $ \pm 1 $. Hence, the units of $ \mathbb{Z}[i_{-1}] $ are exactly the points on the $ 1 $-level set intersecting the $ \left(\mathbb{Z} \times \mathbb{Z}i_z\right) $-grid because the $ \left(-1\right) $-level set is empty by \Cref{ex_non_neg_sol}. By multiplying the imaginary parts of these units by $ -1 $ we get the units of $ \mathbb{Z}[i_{1}] $. These units are all generated by $ i_1 $, i.e.
	\begin{align*}
		i_1^0 &= 1 \\
		i_1^1 &= i_1 \\
		i_1^2 &= i_1-1 \\
		i_1^3 &= -1 \\
		i_1^4 &= -i_1 \\
		i_1^5 &= -i_1+1
	\end{align*}
where the multiplicative order of $ i_1 $ is $ 6 $. Moreover,
	$$ i_{-1}^3 = i_{-1}\left(-i_{-1}-1\right) = -i_{-1}^2 -i_{-1} = 1 $$
which shows that the multiplicative order of $ i_{-1} $ is $ 3 $. Is this a contradiction to the isomorphy of $ \mathbb{Z}[i_1] $ and $ \mathbb{Z}[i_{-1}] $? Not at all as $ \Phi \left( i_1 \right) = -i_{-1} $. Therefore $ -i_{-1} $ is a generator of the units in $ \mathbb{Z}[i_{-1}] $. Indeed, it is easy to see that $ -i_{-1} $ generates all the units indicated in \Cref{eisenstein_integers} anti-clockwise starting with $ 1 $.
\end{example}

\vspace{5mm}
\begin{figure}[h]  
\begin{center}
\pagestyle{empty}

\definecolor{ffqqqq}{rgb}{1.,0.,0.}
\definecolor{qqqqff}{rgb}{0.,0.,1.}
\definecolor{cqcqcq}{rgb}{0.7529411764705882,0.7529411764705882,0.7529411764705882}
\begin{tikzpicture}[line cap=round,line join=round,>=triangle 45,x=2.0cm,y=2.0cm]
\draw [color=cqcqcq,, xstep=1.0cm,ystep=1.0cm] (-1.8526010569812594,-1.610569591375662) grid (1.9,1.630766345311326);
\draw[->,color=black] (-1.8526010569812594,0.) -- (1.9,0.);
\foreach \x in {-1.5,-1.,-0.5,0.5,1.,1.5}
\draw[shift={(\x,0)},color=black] (0pt,2pt) -- (0pt,-2pt) node[below] {\footnotesize $\x$};
\draw[->,color=black] (0.,-1.610569591375662) -- (0.,1.630766345311326);
\foreach \y in {-1.5,-1.,-0.5,0.5,1.,1.5}
\draw[shift={(0,\y)},color=black] (2pt,0pt) -- (-2pt,0pt) node[left] {\footnotesize $\y$};
\draw[color=black] (0pt,-10pt) node[right] {\footnotesize $0$};
\clip(-2.2,-1.610569591375662) rectangle (2.3,1.630766345311326);
\draw [rotate around={45.:(0.,0.)},line width=2.pt,color=qqqqff] (0.,0.) ellipse (2.8284271247461903cm and 1.632993161855452cm);
\begin{scriptsize}
\draw [fill=ffqqqq] (0.,1.) circle (2.5pt);
\draw [fill=ffqqqq] (1.,1.) circle (2.5pt);
\draw [fill=ffqqqq] (1.,0.) circle (2.5pt);
\draw [fill=ffqqqq] (0.,-1.) circle (2.5pt);
\draw [fill=ffqqqq] (-1.,-1.) circle (2.5pt);
\draw [fill=ffqqqq] (-1.,0.) circle (2.5pt);

\end{scriptsize}
\end{tikzpicture}
\caption{Units of the Eisenstein integers}
\label{eisenstein_integers}
\end{center}
\end{figure}
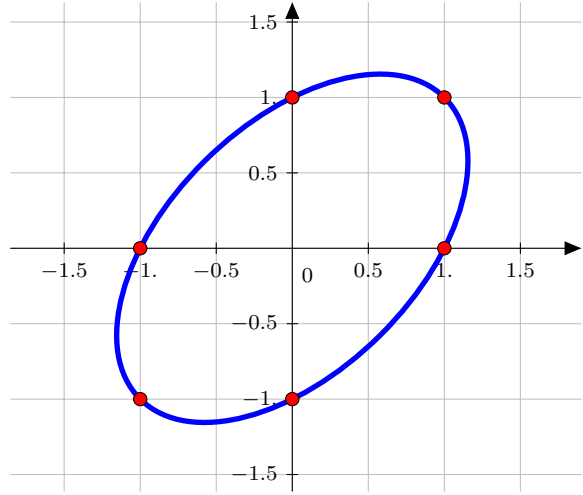
\vspace{5mm}

If we have two elements of a $ z $-ring on a given level set with oriented area equal to zero, then the following statement about their location will be useful.

\begin{lemma} \label{lem_orient_prod_zero}
Let $ z \in \mathbb{Z} $, $ M \in \mathbb{Z} \setminus \left\{ 0 \right\} $ and $ \alpha, \beta \in S_M \subset \mathbb{R}[i_z] $. If $ \big< \alpha, \beta \big> = 0 $, then $ \beta \in \left\{ -\alpha, \alpha \right\} $.
\end{lemma}

\begin{proof}
	Let $ \alpha = a_1 + a_2i_z \in \mathbb{Z}[i_z] $ and $ \beta = b_1 + b_2i_z \in \mathbb{Z}[i_z] $ with
	$$ \big< \alpha, \beta \big> = a_1b_2 - a_2b_1 = 0 .$$
Then $ b_1,b_2 $ cannot be zero at the same time because $ \beta \in S_M $ and $ M \neq 0 $. Therefore if $ b_1 = 0 $, then also $ a_1 = 0 $ and we can define $ \lambda \coloneqq \tfrac{b_2}{a_2} $. Otherwise if $ b_1 \neq 0 $, then $ a_2 = \tfrac{a_1b_2}{b_1} $ and and we set $ \lambda \coloneqq \tfrac{b_1}{a_1} $. In both cases we see that we find $ \lambda \in \mathbb{R} $ such that $ \beta = \lambda \alpha $.

Since $ \alpha, \beta \in S_M $, we have that 
	$$ a_1^2+za_1a_2+a_2^2 = M = \lambda^2 \left( a_1^2+za_1a_2+a_2^2 \right) $$  
and so we get $ \lambda \in \left\{ -1,1 \right\} $.
\end{proof}

Observe that a level set can contain at most two different branches because the level sets are defined by a quadratic equation. If $ \vert z \vert \leq 1 $, then each level set is one branch (compare with $ S_{1} \subset \mathbb{R}[i_{-1}] $ in \Cref{eisenstein_integers}). Branches can also consist of just one element, e.g. $ S_{0} \subset \mathbb{R}[i_z] $ if $ z \notin \left\{ -2,2 \right\} $. However, if $ \vert z \vert > 1 $, then all level sets $ S_M $ for $ M \in \mathbb{Z} \setminus \{ 0 \} $ consists of two branches. In this case we would like to distinguish them which we can do by “separation”.

\begin{definition}
The set 
$$ l_{\lambda_1,\lambda_2} \coloneqq \{ b_1 + b_2i_z \in \mathbb{R}[i_z] \mid \lambda_1 b_1 + \lambda_2 b_2 = 0 \} $$
with $ \lambda_1,\lambda_2 \in \mathbb{Z} $ not both zero is called {\em line in the complex plane (through the origin)}. If $ \lambda_1 \in \mathbb{Z} $ and $ \lambda_2 \in \mathbb{N} \setminus \{ 0 \} $ we say that $ \alpha \coloneqq a_1 + a_2i_z \in \mathbb{R}[i_z] $ {\em is/lies above $ l_{\lambda_1,\lambda_2} $} if $ \lambda_1 a_1 + \lambda_2 a_2 > 0 $, {\em below $ l_{\lambda_1,\lambda_2} $} if $ \lambda_1 a_1 + \lambda_2 a_2 < 0 $ and {\em on $ l_{\lambda_1,\lambda_2} $} if $ a_1+a_2i_z \in l_{\lambda_1,\lambda_2} $. If $ M,z \in \mathbb{Z} $, then we say that $ l_{\lambda_1,\lambda_2} $ {\em separates a level set $ S_M \subset \mathbb{R}[i_z] $} if and only if $ l_{\lambda_1,\lambda_2} \cap S_M = \emptyset $ and there exist $ \gamma_1,\gamma_2 \in S_M $ such that one of the elements lies above and the other one below $ l_{\lambda_1,\lambda_2} $.
\end{definition}

\begin{lemma} \label{lem_at_most_two_solutions}
	Let $ z \in \mathbb{Z} $, $ M \in \mathbb{Z} \setminus \{ 0 \} $ and $ \lambda_1,\lambda_2 \in \mathbb{Z} $ be not both zero. Then the set $ l_{\lambda_1,\lambda_2} \cap S_M $ is either empty or contains exactly two solutions. Moreover, if $ \gamma_1, \gamma_2 \in l_{\lambda_1,\lambda_2} \cap S_M $ and $ \gamma_1 \neq \gamma_2 $, then $ \gamma_1 = -\gamma_2 $.
\end{lemma}

\begin{proof}
	That there are no more solutions than two is clear since a conic and a line can intersect in two points at most. Moreover, if there is a solution $ \gamma \in l_{\lambda_1,\lambda_2} \cap S_M $, then $ -\gamma $ is different from $ \gamma $ (as $ \gamma \neq 0 $) and both of them have the same norm and they lie on the same line through the origin.
\end{proof}

\subsection{The functions $ \I_{+},\I_{-} $ and their properties}

In this section we will introduce the functions $ \I_{+},\I_{-} $ i.e. multiplication with the imaginary units $ \pm i_z $. Especially for subbranches and closed branches as well as for characterizing the unit groups of the $ z $-rings these functions will be important.




\begin{definition} \label{def_algebraic}
	Let $ z \in \mathbb{Z} $. Define
	$ \I_{+}: \mathbb{R}[i_z] \to \mathbb{R}[i_z] $ by $ \I_{+}(\alpha) = i_z \alpha $ and $ \I_{-}: \mathbb{R}[i_z] \to \mathbb{R}[i_z] $ by $ \I_{-}(\alpha) = -i_z \alpha $, then we call $ \I_{+},\I_{-} $ {\em positive} and {\em negative imaginary unit multiplication function}, respectively. For $ n \in \mathbb{Z} $ we also write $ \I_{+}^n $ or $ \I_{-}^n $ for applying $ \I_{+} $, $ \I_{-} $, or, their inverses, $ \I_{+}^{-1} $, $ \I_{-}^{-1} $ $ \left\vert n \right\vert $ times  depending on the sign of $ n $. $ \I_{+}^{0} $ and $ \I_{-}^{0} $ denote the identity functions. 
\end{definition}




To prove a statement about properties of $ \I_{+} $, we will use the fact:


\begin{fact} \label{fact_area}
	Let $ A \in \mathbb{R}^{2 \times 2} $ and $ b,c \in \mathbb{R}^2 $. Then the area (could also be negative depending on the orientation of the vectors) of the parallelogram defined by the vectors $ Ab,Ac $ is equal to the area of the parallelogram defined by $ b,c $ times $ \det(A) $.
\end{fact}

\begin{proposition}[Multiplication with the imaginary unit] \label{proposition_prop_iz}
	Let $ z,w \in \mathbb{Z} $, $ \alpha, \beta \in \mathbb{R}[i_z]$ and $ S_M \subset \mathbb{R}[i_z] $ be arbitrary, then the following holds true:
	\begin{enumerate} \label{prop_units}
		\item[i)] $ \I_{+} \left( S_M \right) = S_M $ and hence $ \I_{+} $ preserves the norm.
		\item[ii)] If $ z \geq 0 $, then $ \I_{+} $ preserves the branches of $ S_M $ for any $ M \in \mathbb{Z} $.
		\item[iii)] $ \I_{+} $ preserves areas and orientation, i.e. if $ P \subset \mathbb{R}[i_z] $ defines a polygon, then the size of the areas in the complex plane of $ P $ and $ \I_{+}\left( P \right) $ are the same. In particular, we have that $ \big< \I_{+}\left( \alpha \right), \I_{+} \left( \beta \right) \big> = \big< \alpha,  \beta  \big> $.
		\item[iv)] $ \big< \alpha, \I_{+} \left( \alpha \right) \big> = \N \left( \alpha \right) $.
		\item[v)] $ \I_{+} \left( \mathbb{Z}[i_z] \right) \subseteq \mathbb{Z}[i_z] $ and $ \I_{+} \left( \mathbb{R}[i_z]\setminus \mathbb{Z}[i_z] \right) \subseteq \mathbb{R}[i_z]\setminus \mathbb{Z}[i_z] $.
		\item[vi)] $ \I_{+} $ preserves divisors of real and imaginary parts, i.e. $ d \in \mathbb{Z} $ is a common divisor of $ \mathrm{Re}\left(\alpha\right), \mathrm{Im}\left(\alpha\right)$ if and only if $ d $ is a common divisor of $ \mathrm{Re}\left(\I_{+} \left( \alpha \right) \right), \mathrm{Im}\left(\I_{+} \left( \alpha \right)\right) $.
	\end{enumerate}
\end{proposition}

\begin{proof}
	\begin{enumerate}
		\item[i)] That $ \I_{+} $ preserves the value of the norm follows directly by \Cref{Nz_lemma}. Moreover, multiplication with $ i_z $ is reversible because $ i_z $ is a unit which shows  $ \I_{+} \left( S_M \right) = S_M $.
		\item[ii)] We need to show that each element on an arbitrary branch will be mapped to an element on the same branch. By i) this is already clear if $ S_M $ consists of just one branch. Hence, we do not need to consider the cases $ z=0,1 $ (for $ M > 0 $, $ S_M $ is a circle or an ellipse, for $ M<0 $ $ S_M $ is empty and $ S_{0} $ contains just the origin).

Assume now that $ z > 1 $ and $ M \geq 0 $. In case $ M = 0 $, then $ S_M $ is connected (if $ z=2 $, then $ S_M $ is a line and otherwise it is just the origin again by \Cref{Nz_lemma}). We consider now the case that $ M > 0 $. Then we clearly have two branches (compare with \Cref{level_sets} and \Cref{Separating_lines_positive}). These branches are either lines if $ z = 2 $ or they define a hyperbola if $ z > 2  $. We will show now that $ l_{z,2} $ and $ l_{2,z} $ separate the branches. If $ l_{z,2} \cap S_M $ would not be empty, then we find $ x,y \in \mathbb{R} $ such that both equations are satisfied:
	\begin{align*}
		x^2 + zxy + y^2 &= M \\
		zx+2y &= 0
	\end{align*}
However, this is not possible since we can multiply the first equation by $ 4 $ and replace $ 2y = -zx $ and $ 4y^2 = -z^2x^2 $ and then we have
	$$ 4x^2-2z^2x^2+z^2x^2 = \left( 4-z^2 \right)x^2 = 4M $$
where $ M > 0 $ and $ z \geq 2 $. This is a contradiction and hence $ l_{z,2} \cap S_M = \emptyset $. For symmetry reasons the same holds true for $ l_{2,z} $ (the calculation is the same, just $ x $ and $ y $ are exchanged).

Moreover, we have that $ \sqrt{M},-\sqrt{M} \in S_M $ where $ \sqrt{M} $ lies above and $ \sqrt{M} $ below for both $ l_{z,2},l_{2,z} $. Thus, $ l_{z,2} $ and $ l_{2,z} $ separates the two branches in $ S_M $ and both lines have the same elements of $ S_M $ above or below, respectively.




Now let $ a + bi_z \in S_M $ and assume that it is either above or below $ l_{z,2} $. Hence, we have either $ za + 2b > 0 $ or $ za + 2b < 0 $ what we will denote by $ za + 2b \gtrless 0 $ to discuss both cases at the same time. Then 
	$$ \I_{+} \left( a + bi_z \right) = ai_z + bi_z^2 = -b + \left( a + zb \right)i_z .$$ 
Since $ -zb+2 \left( a + zb \right) = 2a + zb \gtrless 0 $ because $ a + bi_z $ lies also above or below $ l_{2,z} $, we get that $ \I_{+} \left( a + bi_z \right) $ is also above or below $ l_{z,2} $, respectively.

We consider the case $ M < 0 $ and $ z \geq 2 $. Then $ S_M $ is empty if $ z = 2 $. Assume now that $ z > 2 $. Then $ S_M $ is a hyperbola with two branches being in the second and fourth quadrant not intersecting the reel and the complex axes (compare again with \Cref{Separating_lines_positive}). Take $ a + bi_z \in S_M $ with $ a \gtrless 0 $ and $ b \lessgtr 0 $ (if $ a > 0 $ and $ b < 0 $ then $ a+bi_z $ lies in the fourth quadrant and otherwise in the second quadrant). Then we have $ \I_{+} \left( a + bi_z \right) = -b + \left( a + zb \right)i_z $, i.e. $ -b \gtrless 0 $. Since $ \I_{+} $ preserves the norm and $ S_M $ has only elements in the second and fourth quadrant, we deduce that $ a + bz \lessgtr 0 $ and so $ \I_{+} \left( a + bi_z \right) $ lies on the same branch as $ a + bi_z $.
	\item[iii)] Define $ \M_{+}: \mathbb{R}^2 \to \mathbb{R}^2 $ by matrix multiplication from the left-hand side of the matrix 
\begin{displaymath}
	{M}^{+} \coloneqq \begin{pmatrix}
 0 & -1 \\
 1 & z \\
\end{pmatrix} 
	\end{displaymath}
and the 	isomorphism $ \Psi: \mathbb{R}[i_z] \to \mathbb{R}^2 $ by $ \Psi \left( a+bi_z \right) = \left(a \ b \right)^T $ (where $ T $ denotes the transpose). Then the following diagram commutes

\begin{center}
\begin{tikzcd}
  \mathbb{R}[i_z] \arrow[r, "\I_{+}"] \arrow["\Psi",d]
    & \mathbb{R}[i_z] \arrow[d, "\Psi"] \\
  \mathbb{R}^2  \arrow[r, "\M_{+}"]
& \mathbb{R}^2 
\end{tikzcd}		
\end{center}		

because
\begin{align*}
	\Psi \left( \I_{+} \left( a+bi_z \right) \right) &= \Psi \left( -b + \left( a + zb \right)i_z \right) \\
	&= 
\begin{pmatrix}
 -b  \\
 a + zb \\
\end{pmatrix} =	
\begin{pmatrix}
 0 & -1  \\
 1 & z  \\
\end{pmatrix} \circ
\begin{pmatrix}
 a  \\
 b  \\
\end{pmatrix} =
\begin{pmatrix}
 0 & -1  \\
 1 & z  \\
\end{pmatrix} \circ \Psi \left(  a+bi_z \right)
\end{align*}
Since $ \det \left( {M}^{+} \right) = 1 $, the area (and the orientation by the sign of the area) of polygons is preserved by $ \M_{+} $ by \Cref{fact_area}. Thus, the same holds true for $ \I_{+} $.

	\item[iv)] Let $ \alpha = a_1 + a_2i_z \in \mathbb{R}[i_z] $, then 
	$$ \I_{+} \left( \alpha \right) = -a_2 + \left( a_1 + za_2 \right)i_z = \widetilde{\overline{\alpha}} $$ and so 
$$ \big< \alpha, \I_{+} \left( \alpha \right)\big> = \big< \alpha, \widetilde{\overline{\alpha}}\big> = \N \left( \alpha \right) $$ 
by \Cref{lemma_oriented_area}. 

	\item[v)] Since $ \mathbb{Z}[i_z] $ is closed as a ring, we get that the multiplication of two elements in  $ \mathbb{Z}[i_z] $ is again in the ring. On the other hand, if there is $ \alpha \in \mathbb{R}[i_z] \setminus \mathbb{Z}[i_z] $ and $ \I_{+} \left( \alpha \right) \in \mathbb{Z}[i_z] $, then the multiplication with the inverse of $ i_z $, namely $ z-i_z $, and $ \I_{+} \left( \alpha \right) $ is $ \alpha $ and so we would have $ \alpha \in \mathbb{Z}[i_z] $ because $ \mathbb{Z}[i_z] $ is closed. Hence, we conclude that also $ \I_{+} \left( \alpha \right) \in \mathbb{R}[i_z] \setminus \mathbb{Z}[i_z] $ if $ \alpha \in \mathbb{R}[i_z] \setminus \mathbb{Z}[i_z] $.
	
	\item[vi)] Finally, let $ d \in \mathbb{Z}$, $ a + bi_z \in \mathbb{Z}[i_z] $ with $ d \mid a $, $ d \mid b $. Then clearly $ d \mid -b $ and $ d \mid a + zb $. Conversely, if $ d \mid -b $ and $ d \mid a + zb $, then $ d \mid b $ and $ d \mid a + bz -bz = a $ which shows the last statement. 	
	\end{enumerate}
\end{proof}	
	

\begin{example} \label{ex_mult_with_imaginary_unit}
	Consider the ring $ \mathbb{Z}[i_3] $, then $ S_{19} $ consists of two branches separated by $ l_{3,2} $ (one above and one below as in \Cref{Separating_lines_positive}). We see that $ \alpha_j $ lies on the same branch as $ \I_{+} \left( \alpha_j \right) $ for $ j = 1,2 $. Similarly, $ S_{-1} $ consists of two branches, one in the second and one in the fourth quadrant of the complex plane. We also have that $ \alpha_j $ and $ \I_{+} \left( \alpha_j \right) $ lies on the same branch for $ j = 3,4 $.
\end{example}

\vspace{5mm}
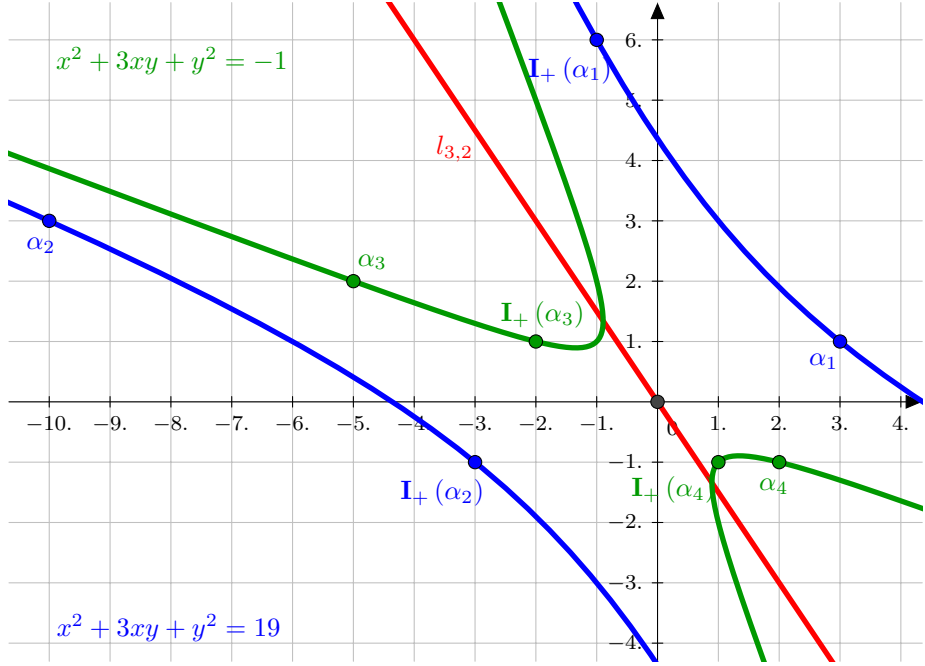
\begin{figure}[h]  
\begin{center}
\pagestyle{empty}

\definecolor{qqzzqq}{rgb}{0.,0.6,0.}
\definecolor{uuuuuu}{rgb}{0.26666666666666666,0.26666666666666666,0.26666666666666666}
\definecolor{ffqqqq}{rgb}{1.,0.,0.}
\definecolor{qqqqff}{rgb}{0.,0.,1.}
\definecolor{cqcqcq}{rgb}{0.7529411764705882,0.7529411764705882,0.7529411764705882}
\begin{tikzpicture}[line cap=round,line join=round,>=triangle 45,x=0.8cm,y=0.8cm]
\draw [color=cqcqcq,, xstep=0.8cm,ystep=0.8cm] (-10.662210136168579,-4.300507372135353) grid (4.35221465831521,6.617389175952259);
\draw[->,color=black] (-10.662210136168579,0.) -- (4.35221465831521,0.);
\foreach \x in {-10.,-9.,-8.,-7.,-6.,-5.,-4.,-3.,-2.,-1.,1.,2.,3.,4.}
\draw[shift={(\x,0)},color=black] (0pt,2pt) -- (0pt,-2pt) node[below] {\footnotesize $\x$};
\draw[->,color=black] (0.,-4.300507372135353) -- (0.,6.617389175952259);
\foreach \y in {-4.,-3.,-2.,-1.,1.,2.,3.,4.,5.,6.}
\draw[shift={(0,\y)},color=black] (2pt,0pt) -- (-2pt,0pt) node[left] {\footnotesize $\y$};
\draw[color=black] (0pt,-10pt) node[right] {\footnotesize $0$};
\clip(-10.662210136168579,-4.300507372135353) rectangle (4.35221465831521,6.617389175952259);
\draw [samples=50,domain=-0.99:0.99,rotate around={-135.:(0.,0.)},xshift=0.cm,yshift=0.cm,line width=2.pt,color=qqqqff] plot ({2.7568097504180447*(1+(\x)^2)/(1-(\x)^2)},{6.164414002968976*2*(\x)/(1-(\x)^2)});
\draw [samples=50,domain=-0.99:0.99,rotate around={-135.:(0.,0.)},xshift=0.cm,yshift=0.cm,line width=2.pt,color=qqqqff] plot ({2.7568097504180447*(-1-(\x)^2)/(1-(\x)^2)},{6.164414002968976*(-2)*(\x)/(1-(\x)^2)});
\draw [line width=2.pt,color=ffqqqq,domain=-10.662210136168579:4.35221465831521] plot(\x,{(-0.-1.5*\x)/1.});
\draw [samples=50,domain=-0.99:0.99,rotate around={-45.:(0.,0.)},xshift=0.cm,yshift=0.cm,line width=2.pt,color=qqzzqq] plot ({1.4142135623730951*(1+(\x)^2)/(1-(\x)^2)},{0.6324555320336759*2*(\x)/(1-(\x)^2)});
\draw [samples=50,domain=-0.99:0.99,rotate around={-45.:(0.,0.)},xshift=0.cm,yshift=0.cm,line width=2.pt,color=qqzzqq] plot ({1.4142135623730951*(-1-(\x)^2)/(1-(\x)^2)},{0.6324555320336759*(-2)*(\x)/(1-(\x)^2)});
\begin{scriptsize}
\draw[color=ffqqqq] (-3.4,4.2) node {\fontsize{10}{0} $ l_{3,2} $};
\draw [fill=qqqqff] (3.,1.) circle (2.5pt);
\draw[color=qqqqff] (2.67,0.67) node {\fontsize{10}{0} $\alpha_1 $};
\draw [fill=qqqqff] (-1.,6.) circle (2.5pt);
\draw[color=qqqqff] (-1.5,5.5) node {\fontsize{10}{0} $ \I_{+} \left( \alpha_1 \right)$};
\draw [fill=uuuuuu] (0.,0.) circle (2.5pt);
\draw [fill=qqqqff] (-10.,3.) circle (2.5pt);
\draw[color=qqqqff] (-10.2,2.6) node {\fontsize{10}{0} $ \alpha_2 $};
\draw [fill=qqqqff] (-3.,-1.) circle (2.5pt);
\draw[color=qqqqff] (-3.6,-1.5) node {\fontsize{10}{0} $ \I_{+} \left( \alpha_2 \right)$};
\draw [fill=qqzzqq] (-5.,2.) circle (2.5pt);
\draw[color=qqzzqq] (-4.75,2.32) node {\fontsize{10}{0} $ \alpha_3 $};
\draw [fill=qqzzqq] (-2.,1.) circle (2.5pt);
\draw[color=qqzzqq] (-1.95,1.45) node {\fontsize{10}{0} $ \I_{+} \left( \alpha_3 \right)$};
\draw [fill=qqzzqq] (1.,-1.) circle (2.5pt);
\draw[color=qqzzqq] (0.2,-1.45) node {\fontsize{10}{0} $ \I_{+} \left( \alpha_4 \right) $};
\draw [fill=qqzzqq] (2.,-1.) circle (2.5pt);
\draw[color=qqzzqq] (1.85,-1.4) node {\fontsize{10}{0} $ \alpha_4 $};
\draw[color=qqqqff] (-8.1,-3.75) node {\fontsize{10}{0} $x^2+3xy+y^2 = 19$};
\draw[color=qqzzqq] (-8.05,5.65) node {\fontsize{10}{0} $x^2+3xy+y^2 = -1$};

\end{scriptsize}
\end{tikzpicture}
\caption{Multiplication with the imaginary unit}
\label{Separating_lines_positive}
\end{center}
\end{figure}
\vspace{5mm}



In fact, ii) in \Cref{proposition_prop_iz} does not hold true for $ z = -4 $ what we will see in the next example.

\begin{example} \label{ex_properties_In}
	Consider $ S_1 \subset \mathbb{R}[i_{-4}] $ and define $ g \coloneqq -i_{-4} $. Then we clearly have $ \I_{+} \left( S_1 \right) = S_1 = \I_{-} \left( S_1 \right) $ because multiplication with units is reversible and does not change the norm as long as the multiplied element has norm equal to $ 1 $ (which is the case, i.e. $ \N \left(i_{-4} \right) = 1 = \N \left(-i_{-4} \right) $). However, we will see that the branches of $ S_1 $ are not preserved by $ \I_{+} $. Since $ i_{-4} $ satisfies the equation $ i_{-4}^2 + 4i_{-4}+1 = 0 $, we can easily deduce that $ g^0 = 1,g^1,g^2,g^{-1} $ are on the same branch of $ S_1 $. Whereas multiplication of $ i_{-4} $ will let a unit change the branch. For example, $ 1 \in S_1 $ is on a different branch than $ \I_{+}\left( 1\right) = i_{-4} \in S_1 $ and the branch containing $ i_{-4} $ seems to be preserved by $ \I_{-} $ as $ \I_{-}^n \left( i_{-4} \right)  $ lies on the same branch for $ n = -2,-1,0,1 $, see \Cref{example_In}. It is therefore plausible that if $ z \in \mathbb{N} $, then $ \I_{-} $ satisfies similar properties for $ \mathbb{R}[i_{-z}] $ as $ \I_{+} $ for $ \mathbb{R}[i_z] $.
\end{example}

\vspace{5mm}
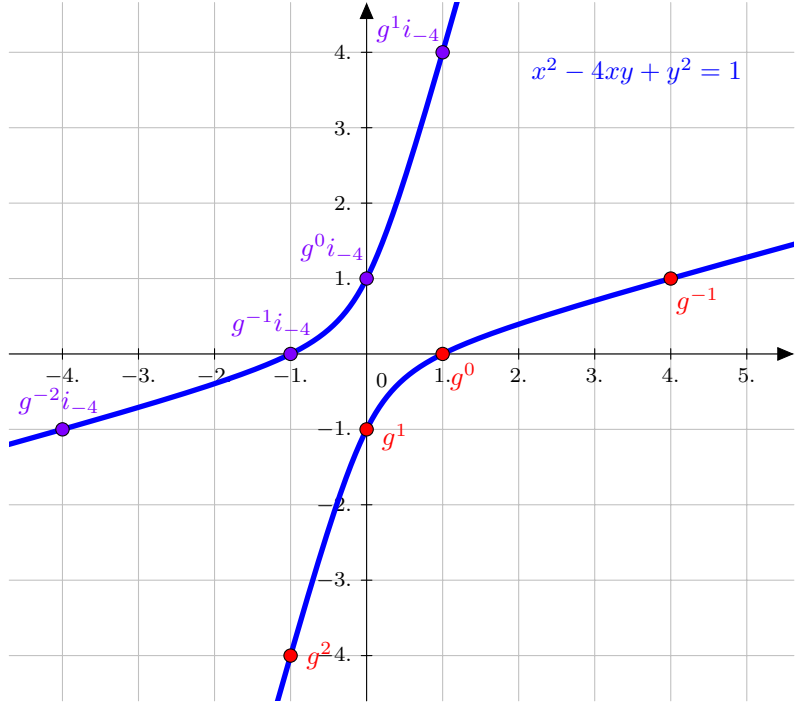
\begin{figure}[h]  
\begin{center}
\pagestyle{empty}

\definecolor{xfqqff}{rgb}{0.4980392156862745,0.,1.}
\definecolor{ffqqqq}{rgb}{1.,0.,0.}
\definecolor{qqqqff}{rgb}{0.,0.,1.}
\definecolor{cqcqcq}{rgb}{0.7529411764705882,0.7529411764705882,0.7529411764705882}
\begin{tikzpicture}[line cap=round,line join=round,>=triangle 45,x=1.0cm,y=1.0cm]
\draw [color=cqcqcq,, xstep=1.0cm,ystep=1.0cm] (-4.7,-4.6) grid (5.62,4.66);
\draw[->,color=black] (-4.7,0.) -- (5.62,0.);
\foreach \x in {-4.,-3.,-2.,-1.,1.,2.,3.,4.,5.}
\draw[shift={(\x,0)},color=black] (0pt,2pt) -- (0pt,-2pt) node[below] {\footnotesize $\x$};
\draw[->,color=black] (0.,-4.6) -- (0.,4.66);
\foreach \y in {-4.,-3.,-2.,-1.,1.,2.,3.,4.}
\draw[shift={(0,\y)},color=black] (2pt,0pt) -- (-2pt,0pt) node[left] {\footnotesize $\y$};
\draw[color=black] (0pt,-10pt) node[right] {\footnotesize $0$};
\clip(-4.7,-4.6) rectangle (5.62,4.66);
\draw [samples=50,domain=-0.99:0.99,rotate around={-45.:(0.,0.)},xshift=0.cm,yshift=0.cm,line width=2.pt,color=qqqqff] plot ({0.5773502691896257*(1+(\x)^2)/(1-(\x)^2)},{1.*2*(\x)/(1-(\x)^2)});
\draw [samples=50,domain=-0.99:0.99,rotate around={-45.:(0.,0.)},xshift=0.cm,yshift=0.cm,line width=2.pt,color=qqqqff] plot ({0.5773502691896257*(-1-(\x)^2)/(1-(\x)^2)},{1.*(-2)*(\x)/(1-(\x)^2)});
\begin{scriptsize}
\draw [fill=ffqqqq] (1.,0.) circle (2.5pt);
\draw[color=ffqqqq] (1.22,-0.3) node {\fontsize{10}{0} $g^0$};
\draw [fill=ffqqqq] (0.,-1.) circle (2.5pt);
\draw[color=ffqqqq] (0.32,-1.1) node {\fontsize{10}{0} $g^1$};
\draw [fill=ffqqqq] (-1.,-4.) circle (2.5pt);
\draw[color=ffqqqq] (-0.67,-4) node {\fontsize{10}{0} $g^2$};
\draw [fill=ffqqqq] (4.,1.) circle (2.5pt);
\draw[color=ffqqqq] (4.3,0.7) node {\fontsize{10}{0} $g^{-1}$};
\draw [fill=xfqqff] (0.,1.) circle (2.5pt);
\draw[color=xfqqff] (-0.5,1.4) node {\fontsize{10}{0} $g^0i_{-4}$};
\draw [fill=xfqqff] (1.,4.) circle (2.5pt);
\draw[color=xfqqff] (0.5,4.3) node {\fontsize{10}{0} $ g^1i_{-4}$};
\draw [fill=xfqqff] (-1.,0.) circle (2.5pt);
\draw[color=xfqqff] (-1.3,0.4) node {\fontsize{10}{0} $ g^{-1}i_{-4}$};
\draw [fill=xfqqff] (-4.,-1.) circle (2.5pt);
\draw[color=xfqqff] (-4.1,-0.63) node {\fontsize{10}{0} $ g^{-2}i_{-4}$};
\draw[color=qqqqff] (3.5,3.75) node {\fontsize{10}{0} $ x^2 -4xy + y^2 = 1 $};

\end{scriptsize}
\end{tikzpicture}
\caption{Units of $ \mathbb{Z}[i_{-4}] $ on $ S_{1} \subseteq \mathbb{R}[i_{-4}] $}
\label{example_In}
\end{center}
\end{figure}
\vspace{5mm}

Since ii) of \Cref{proposition_prop_iz} is generally not true for negative integers $ z $, we also need to work with $ \I_{-} $, the counterpart of $ \I_{+} $. Moreover, we will see that iv) of \Cref{proposition_prop_iz} needs some small adjustment if we want to replace $ \I_{+} $ by $ \I_{-} $.


\begin{corollary} \label{coro_In}
	Let $ z,M \in \mathbb{Z} $, $ \alpha, \beta \in \mathbb{R}[i_z]$ and $ S_M \subset \mathbb{R}[i_z] $ be arbitrary, then the following holds true:	
\begin{enumerate} 
		\item[i)] $ \I_{-} \left( S_M \right) = S_M $ and hence $ \I_{-} $ preserves the norm.
		\item[ii)] If $ z \leq 0 $, then $ \I_{-} $ preserves the branches of $ S_M $ for any $ M \in \mathbb{Z} $.
		\item[iii)] $ \I_{-} $ preserves areas and orientation, i.e. if $ P \subseteq \mathbb{R}[i_z] $ defines a polygon, then the size of the areas in the complex plane of $ P $ and $ \I_{-}\left( P \right) $ are the same. In particular, we have that $ \big< \I_{-} \left( \alpha \right), \I_{-} \left( \beta \right) \big> = \big< \alpha,  \beta  \big> $.
		\item[iv)] $ \big< \I_{-} \left( \alpha \right),\alpha \big> = \N \left( \alpha \right) $.
		\item[v)] $ \I_{-} \left( \mathbb{Z}[i_{z}] \right) \subseteq \mathbb{Z}[i_{z}] $ and $ \I_{-} \left( \mathbb{R}[i_{z}]\setminus \mathbb{Z}[i_{z}] \right) \subseteq \mathbb{R}[i_{z}]\setminus \mathbb{Z}[i_{z}] $.
		\item[vi)] $ \I_{-} $ preserves prime divisors of real and imaginary parts, i.e. $ d \in \mathbb{Z} $ is a common divisor of $ \mathrm{Re}\left(\alpha\right), \mathrm{Im}\left(\alpha\right)$ if and only if $ d $ is a common divisor of $ \mathrm{Re}\left(\I_{-} \left( \alpha \right) \right), \mathrm{Im}\left(\I_{-} \left( \alpha \right)\right) $.
	\end{enumerate}	
\end{corollary}

\begin{proof}	
Let $ \Phi: \mathbb{R}[i_z] \to \mathbb{R}[i_{-z}] $ be the isomorphism defined before and $ \alpha = a_1 + a_2 i_z \in \mathbb{R}[i_z] $, $ \beta = b_1+b_2i_z \in \mathbb{R}[i_z] $. We would like to show that the function $ \I_{-} $ on $ \mathbb{R}[i_{-z}] $ is the equivalent to $ \I_{+} $ on $ \mathbb{R}[i_{z}] $. Indeed, we have $$ \Phi \left( \I_{+} \left( \alpha \right) \right) = \Phi \left( i_z \alpha \right) = \Phi \left( i_z  \right)\Phi \left( \alpha \right) = -i_{-z}\Phi \left( \alpha \right) = \I_{-} \left( \Phi \left( \alpha \right) \right) $$ and so the following diagram commutes:
	 
\begin{center}
\begin{tikzcd}
  \mathbb{R}[i_z] \arrow[r, "\I_{+}"] \arrow["\Phi",d]
    & \mathbb{R}[i_z] \arrow[d, "\Phi"] \\
  \mathbb{R}[i_{-z}] \arrow[r, "\I_{-}"]
& \mathbb{R}[i_{-z}]
\end{tikzcd}		
\end{center}			 
	 
Moreover, let $ \Phi \times \Phi: \mathbb{R}[i_z] \times \mathbb{R}[i_z] \to \mathbb{R}[i_{-z}] \times \mathbb{R}[i_{-z}] $ be the product isomorphism defined by $ \left( \Phi  \times \Phi \right) \left( \alpha, \beta \right) = \left( \Phi \left( \alpha \right), \Phi \left( \beta \right)\right) $. Then we have 
$$ \big< \alpha , \beta \big> = a_1b_2-a_2b_1 = b_1 \left(-a_2 \right) - \left( -b_2\right)a_1 = \big< \Phi \left( \beta \right), \Phi\left( \alpha \right)\big> $$
and therefore the following diagram also commutes because the oriented area is anti-commutative:

\begin{center}
\begin{tikzcd}
  \mathbb{R}[i_z] \times \mathbb{R}[i_z] \arrow[r, "\langle \ \kk \ \rangle"] \arrow["\Phi \times \Phi",d]
    & \mathbb{Z} \arrow[d, "\id"] \\
  \mathbb{R}[i_{-z}] \times \mathbb{R}[i_{-z}] \arrow[r, "-\langle \ \kk \ \rangle"]
& \mathbb{Z}
\end{tikzcd}		
\end{center}	
	 
	 
Hence, i), ii), iv) and v) follow directly from the isomorphy between $ \mathbb{R}[i_z] $, $ \mathbb{R}[i_{-z}] $ and \Cref{proposition_prop_iz}.

iii) is a consequence of \Cref{fact_area} and the following commuting diagram

\begin{center}
\begin{tikzcd}
  \mathbb{R}[i_z] \arrow[r, "\I_{-}"] \arrow["\Psi",d]
    & \mathbb{R}[i_z] \arrow[d, "\Psi"] \\
  \mathbb{R}^2  \arrow[r, "\M_{-}"]
& \mathbb{R}^2 
\end{tikzcd}		
\end{center}			
	
where $ \M_{-}: \mathbb{R}^2 \to \mathbb{R}^2 $ is the function defined by matrix multiplication of 
	\begin{displaymath}
	{M}^{-} \coloneqq \begin{pmatrix}
 0 & 1 \\
 -1 & -z \\
\end{pmatrix} 
	\end{displaymath}
from the left-hand side and $ \Psi: \mathbb{R}[i_z] \mapsto \mathbb{R}^2 $ is defined as in \Cref{proposition_prop_iz}. Then $ \det({M}^{-}) = 1 $.

vi) is a consequence of the \Cref{proposition_prop_iz} and the fact that 
\begin{align*}
	\mathrm{Re}\left( \I_{+} \left( \alpha \right) \right) &= -\mathrm{Re}\left( \I_{-} \left( \alpha \right) \right) \\ 
	\mathrm{Im}\left( \I_{+} \left( \alpha \right) \right) &= -\mathrm{Im}\left( \I_{-} \left( \alpha \right) \right).
\end{align*}  
\end{proof}

\subsection{Partition and local solution theorems}

In this section we will develop a simple criterion to prove or disprove the existence of a solution to the Diophantine equation $ x^2 + zxy + y^2 = M $ for given $ M,z \in \mathbb{Z} $ in general (recall that we already discussed the case if $ M = 0 $, see \Cref{Nz_lemma} and we already know that there is no solution if $ M < 0 $ and $ \vert z \vert \leq 2 $). In case $ \vert z \vert \leq 1 $ and $ M > 0 $ the solutions to the equation above must be in $ [-\sqrt{2M},\sqrt{2M}] \times [-\sqrt{2M},\sqrt{2M}]i_z $ (as $ \sqrt{2M} $ is the smallest radius of a circle such that it entirely contains an ellipse defined by $ x^2 \pm xy+y^2= M $ for both signs) and so the possible solution range is bounded. Therefore if $ \vert z \vert \leq 1 $ we could find at most finitely many solutions in $ \mathbb{Z}[i_z] $. This theoretically means we could prove or disprove the existence of a solution of $ x^2 + zxy + y^2 = M $ by plugging in all elements of $ \left(\left[-\sqrt{2M},\sqrt{2M}\right] \times \left[-\sqrt{2M},\sqrt{2M}\right] \right) \cap \mathbb{Z}[i_z] $ to the Diophantine equation and see whether the equation is satisfied or not. However, this attempt is time-consuming if $ \vert M \vert $ is large. Moreover, if $ \vert z \vert > 1 $, then our solution range is not bounded any more. We will see that it is still possible to deduce the existence or non-existence of solutions to $ x^2 + zxy + y^2 = M $ for given $ z,M $ by local considerations on a bounded and connected part of a branch.

At first we will introduce the so called subbranches. As mentioned before they will be the useful tool to study the solvability of the above Diophantine equations.


\begin{definition}
	Let $ z,M \in \mathbb{Z} $, $ M \neq 0$ and $ \alpha \in S_M $. If $ M > 0 $, then we call
	$$ B_{\alpha} \coloneqq \left\{ 
\begin{array}{ll}
	\left\{ \beta \in S_M \mid \big< \alpha , \beta \big> \geq 0 \wedge \big< \I_{+} \left( \alpha \right), \beta \big> < 0 \right\} & z \geq 0, \ M > 0 \\  [1ex]
	\left\{ \beta \in S_M \mid \big< \alpha , \beta \big> \leq 0 \wedge \big< \I_{+} \left( \alpha \right), \beta \big> > 0 \right\} & z \geq 0, \ M < 0 \\ [1ex]
	\left\{ \beta \in S_M \mid \big< \alpha , \beta \big> \leq 0 \wedge \big< \I_{-} \left( \alpha \right), \beta \big> > 0 \right\} & z < 0, \ M > 0 \\ [1ex]
	\left\{ \beta \in S_M \mid \big< \alpha , \beta \big> \geq 0 \wedge \big< \I_{-} \left( \alpha \right), \beta \big> < 0 \right\} & z < 0, \ M < 0
\end{array}	
\right.  $$
the {\em subbranch} with respect to $ \alpha $. 
\end{definition}

The definition of the subbranch seems to be involved. However, if we consider the complex plane it is much more simple to interpret. Consider the case if $ z \geq 0 $ and $ M > 0 $. By \Cref{proposition_prop_iz} we know that $ \big< \alpha, \I_{+} \left( \alpha \right) \big> = \N \left( \alpha \right) $ and that $ \alpha, \I_{+}\left( \alpha \right) $ are both on the same branch, i.e. there are points on the branch between $ \alpha $ and $ \I_{+}\left( \alpha \right) $. Now we explain why all these elements on the same branch “between” $ \alpha $ and $ \I_{+}\left( \alpha \right) $ including $ \alpha $ and excluding $ \I_{+}\left( \alpha \right) $ are contained in $ B_{\alpha} $. Observe that these elements $ \gamma \in S_M $ satisfy the definition $ \big< \alpha , \gamma \big> \geq 0 \wedge \big< \I_{+} \left( \alpha \right), \gamma \big> < 0 $ even if $ \gamma = \alpha $, but not if $ \gamma = \I_{+} \left( \alpha \right) $. Hence, we only need to show why all the other elements in $ S_M $ do not satisfy the definition. Remark that all the elements “between” $ -\alpha $ and $ -\I_{+}\left( \alpha \right) $ do not satisfy them because the sign is not correct, i.e. for an element $ \gamma \in S_M $ “between” $ -\alpha $ and $ -\I_{+}\left( \alpha \right) $ the sign of the oriented area is swapped. Moreover, for all the other elements in $ S_M $ which are neither between $ \alpha,\I_{+}\left( \alpha \right) $ not $ -\alpha,-\I_{+}\left( \alpha \right) $ we have that the sign of both oriented areas are the same and so they cannot belong to the set $ B_{\alpha} $. 

In the case $ z \geq 0 $ and $ M < 0 $ we have that the orientation changes (compare with \Cref{Separating_lines_positive}), so the signs of the oriented areas have to switch. If $ z < 0 $ and $ M > 0 $, then the orientation compared to the case $ z \geq 0 $ also changes because the isomorphism $ \Phi $ is like a mirror on the real axis and the function $ \I_{+} $ will be replaced by $ \I_{-} $ as $ \alpha $ and $ \I_{+} \left(\alpha\right) $ are not on the same branch if $ z < -1 $. From $ z < 0 $ and $ M > 0 $ to $ z < 0 $ and $ M < 0 $ the orientation changes and so the signs of the oriented areas change again.

\begin{example}
Let $ \alpha = \sqrt{6}-2i \in \mathbb{R}[i] $, then $ M = \N \left( \alpha \right) = 10 $ and $ B_{\alpha} $ consists of the elements between $ \alpha $ and $ \I_{+} \left( \alpha \right) = i\left(\sqrt{6}-2i\right) = 2-\sqrt{6}i$ including $ \alpha $ and excluding $ \I_{+} \left( \alpha \right) $, see \Cref{units_z3}.
\end{example}

\vspace{5mm}
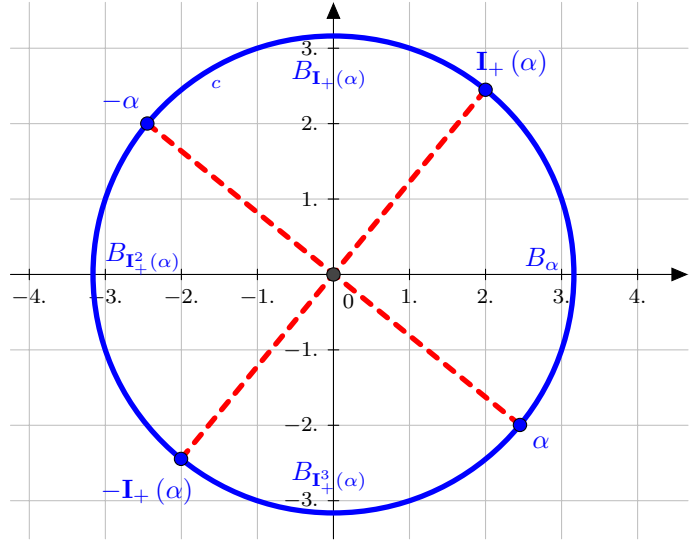
\begin{figure}[h]  
\begin{center}
\pagestyle{empty}

\definecolor{ffqqqq}{rgb}{1.,0.,0.}
\definecolor{uuuuuu}{rgb}{0.26666666666666666,0.26666666666666666,0.26666666666666666}
\definecolor{qqqqff}{rgb}{0.,0.,1.}
\definecolor{cqcqcq}{rgb}{0.7529411764705882,0.7529411764705882,0.7529411764705882}
\begin{tikzpicture}[line cap=round,line join=round,>=triangle 45,x=1.0cm,y=1.0cm]
\draw [color=cqcqcq,, xstep=1.0cm,ystep=1.0cm] (-4.247776103378497,-3.5056020036309796) grid (4.663648613392571,3.6081732444121273);
\draw[->,color=black] (-4.247776103378497,0.) -- (4.663648613392571,0.);
\foreach \x in {-4.,-3.,-2.,-1.,1.,2.,3.,4.}
\draw[shift={(\x,0)},color=black] (0pt,2pt) -- (0pt,-2pt) node[below] {\footnotesize $\x$};
\draw[->,color=black] (0.,-3.5056020036309796) -- (0.,3.6081732444121273);
\foreach \y in {-3.,-2.,-1.,1.,2.,3.}
\draw[shift={(0,\y)},color=black] (2pt,0pt) -- (-2pt,0pt) node[left] {\footnotesize $\y$};
\draw[color=black] (0pt,-10pt) node[right] {\footnotesize $0$};
\clip(-4.247776103378497,-3.5056020036309796) rectangle (4.663648613392571,3.6081732444121273);
\draw [line width=2.pt,color=qqqqff] (0.,0.) circle (3.1622776601683795cm);
\draw [line width=2.pt,dash pattern=on 4pt off 4pt,color=ffqqqq] (2.000242814908519,2.4492914651806643)-- (-2.0047090816328916,-2.445637237616938);
\draw [line width=2.pt,dash pattern=on 4pt off 4pt,color=ffqqqq] (2.4530217761611834,-1.9956663462811195)-- (-2.4481264048285176,2.0016685804551657);
\begin{scriptsize}
\draw[color=qqqqff] (-1.5436196375996891,2.5249742055632307) node {$c$};
\draw [fill=qqqqff] (2.4530217761611834,-1.9956663462811195) circle (2.5pt);
\draw[color=qqqqff] (2.68,-2.23) node {\fontsize{10}{0} $ \alpha $};
\draw [fill=qqqqff] (2.000242814908519,2.4492914651806643) circle (2.5pt);
\draw[color=qqqqff] (2.3,2.8) node {\fontsize{10}{0} $ \I_{+} \left( \alpha \right) $};
\draw [fill=uuuuuu] (0.,0.) circle (2.5pt);
\draw [fill=qqqqff] (-2.0047090816328916,-2.445637237616938) circle (2.5pt);
\draw[color=qqqqff] (-2.5,-2.87) node {\fontsize{10}{0} $-\I_{+} \left( \alpha \right)$};
\draw [fill=qqqqff] (-2.4481264048285176,2.0016685804551657) circle (2.5pt);
\draw[color=qqqqff] (-2.85,2.3) node {\fontsize{10}{0} $-\alpha $};
\draw[color=qqqqff]
(2.7,0.2) node {\fontsize{10}{0} $B_{\alpha}$};
\draw[color=qqqqff]
(-0.1,2.65) node {\fontsize{10}{0} $B_{\I_{+} \left( \alpha \right)}$};
\draw[color=qqqqff]
(-2.55,0.2) node {\fontsize{10}{0} $B_{\I_{+}^2 \left( \alpha \right)}$};
\draw[color=qqqqff]
(-0.1,-2.7) node {\fontsize{10}{0} $B_{\I_{+}^{3} \left( \alpha \right)}$};

\end{scriptsize}
\end{tikzpicture}
\caption{Partition of subbranches in $ \mathbb{R}[i] $}
\label{units_z0}
\end{center}
\end{figure}
\vspace{5mm}

Now we are ready to define the so called closed branch which is the part of a branch “between” two elements on the same branch.

\begin{definition}
	Let $ z \in \mathbb{Z} \setminus \{ -1,0,1 \} $, $ M \in \mathbb{N} \setminus \{ 0 \} $, $ B \subseteq S_M \subseteq \mathbb{R}[i_z] $ be a branch and $ \alpha_1,\alpha_2 \in B $. Then we call
	$$ B_{\alpha_1, \alpha_2} \coloneqq \{ \beta \in B \mid \big< \alpha_1 , \beta \big> \big< \alpha_2, \beta \big> \leq 0 \} $$
the {\em closed branch bounded by $ \alpha_1 $ and $ \alpha_2 $}.	
\end{definition}

Observe, that $ B_{\alpha_1, \alpha_2} $ really contains all the elements of a branch “between” $ \alpha_1 $ and $ \alpha_2 $ inclusively $ \alpha_1,\alpha_2 $ (in case $ \beta \in \{ \alpha_1,\alpha_2 \} $, then the product $ \big< \alpha_1 , \beta \big> \big< \alpha_2, \beta \big> $ is zero by \Cref{lem_orient_prod_zero}). Since only the elements $ \beta \in S_M $ between $ \alpha_1,\alpha_2 $ and $ -\alpha_1,-\alpha_2 $ satisfy the condition that the signs of $ \big< \alpha_1 , \beta \big> $ and $ \big< \alpha_2 , \beta \big> $ are different from each other (or one is zero and the other positive or negative) and we require $ \beta $ to be on the same branch and $ -\alpha_1,-\alpha_2 \notin B $ (because $ \vert z \vert > 1 $) we are sure that $ B_{\alpha_1, \alpha_2} $ contains exactly the elements on $ B $ “between” $ \alpha_1 $ and $ \alpha_2 $ if we consider the complex plane.

The following fact states an inequality which is very useful if we work with closed branches. But before we come to it observe that the branches (or branch if $ \vert z \vert \leq 1 $) of $ \mathbb{R}[i_z] $ separate the complex plane. In case $ \vert z \vert \leq 1 $ we have that the part containing the origin is strictly convex since all lines between two points in this part are entirely in the same part where as this does not hold true if $ \vert z \vert \geq 3 $. Therefore we will call the branch in $ \mathbb{R}[i_z] $ convex if $ z \in \left\{ -1,0,1 \right\} $ and we call the branches concave otherwise. 
In case the branches are concave we get a useful inequality for the oriented areas of three elements on the same branch:


\begin{fact} \label{fact_ineq_concave}
Let $ z \in \mathbb{Z} \setminus \{-1,0,1\} $, $ M \in \mathbb{Z} \setminus \{ 0 \} $ and $ B \subseteq S_M $ be a branch with $ \alpha_1, \alpha_2, \delta \in B $ where $ \delta \in B_{\alpha_1,\alpha_2} $. Then all the curves $ x^2+zxy+y^2 = M $ consist of two different branches. 
Moreover, the following inequality holds true:
	$$ \vert \big< \alpha_1,\alpha_2 \big> \vert \geq \vert \big< \alpha_1,\delta \big> \vert + \vert \big< \delta,\alpha_2 \big> \vert.$$
\end{fact}

\begin{example}
	The inequality of \Cref{fact_ineq_concave} says that the absolute value of the red area in \Cref{ex_fact_ineq} is greater or equal to the sum of the corresponding green and blue area, respectively. This holds clearly true if the considered branches are concave, i.e. if the branches are defined by the Diophantine inequality $ x^2 + zxy + y^2 = M $ for $ z \in \mathbb{Z} \setminus \{-1,0,1\} $ and $ M \in \mathbb{Z} \setminus \{ 0 \} $. In case $ z = 2 $ the branches are lines, so the inequality of \Cref{fact_ineq_concave} will then be an equality.
\end{example}

\vspace{5mm}
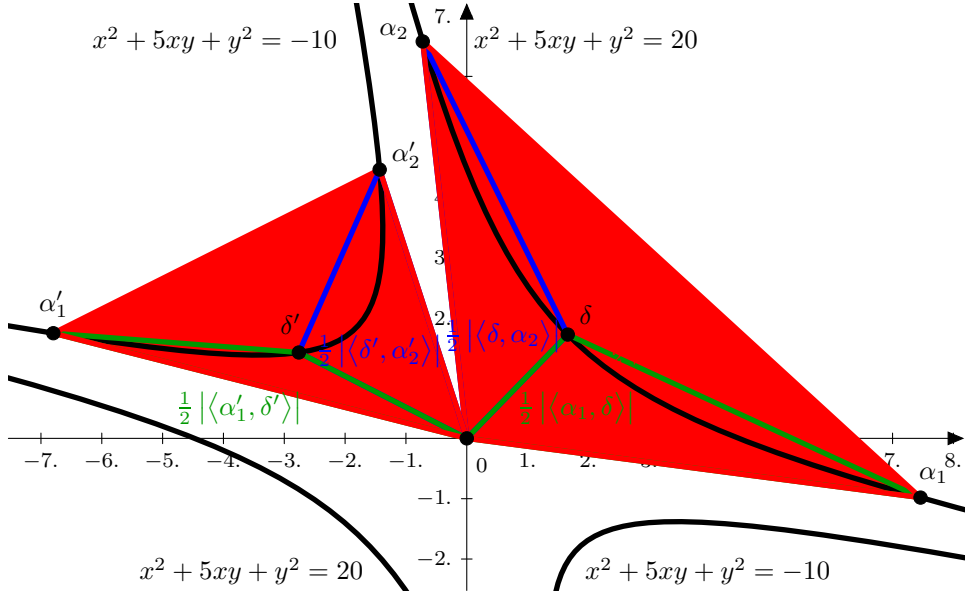
\begin{figure}[h]  
\begin{center}
\pagestyle{empty}

\definecolor{ffqqqq}{rgb}{1.,0.,0.}
\definecolor{qqzzqq}{rgb}{0.,0.6,0.}
\definecolor{qqqqff}{rgb}{0.,0.,1.}
\definecolor{ffxfqq}{rgb}{1.,0.4980392156862745,0.}
\begin{tikzpicture}[line cap=round,line join=round,>=triangle 45,x=0.8cm,y=0.8cm]
\draw[->,color=black] (-7.535371744336492,0.) -- (8.17863013548406,0.);
\foreach \x in {-7.,-6.,-5.,-4.,-3.,-2.,-1.,1.,2.,3.,4.,5.,6.,7.,8.}
\draw[shift={(\x,0)},color=black] (0pt,2pt) -- (0pt,-2pt) node[below] {\footnotesize $\x$};
\draw[->,color=black] (0.,-2.5208135511574863) -- (0.,7.214834256599425);
\foreach \y in {-2.,-1.,1.,2.,3.,4.,5.,6.,7.}
\draw[shift={(0,\y)},color=black] (2pt,0pt) -- (-2pt,0pt) node[left] {\footnotesize $\y$};
\draw[color=black] (0pt,-10pt) node[right] {\footnotesize $0$};
\clip(-7.535371744336492,-2.5208135511574863) rectangle (8.17863013548406,7.214834256599425);
\fill[line width=2.pt,color=qqqqff,fill=qqqqff,fill opacity=0.10000000149011612] (0.,0.) -- (-0.7236913742465916,6.57888023916402) -- (1.6602741052528058,1.7205733795278597) -- cycle;
\fill[line width=2.pt,color=qqzzqq,fill=qqzzqq,fill opacity=0.10000000149011612] (0.,0.) -- (7.459569528102715,-0.9815193105631677) -- (1.6602741052528058,1.7205733795278597) -- cycle;
\fill[line width=2.pt,color=ffqqqq,fill=ffqqqq,fill opacity=0.10000000149011612] (0.,0.) -- (-0.7236913742465916,6.57888023916402) -- (7.459569528102715,-0.9815193105631677) -- cycle;
\fill[line width=2.pt,color=qqqqff,fill=qqqqff,fill opacity=0.10000000149011612] (0.,0.) -- (-1.4319868607026376,4.4549405447413415) -- (-2.7603611732992532,1.4234155998940285) -- cycle;
\fill[line width=2.pt,color=qqzzqq,fill=qqzzqq,fill opacity=0.10000000149011612] (0.,0.) -- (-2.7603611732992532,1.4234155998940285) -- (-6.7968897518727225,1.7430283657819157) -- cycle;
\fill[line width=2.pt,color=ffqqqq,fill=ffqqqq,fill opacity=0.10000000149011612] (0.,0.) -- (-6.7968897518727225,1.7430283657819157) -- (-1.4319868607026376,4.4549405447413415) -- cycle;
\draw [samples=50,domain=-0.99:0.99,rotate around={-135.:(0.,0.)},xshift=0.cm,yshift=0.cm,line width=2.pt] plot ({2.390457218668787*(1+(\x)^2)/(1-(\x)^2)},{3.6514837167011076*2*(\x)/(1-(\x)^2)});
\draw [samples=50,domain=-0.99:0.99,rotate around={-135.:(0.,0.)},xshift=0.cm,yshift=0.cm,line width=2.pt] plot ({2.390457218668787*(-1-(\x)^2)/(1-(\x)^2)},{3.6514837167011076*(-2)*(\x)/(1-(\x)^2)});
\draw [line width=2.pt,color=qqqqff] (0.,0.)-- (-0.7236913742465916,6.57888023916402);
\draw [line width=2.pt,color=qqqqff] (-0.7236913742465916,6.57888023916402)-- (1.6602741052528058,1.7205733795278597);
\draw [line width=2.pt,color=qqqqff] (1.6602741052528058,1.7205733795278597)-- (0.,0.);
\draw [line width=2.pt,color=qqzzqq] (0.,0.)-- (7.459569528102715,-0.9815193105631677);
\draw [line width=2.pt,color=qqzzqq] (7.459569528102715,-0.9815193105631677)-- (1.6602741052528058,1.7205733795278597);
\draw [line width=2.pt,color=qqzzqq] (1.6602741052528058,1.7205733795278597)-- (0.,0.);
\draw [line width=2.pt,color=ffqqqq] (0.,0.)-- (-0.7236913742465916,6.57888023916402);
\draw [line width=2.pt,color=ffqqqq] (-0.7236913742465916,6.57888023916402)-- (7.459569528102715,-0.9815193105631677);
\draw [line width=2.pt,color=ffqqqq] (7.459569528102715,-0.9815193105631677)-- (0.,0.);
\draw [samples=50,domain=-0.99:0.99,rotate around={-45.:(0.,0.)},xshift=0.cm,yshift=0.cm,line width=2.pt] plot ({2.581988897471611*(1+(\x)^2)/(1-(\x)^2)},{1.6903085094570331*2*(\x)/(1-(\x)^2)});
\draw [samples=50,domain=-0.99:0.99,rotate around={-45.:(0.,0.)},xshift=0.cm,yshift=0.cm,line width=2.pt] plot ({2.581988897471611*(-1-(\x)^2)/(1-(\x)^2)},{1.6903085094570331*(-2)*(\x)/(1-(\x)^2)});
\draw [line width=2.pt,color=qqqqff] (0.,0.)-- (-1.4319868607026376,4.4549405447413415);
\draw [line width=2.pt,color=qqqqff] (-1.4319868607026376,4.4549405447413415)-- (-2.7603611732992532,1.4234155998940285);
\draw [line width=2.pt,color=qqqqff] (-2.7603611732992532,1.4234155998940285)-- (0.,0.);
\draw [line width=2.pt,color=qqzzqq] (0.,0.)-- (-2.7603611732992532,1.4234155998940285);
\draw [line width=2.pt,color=qqzzqq] (-2.7603611732992532,1.4234155998940285)-- (-6.7968897518727225,1.7430283657819157);
\draw [line width=2.pt,color=qqzzqq] (-6.7968897518727225,1.7430283657819157)-- (0.,0.);
\draw [line width=2.pt,color=ffqqqq] (0.,0.)-- (-6.7968897518727225,1.7430283657819157);
\draw [line width=2.pt,color=ffqqqq] (-6.7968897518727225,1.7430283657819157)-- (-1.4319868607026376,4.4549405447413415);
\draw [line width=2.pt,color=ffqqqq] (-1.4319868607026376,4.4549405447413415)-- (0.,0.);
\begin{scriptsize}
\draw [fill=black] (0.,0.) circle (2.5pt);
\draw[color=black] (1.9,6.6) node {\fontsize{10}{0} $x^2+5xy+y^2 = 20$};
\draw[color=black] (-4.2,6.6) node {\fontsize{10}{0} $x^2+5xy+y^2 = -10$};
\draw[color=black] (-3.6,-2.2) node {\fontsize{10}{0} $x^2+5xy+y^2 = 20$};
\draw[color=black] (3.9,-2.2) node {\fontsize{10}{0} $x^2+5xy+y^2 = -10$};
\draw [fill=black]
(-0.7236913742465916,6.57888023916402) circle (2.5pt);
\draw[color=black] (-1.25,6.75) node {\fontsize{10}{0} $\alpha_2 $};
\draw [fill=black] (1.6602741052528058,1.7205733795278597) circle (2.5pt);
\draw[color=black] (1.9,2.05) node {\fontsize{10}{0} $\delta$};
\draw[color=qqqqff] (0.53,1.7) node {\fontsize{10}{0} $ \tfrac{1}{2}\left\vert \big< \delta, \alpha_2 \big> \right\vert $};
\draw [fill=black] (7.459569528102715,-0.9815193105631677) circle (2.5pt);
\draw[color=black] (7.62,-0.62) node {\fontsize{10}{0} $\alpha_1 $};
\draw[color=qqzzqq] (1.74,0.5) node {\fontsize{10}{0} $ \tfrac{1}{2}\left\vert \big< \alpha_1, \delta \big> \right\vert $};
\draw[color=ffqqqq] (3.38,1.5) node {\fontsize{10}{0} $ \tfrac{1}{2}\left\vert \big< \alpha_1, \alpha_2 \big> \right\vert $};
\draw [fill=black] (-1.4319868607026376,4.4549405447413415) circle (2.5pt);
\draw[color=black] (-1.05,4.7) node {\fontsize{10}{0} $\alpha_2' $};
\draw [fill=black] (-2.7603611732992532,1.4234155998940285) circle (2.5pt);
\draw[color=black] (-2.95,1.9) node {\fontsize{10}{0} $\delta' $};
\draw[color=qqqqff] (-1.5,1.45) node {\fontsize{10}{0} $ \tfrac{1}{2}\left\vert \big< \delta',\alpha_2' \big> \right\vert $};
\draw [fill=black] (-6.7968897518727225,1.7430283657819157) circle (2.5pt);
\draw[color=black] (-6.85,2.23) node {\fontsize{10}{0} $ \alpha_1' $};
\draw[color=qqzzqq] (-3.8,0.5) node {\fontsize{10}{0} $ \tfrac{1}{2}\left\vert \big< \alpha_1', \delta' \big> \right\vert $};
\draw[color=ffqqqq] (-3.5,2.65) node {\fontsize{10}{0} $ \tfrac{1}{2}\left\vert \big< \alpha_1', \alpha_2' \big> \right\vert$};

\end{scriptsize}
\end{tikzpicture}
\caption{The inequality of \Cref{fact_ineq_concave}}
\label{ex_fact_ineq}
\end{center}
\end{figure}
\vspace{5mm}

Now we will show a lot of statements which we finally use to prove the Local Solution Theorems 1 and 2. A big milestone for the proof of the first local solution theorem will be \Cref{prop_part_of_branches} and \Cref{coro_neg_partition} where we will show that for each branch and level set we find a useful partition consisting of subbranches.

\begin{lemma} \label{lem_branch_rules}
	Let $ z \in \mathbb{Z} $, $ M \in \mathbb{Z} \setminus \left\{ 0 \right\} $ be arbitrary, $ \alpha_0,\alpha_1, \alpha_2 \in S_M \subseteq \mathbb{Z}[i_z] $, $ B_{\alpha_{0}} \subseteq S_M $ be a subbranch, $ B_{\alpha_1,\alpha_2} \subseteq S_M $ be a closed branch contained on the branch $ B $ and let $ \Phi : \mathbb{R}[i_z] \to \mathbb{R}[i_{-z}] $ the ring isomorphism as already defined before. Then the following holds true:
	\begin{enumerate}
		\item[i)] $ \I_{+}^{n}\left( B_{\alpha_0} \right) = B_{\I_{+}^{n} \left(\alpha_0\right)} $  and $ \I_{-}^{n}\left( B_{\alpha_0} \right) = B_{\I_{-}^{n} \left(\alpha_0\right)} $
		\item[ii)] $ \Phi \left( B_{\alpha_0} \right) = B_{\Phi \left( \alpha_0 \right)} $ if $ z \neq 0 $
		\item[iii)] $ \Phi \left(B_{\alpha_1, \alpha_2} \right) = B_{\Phi \left( \alpha_1 \right),\Phi \left( \alpha_2 \right)} $ if $ z \notin \left\{ -1,0,1 \right\} $
	\end{enumerate}
\end{lemma}

\begin{proof}
Let $ M \gtrless 0 $. Use that the functions $ \I_{+} $ and $ \I_{-} $ commute and the statements in the proof of \Cref{proposition_prop_iz} and \Cref{coro_In} (in particular, use the commutative diagrams which tell us that for $ \alpha_1,\alpha_2 \in \mathbb{Z}[i_z] $ we have $ \big< \alpha_1, \alpha_2 \big> = -\big< \Phi \left(\alpha_1\right), \Phi \left(\alpha_2\right) \big> $, $ \Phi \circ \I_{+} = \I_{-} \circ \Phi $ and also $ \Phi \circ \I_{-} = \I_{+} \circ \Phi $).
	\begin{enumerate}	
		\item[i)] At first assume $ z \geq 0 $, then we have:
			\begin{align*}
				\gamma \in \I_{+}^{n}\left( B_{\alpha_0} \right) 			&\Longleftrightarrow \I_{+}^{-n} \left( \gamma \right) \in B_{\alpha_0} \\
				&\Longleftrightarrow	\big< \alpha_0, \I_{+}^{-n} \left( \gamma \right) \big> \gtreqless 0 \wedge  \big< \I_{+} \left(\alpha_0\right), \I_{+}^{-n} \left( \gamma \right) \big> \lessgtr 0 \\
				&\Longleftrightarrow \big< \I_{+}^n \left(\alpha_0 \right), \gamma \big> \gtreqless 0 \wedge \big< \I_{+}^{n+1} \left(\alpha_0\right),\gamma \big> \lessgtr 0 \\
				&\Longleftrightarrow \gamma \in B_{\I_{+}^{n} \left( \alpha_{0} \right)}.
			\end{align*}
On the other hand, if	$ z < 0 $, then we get:	
	\begin{align*}
				\gamma \in \I_{+}^{n}\left( B_{\alpha_0} \right) 			&\Longleftrightarrow \I_{+}^{-n} \left( \gamma \right) \in B_{\alpha_0} \\
				&\Longleftrightarrow	\big< \alpha_0, \I_{+}^{-n} \left( \gamma \right) \big> \lesseqgtr 0 \wedge  \big< \I_{-} \left(\alpha_0\right), \I_{+}^{-n} \left( \gamma \right) \big> \gtrless 0 \\
				&\Longleftrightarrow \big< \I_{+}^n \left(\alpha_0 \right), \gamma \big> \lesseqgtr 0 \wedge \big< \I_{-} \left(\I_{+}^{n} \left(\alpha_0\right) \right),\gamma \big> \gtrless 0 \\
				&\Longleftrightarrow \gamma \in B_{\I_{+}^{n} \left( \alpha_{0} \right)}				
	\end{align*}	
This implies $ \I_{+}^{n}\left( B_{\alpha_0} \right) = B_{\I_{+}^{n} \left(\alpha_0\right)} $. Analogously, we can show for $ z \geq 0 $
	\begin{align*}
		\gamma \in \I_{-}^{n}\left( B_{\alpha_0} \right) 			&\Longleftrightarrow \I_{-}^{-n} \left( \gamma \right) \in B_{\alpha_0} \\
				&\Longleftrightarrow	\big< \alpha_0, \I_{-}^{-n} \left( \gamma \right) \big> \gtreqless 0 \wedge  \big< \I_{+} \left(\alpha_0\right), \I_{-}^{-n} \left( \gamma \right) \big> \lessgtr 0 \\
				&\Longleftrightarrow \big< \I_{+}^n \left(\alpha_0 \right), \gamma \big> \gtreqless 0 \wedge \big< \I_{+} \left(\I_{-}^{n} \left(\alpha_0\right) \right),\gamma \big> \lessgtr 0 \\
				&\Longleftrightarrow \gamma \in B_{\I_{-}^{n} \left( \alpha_{0} \right)}				
	\end{align*}	
and for $ z < 0 $
	\begin{align*}
				\gamma \in \I_{-}^{n}\left( B_{\alpha_0} \right) 			&\Longleftrightarrow \I_{-}^{-n} \left( \gamma \right) \in B_{\alpha_0} \\
				&\Longleftrightarrow	\big< \alpha_0, \I_{-}^{-n} \left( \gamma \right) \big> \lesseqgtr 0 \wedge  \big< \I_{-} \left(\alpha_0\right), \I_{-}^{-n} \left( \gamma \right) \big> \gtrless 0 \\
				&\Longleftrightarrow \big< \I_{-}^n \left(\alpha_0 \right), \gamma \big> \lesseqgtr 0 \wedge \big< \I_{-}^{n+1} \left(\alpha_0\right),\gamma \big> \gtrless 0 \\
				&\Longleftrightarrow \gamma \in B_{\I_{-}^{n} \left( \alpha_{0} \right)}			
	\end{align*}	
which proves $ \I_{-}^{n}\left( B_{\alpha_0} \right) = B_{\I_{-}^{n} \left(\alpha_0\right)} $.
		\item[ii)] Assume $ z > 0 $, then we have
			\begin{align*}
				\gamma \in \Phi \left( B_{\alpha_0} \right) &\Longleftrightarrow \Phi^{-1} \left( \gamma \right) \in B_{\alpha_0} \\
		&\Longleftrightarrow \underbrace{\big< \alpha_0, \Phi^{-1} \left( \gamma \right) \big>}_{=-\big< \Phi \left( \alpha_0 \right), \Phi \left( \Phi^{-1} \left( \gamma \right) \right) \big>}	\gtreqless 0 \wedge 	\underbrace{\big< \I_{+} \left( \alpha_0 \right), \Phi^{-1} \left( \gamma \right) \big>}_{= -\big< \Phi \left( \I_{+} \left( \alpha_0 \right) \right), \Phi \left( \Phi^{-1} \left( \gamma \right) \right)\big>} \lessgtr 0 \\
		&\Longleftrightarrow \big< \Phi \left( \alpha_0 \right), \gamma \big> \lesseqgtr 0 \wedge \big< \I_{-} \left( \Phi \left( \alpha_0 \right) \right), \gamma \big> \gtrless 0 \\
		&\Longleftrightarrow \gamma \in B_{\Phi \left( \alpha_0 \right)}
			\end{align*}				
and also for	 $ z < 0 $ it follows
	\begin{align*}
				\gamma \in \Phi \left( B_{\alpha_0} \right) &\Longleftrightarrow \Phi^{-1} \left( \gamma \right) \in B_{\alpha_0} \\
		&\Longleftrightarrow \underbrace{\big< \alpha_0, \Phi^{-1} \left( \gamma \right) \big>}_{=-\big< \Phi \left( \alpha_0 \right), \Phi \left( \Phi^{-1} \left( \gamma \right) \right) \big>} \lesseqgtr 0 \wedge 	\underbrace{\big< \I_{+} \left( \alpha_0 \right), \Phi^{-1} \left( \gamma \right) \big>}_{= -\big< \Phi \left( \I_{+} \left( \alpha_0 \right) \right), \Phi \left( \Phi^{-1} \left( \gamma \right) \right)\big>} \gtrless 0 \\
		&\Longleftrightarrow \big< \Phi \left( \alpha_0 \right), \gamma \big> \gtreqless 0 \wedge \big< \I_{-} \left( \Phi \left( \alpha_0 \right) \right), \gamma \big> \lessgtr 0 \\
		&\Longleftrightarrow \gamma \in B_{\Phi \left( \alpha_0 \right)}.
			\end{align*}			
Hence, we conclude $ \Phi \left( B_{\alpha_0} \right) = B_{\Phi \left( \alpha_0 \right)} $ if $ z \neq 0 $.		
		\item[iii)] Observe that $ \Phi \left( \gamma \right), \Phi \left( \alpha_1 \right), \Phi \left( \alpha_2 \right) \in \Phi \left( B \right) $ where $ \Phi \left( B \right) \subseteq \mathbb{R}[i_{-z}] $ is a branch, too. Therefore we get
	\begin{align*}
		\gamma \in \Phi \left( B_{\alpha_1,\alpha_2} \right)	&\Longleftrightarrow	\Phi^{-1} \left( \gamma \right) \in B_{\alpha_1, \alpha_2} \\
		&\Longleftrightarrow \underbrace{\big< \alpha_1,\Phi^{-1} \left( \gamma \right) \big>}_{=-\big< \Phi \left( \alpha_1 \right), \gamma \big>} \underbrace{\big< \alpha_2,\Phi^{-1} \left( \gamma \right) \big>}_{=-\big< \Phi \left( \alpha_2 \right), \gamma \big>} \leq 0 \\
		&\Longleftrightarrow \big< \Phi \left( \alpha_1 \right), \gamma \big>\big< \Phi \left( \alpha_2 \right), \gamma \big> \leq 0 \\
		&\Longleftrightarrow \gamma \in B_{\Phi \left( \alpha_1 \right),\Phi \left( \alpha_2 \right)}
	\end{align*}	 
which implies $ \Phi \left(B_{\alpha_1, \alpha_2} \right) = B_{\Phi \left( \alpha_1 \right),\Phi \left( \alpha_2 \right)} $. 	
	\end{enumerate}
\end{proof}

\begin{lemma} \label{lem_sub_closed_branch}
Let $ z \in \mathbb{Z} \setminus \left\{ -1,0,1 \right\} $, $ M \in \mathbb{Z} \setminus \left\{ 0 \right\} $, $ B $ be a branch with $ \alpha_0, \alpha_1, \alpha_2, \gamma \in B $. Then the following holds true:
	\begin{enumerate}
		\item[i)] $ B_{\alpha_0} \subseteq B_{\alpha_0,\I_{+}\left( \alpha_0 \right)} \subseteq B $ if $ z \geq 0 $ and $ B_{\alpha_0} \subseteq B_{\alpha_0,\I_{-}\left( \alpha_0 \right)} \subseteq B $ if $ z < 0 $
		\item[ii)] If $ \gamma \in B_{\alpha_1,\alpha_2} $, then $ B_{\alpha_1, \gamma} \subseteq B_{\alpha_1,\alpha_2} $ and $ B_{\alpha_2, \gamma} \subseteq B_{\alpha_1,\alpha_2} $
	\end{enumerate}
\end{lemma}

\begin{proof}
	\begin{enumerate}
		
\item[i)] Let $ M \gtrless 0 $ and assume at first that $ z > 1 $.  Let $ \gamma \coloneqq c_1 + c_2 i_z \in B_{\alpha_0} $ and $ \alpha_0 = a_1 + a_2i_z \in \mathbb{R}[i_z] $, then we have
	$$ \big< \alpha_0, \gamma \big> = a_1c_2- a_2c_1 \gtreqless 0 $$
and 
	$$ \big< \I_{+} \left( \alpha_0 \right), \gamma \big> = -a_2c_2 - \left( a_1+za_2 \right)c_2i_z \lessgtr 0 .$$ The line through $ \alpha_0  $ and the origin is line $ l_{-a_2,a_1} $. We can interpret $ \alpha_0 $ and $ \I_{+} \left( \alpha_0 \right) $ as vectors in the complex plane. Depending whether $ M > 0 $ or $ M < 0 $ the elements on the left or right side including the line itself, respectively, where left or right refers  to the direction of the vector $ \alpha_0 $ on $ l_{-a_2,a_1} $,  satisfy the condition $ \big< \alpha_0, \gamma \big> \gtreqless 0 $. Whereas the line $ l_{a_1+za_2,a_2} $ is defined by the vector $ \I_{+} \left( \alpha_0 \right) $ and the origin and all elements on the complex plane on the right or left side of the line, respectively, which are not included on the line and do satisfy the condition $ \big< \I_{+} \left( \alpha_0 \right), \gamma \big> \lessgtr 0 $. Hence, the elements which satisfy both conditions must lie in a cone defined by the origin and the two lines (with $ l_{-a_2,a_1} $ and without $ l_{a_1+za_2,a_2} $). This means that $ \gamma $ must lie in this cone and in $ S_M $.

Now we know that the line $ l_{1,1} $ separates the branches of $ S_M $ if $ M > 0 $ and otherwise, i.e. if $ M < 0 $, then the branches lie entirely in the second or fourth quadrant of the complex plane. By \Cref{lem_at_most_two_solutions} we observe that each element of $ S_M $ can be described uniquely by its angle in polar coordinates. Since $ l_{1,1} $ separates the branches, we could describe all the elements on $ S_M $ above or below uniquely by an angle in $ \theta \in (-\tfrac{1}{4}\pi, \tfrac{3}{4}\pi) $ or $ (\tfrac{3}{4}\pi, \tfrac{7}{4}\pi) $, respectively. Moreover, if $ M < 0 $, then $ \theta \in \left( \tfrac{1}{2}\pi, \pi \right) $ or $ \theta \in \left( -\pi,0 \right) $. Now the elements in the cone have an angle in polar coordinates which lies between the angles of $ \alpha_0 $ (including the angle of it) to  $ \I_{+} \left( \alpha_0 \right) $ (not included this angle). So the elements in $ B_{\alpha_0} $ are the ones in the cone and in $ S_M $. However, all these elements have either an angle which is also above or below $ l_{1,1} $ for $ M > 0 $ and for $ M < 0 $, the cone lies entirely in the second or fourth quadrant. Hence $ \gamma \in B $ and $ \big< \alpha_0, \gamma \big> \big< \I_{+} \left( \alpha_0 \right), \gamma \big> \leq 0 $, so $ \gamma \in B_{\alpha_0,\I_{+}\left( \alpha_0 \right)} $.

In case $ z < -1 $ we have
	$$ \Phi \left( \alpha_0 \right) \in \Phi \left( B_{\alpha_0} \right) = B_{\Phi \left( \alpha_0 \right)} \subseteq B_{\Phi \left( \alpha_0 \right), \I_{+}\left( \Phi \left(\alpha_0 \right) \right)} = \Phi \left( B_{\alpha_0, \I_{-}\left( \alpha_0 \right)} \right) $$
by \Cref{lem_branch_rules} and the previous part. Since $ \Phi $ is an isomorphism we deduce that $ \alpha_0 \in B_{\alpha_0, \I_{-}\left( \alpha_0 \right)} $.

\item[ii)] As $ \alpha_1, \alpha_2 $ have symmetric roles it remains to show $ B_{\alpha_1, \gamma} \subset B_{\alpha_1, \alpha_2} $. Let $ \delta \in B_{\alpha_1, \gamma} $, then
	$$ \big< \alpha_1, \delta \big> \big< \gamma, \delta\big> \leq 0 .$$
We assume that $ \delta \notin B_{\alpha_1, \alpha_2} $ and lead it to contradiction. Since $ \alpha_1,\alpha_2, \delta $ are all on the same branch we have
	$$ \big< \alpha_1, \delta \big>\big< \alpha_2, \delta \big> > 0 $$
and so $ \big< \alpha_1, \delta \big> $	 and $ \big< \alpha_2, \delta \big> $ have the same sign. Furthermore, we have that 
	$$ \big< \alpha_2, \delta \big> \big< \gamma, \delta \big> \leq 0 $$
and therefore $ \delta \in B_{\alpha_2, \gamma} $.	

We now have two cases: Either $ \big< \alpha_1,\alpha_2 \big> $ or $ \big< \alpha_2,\alpha_1 \big> $ is zero or has another sign than $ \big< \alpha_1, \delta \big> $ and $ \big< \alpha_2, \delta \big> $. Assume that $ \big< \alpha_1,\alpha_2 \big> \big< \alpha_1,\delta \big> \leq 0 $, then we have that $ \alpha_1 \in B_{\alpha_2, \delta} $. Since also $ \gamma \in B_{\alpha_1,\alpha_2} $ we get by applying \Cref{fact_ineq_concave} three times:
	\begin{align*}
		\left\vert \big< \alpha_2,\delta \big> \right\vert &\geq \left\vert \big< \alpha_2,\alpha_1 \big> \right\vert  + \left\vert \big< \alpha_1, \delta \big> \right\vert \\
		&\geq \left\vert \big< \alpha_2, \gamma \big> \right\vert + \left\vert \big< \gamma,\alpha_1 \big> \right\vert + \left\vert \big< \alpha_1, \delta \big> \right\vert \\
		&\geq \left\vert \big< \alpha_2, \delta \big> \right\vert + \left\vert \big< \delta, \gamma \big> \right\vert + \left\vert \big< \gamma, \alpha_1  \big> \right\vert + \left\vert \big< \alpha_1, \delta \big> \right\vert
	\end{align*}
However, this can only hold true if $ \alpha_1 = \gamma = \delta  $ by \Cref{lem_orient_prod_zero}. But then $ \delta \in B_{\alpha_1, \alpha_2} $ which is a contradiction.

On the other hand, if $ \big< \alpha_2,\alpha_1 \big> $ and $ \big< \alpha_2,\delta \big> $ are not both strictly positive or negative, we have
	$$ \big< \alpha_2,\alpha_1 \big> \big< \alpha_2,\delta \big> \leq 0 $$
and so $ \alpha_2 \in B_{\alpha_1, \delta} $. Again by applying \Cref{lem_orient_prod_zero} three times we get
	\begin{align*}
		\left\vert \big< \alpha_1,\delta \big> \right\vert &\geq \left\vert \big< \alpha_1,\alpha_2 \big> \right\vert  + \left\vert \big< \alpha_2, \delta \big> \right\vert \\
		&\geq \left\vert \big< \alpha_1, \gamma \big> \right\vert + \left\vert \big< \gamma,\alpha_2  \big> \right\vert + \left\vert \big< \alpha_2, \delta \big> \right\vert \\
		&\geq \left\vert \big< \gamma, \delta \big> \right\vert + \left\vert \big< \delta, \alpha_1 \big> \right\vert + \left\vert \big< \gamma,\alpha_2 \big> \right\vert + \left\vert \big< \alpha_2, \delta \big> \right\vert
	\end{align*}
We deduce $ \alpha_2 = \gamma = \delta $ and so clearly $ \delta \in B_{\alpha_1,\alpha_2} $ and again we have a contradiction. Hence, $ \delta \in B_{\alpha_1, \alpha_2} $ and we are done.
	\end{enumerate}
\end{proof}	

\vspace{5mm}
\begin{figure}[h]  
\begin{center}
\pagestyle{empty}

\definecolor{qqwuqq}{rgb}{0.,0.39215686274509803,0.}
\definecolor{qqzzqq}{rgb}{0.,0.6,0.}
\definecolor{ududff}{rgb}{0.30196078431372547,0.30196078431372547,1.}
\definecolor{xdxdff}{rgb}{0.49019607843137253,0.49019607843137253,1.}
\definecolor{uuuuuu}{rgb}{0.26666666666666666,0.26666666666666666,0.26666666666666666}
\definecolor{ffqqqq}{rgb}{1.,0.,0.}
\definecolor{qqqqff}{rgb}{0.,0.,1.}
\definecolor{cqcqcq}{rgb}{0.7529411764705882,0.7529411764705882,0.7529411764705882}
\begin{tikzpicture}[line cap=round,line join=round,>=triangle 45,x=0.8cm,y=0.8cm]
\draw [color=cqcqcq,, xstep=0.8cm,ystep=0.8cm] (-8.5,-2.5) grid (4.5,7.5);
\draw[->,color=black] (-8.5,0.) -- (4.5,0.);
\foreach \x in {-8.,-7.,-6.,-5.,-4.,-3.,-2.,-1.,1.,2.,3.,4.}
\draw[shift={(\x,0)},color=black] (0pt,2pt) -- (0pt,-2pt) node[below] {\footnotesize $\x$};
\draw[->,color=black] (0.,-2.5) -- (0.,7.5);
\foreach \y in {-2.,-1.,1.,2.,3.,4.,5.,6.,7.}
\draw[shift={(0,\y)},color=black] (2pt,0pt) -- (-2pt,0pt) node[left] {\footnotesize $\y$};
\draw[color=black] (0pt,-10pt) node[right] {\footnotesize $0$};
\clip(-8.5,-2.5) rectangle (4.5,7.5);
\fill[line width=2.pt,color=qqzzqq,fill=qqzzqq,fill opacity=0.10000000149011612] (0.,0.) -- (7.098023020731952,49.68616114512364) -- (-60.,20.) -- cycle;
\draw [shift={(0.,0.)},line width=2.pt,color=qqzzqq,fill=qqzzqq,fill opacity=0.10000000149011612] (0,0) -- (81.869897645844:0.5829755139804748) arc (81.869897645844:161.56505117707798:0.5829755139804748) -- cycle;
\draw [samples=50,domain=-0.99:0.99,rotate around={-45.:(0.,0.)},xshift=0.cm,yshift=0.cm,line width=2.pt,color=qqqqff] plot ({2.7080128015453204*(1+(\x)^2)/(1-(\x)^2)},{4.69041575982343*2*(\x)/(1-(\x)^2)});
\draw [samples=50,domain=-0.99:0.99,rotate around={-45.:(0.,0.)},xshift=0.cm,yshift=0.cm,line width=2.pt,color=qqqqff] plot ({2.7080128015453204*(-1-(\x)^2)/(1-(\x)^2)},{4.69041575982343*(-2)*(\x)/(1-(\x)^2)});
\draw [line width=2.pt,dash pattern=on 4pt off 4pt,color=qqzzqq] (0.,0.)-- (7.098023020731952,49.68616114512364);
\draw [line width=2.pt,color=qqzzqq] (7.098023020731952,49.68616114512364)-- (-60.,20.);
\draw [line width=2.pt,color=qqzzqq] (-60.,20.)-- (0.,0.);
\begin{scriptsize}
\draw [fill=ffqqqq] (-3.,1.) circle (2.5pt);
\draw[color=ffqqqq] (-3.05,0.6) node {\fontsize{10}{0} $ \alpha $};
\draw [fill=ffqqqq] (1.,7.) circle (2.5pt);
\draw[color=ffqqqq] (1.85,6.85) node {\fontsize{10}{0} $ \I_{-} \left( \alpha \right) $};
\draw [fill=uuuuuu] (0.,0.) circle (2.5pt);
\draw [fill=xdxdff] (-263.4615214048064,87.82050713493545) circle (2.5pt);
\draw [fill=xdxdff] (7.098023020731952,49.68616114512364) circle (2.5pt);
\draw [fill=ududff] (-60.,20.) circle (2.5pt);
\draw[color=qqqqff] (-6.15,-1.8) node {\fontsize{10}{0} $x^2-4xy+y^2 = 22$};
\draw[color=qqqqff] (-1.8,2.8) node {\fontsize{10}{0} $B_{\alpha}$};
\draw[color=qqzzqq]  (-0.2,0.3) node {\fontsize{10}{0} $ \theta $};

\end{scriptsize}
\end{tikzpicture}
\caption{A subbranch in a cone as in the proof of \Cref{lem_sub_closed_branch}}
\label{example_cone_of_branch}
\end{center}
\end{figure}
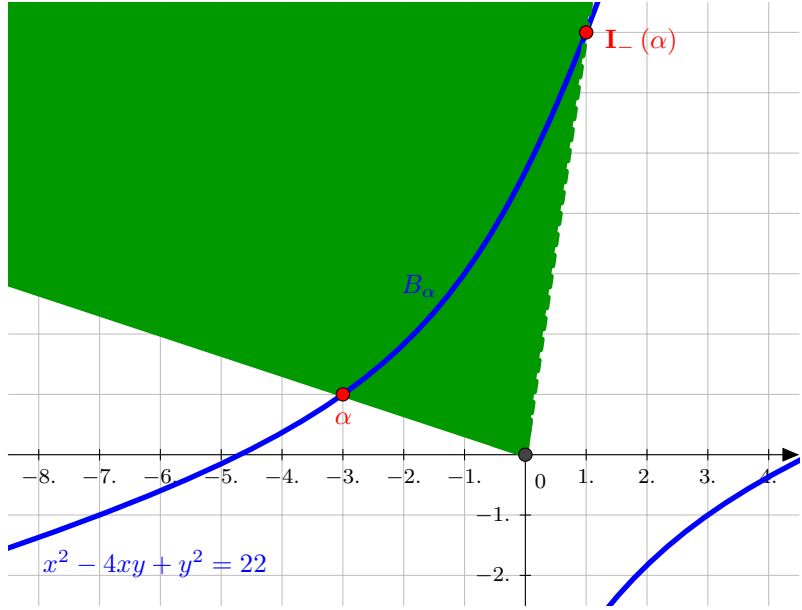
\vspace{5mm}

\begin{lemma} \label{lem_orientation_and_approx}
	Let $ z \in \mathbb{N} \setminus \left\{ 0,1 \right\} $, $ M \in \mathbb{Z} \setminus \left\{ 0 \right\} $, $ n \in \mathbb{N} \setminus \left\{ 0 \right\} $, $ \alpha, \gamma \in B $ where $ B \subseteq S_M $ is a branch and $ \gamma \in B_{\alpha} $. Then we have 
	$$ \left\vert \big< \I_{+}^{n+1} \left( \alpha \right),\gamma \big> \right\vert \geq \left\vert \big< \I_{+}^{n} \left( \alpha \right),\gamma \big> \right\vert + \left\vert \N \left( \alpha \right) \right\vert $$	
and the terms $ \big< \I_{+}^{n} \left( \alpha \right), \gamma \big> $ have the same (strictly negative or positive) sign for all $ n \in \mathbb{N} \setminus \left\{ 0 \right\} $. Moreover, for all $ n \in \mathbb{N} $ it holds true
	$$ \left\vert \big< \I_{+}^{-\left(n+1\right)} \left( \alpha \right),\gamma \big> \right\vert \geq \left\vert \big< \I_{+}^{-n} \left( \alpha \right),\gamma \big> \right\vert + \left\vert \N \left( \alpha \right) \right\vert $$
where the $ \big< \I_{+}^{-n} \left( \alpha \right),\gamma \big> $ have the same (strictly positive or negative) sign opposite to the terms above.
\end{lemma}

\begin{proof}
	Let $ M \gtrless 0 $. Therefore we have $ \big< \alpha, \gamma \big> \gtreqless 0 $ and $ \big< \I_{+} \left( \alpha \right), \gamma \big> \lessgtr 0 $. Now we are going to prove the upper part of the lemma by induction over $ n $. 
	
At first we will show the case $ n = 1 $. Observe that 
$$ \I_{+}^2 \left( \alpha \right) -z\I_{+} \left( \alpha \right) + \alpha = \left( i_z^2 -zi_z+1 \right) \alpha = 0 $$
and therefore we get 
	$$ \big< \I_{+}^2 \left( \alpha \right), \gamma \big> = z \underbrace{\big< \I_{+} \left( \alpha \right), \gamma \big>}_{\lessgtr 0} - \underbrace{\big< \alpha, \gamma \big>}_{\gtreqless 0} \lessgtr 0 $$
by using that the oriented area is bilinear. 

Now we show the inequality above for $ n = 1 $. By \Cref{proposition_prop_iz} we know that 
$$ \big<  \I_{+} \left( \alpha \right) , \I_{+}^2 \left( \alpha \right)  \big> = \N \left( \alpha \right) \gtrless 0 $$ 
and so we get 
$$ \big< \gamma, \I_{+} \left( \alpha \right) \big>\big<  \I_{+}^2 \left( \alpha \right) , \I_{+} \left( \alpha \right)  \big> \leq 0 $$
and hence $ \I_{+} \left( \alpha \right) \in B_{\gamma,\I_{+}^2 \left( \alpha \right)} $ because $ \gamma,\I_{+} \left( \alpha \right),\I_{+}^2 \left( \alpha \right) \in B $ are all on the same branch. By the inequality of \Cref{fact_ineq_concave} we have
	$$ \left\vert \big< \I_{+}^{2} \left( \alpha \right),\gamma \big> \right\vert \geq \underbrace{\left\vert \big< \I_{+}^{2} \left( \alpha \right),\I_{+} \left( \alpha \right) \big>\right\vert}_{=\left\vert \N \left( \alpha \right) \right\vert} + \left\vert \big< \I_{+} \left( \alpha \right),\gamma \big> \right\vert $$
which shows the inequality for $ n = 1 $.

Now we assume that there is an $ n \in \mathbb{N}_{\geq 1} $ such that the above inequality holds true and $ \big< \I_{+}^{m} \left( \alpha \right), \gamma \big> \lessgtr 0 $ for each $ m \in \mathbb{N}_{\leq n} \setminus \left\{ 0 \right\} $. By the induction hypothesis we have
	$$ \left\vert \big< \I_{+}^{n} \left( \alpha \right),\gamma \big> \right\vert \geq \left\vert \big< \I_{+}^{n-1} \left( \alpha \right),\gamma \big> \right\vert + \underbrace{\left\vert \N \left( \alpha \right) \right\vert}_{>0} > \left\vert \big< \I_{+}^{n-1} \left( \alpha \right),\gamma \big> \right\vert .$$

On the other hand, by using the same arguments as above we see that
\begin{align*}
	\big< \I_{+}^{n+1} \left( \alpha \right), \gamma \big> &= z \big< \I_{+}^n \left( \alpha \right), \gamma \big> - \big< \I_{+}^{n-1}\left( \alpha \right), \gamma \big> \\
	&= \underbrace{\left( z-1 \right)}_{>0} \underbrace{\big< \I_{+}^n \left( \alpha \right), \gamma \big>}_{\lessgtr 0} + \underbrace{\left( \big< \I_{+}^n \left( \alpha \right), \gamma \big> - \big< \I_{+}^{n-1}\left( \alpha \right), \gamma \big> \right)}_{\lessgtr0} \lessgtr 0
\end{align*} 
Hence, we get $ \big< \gamma, \I_{+}^n \left( \alpha \right) \big> \big< \I_{+}^{n+1} \left( \alpha \right),\I_{+}^n \left( \alpha \right) \big> \leq 0 $ and so $ \I_{+}^n \left( \alpha \right) \in B_{\gamma, \I_{+}^{n+1} \left( \alpha \right)} $ and
	$$ \left\vert \big< \I_{+}^{n+1} \left( \alpha \right),\gamma \big> \right\vert \geq \underbrace{\left\vert \big< \I_{+}^{n+1} \left( \alpha \right),\I_{+}^{n} \left( \alpha \right) \big>\right\vert}_{=\left\vert \N \left( \alpha \right) \right\vert} + \left\vert \big< \I_{+}^{n} \left( \alpha \right),\gamma \big> \right\vert $$
by \Cref{fact_ineq_concave}.

This shows the first part of the lemma. Now we prove the second part also by induction. We start with $ n = 0 $ and we show that the inequality below in the lemma holds true as well as that we have $ \big< \I_{+}^{-1} \left( \alpha \right),\gamma \big> \gtrless 0 $. Since 
	$$ \I_{+} \left( \alpha \right) -z \alpha + \I_{+}^{-1} \left( \alpha \right) = i_z^{-1}\underbrace{\left( i_z^2-zi_z+1 \right)}_{=0} \alpha = 0 $$
and the oriented area is bilinear, we get that
	$$ \big< \I_{+}^{-1} \left( \alpha \right), \gamma \big> = z\underbrace{\big< \alpha,\gamma \big>}_{ \gtreqless 0} - \underbrace{\big< \I_{+}\left( \alpha \right),\gamma \big>}_{\lessgtr 0}  \gtrless 0$$	
which shows that the sign of $ \big< \I_{+}^{-1} \left( \alpha \right), \gamma \big> $ is as claimed.	 By \Cref{proposition_prop_iz} we know that 
$$ \big<  \I_{+}^{-1} \left( \alpha \right) , \alpha \big> = \big< \alpha, \I_{+} \left( \alpha \right) \big> = \N \left( \alpha \right) \gtrless 0 $$
and so we get 
$$ \big< \I_{+}^{-1} \left( \alpha \right), \alpha \big> \big<  \gamma , \alpha  \big> \leq 0 $$
which shows that $ \alpha \in B_{\I_{+}^{-1} \left( \alpha \right), \gamma} $ as $ \I_{+}^{-1} \left( \alpha \right), \alpha, \gamma \in B $ are located on the same branch. By \Cref{fact_ineq_concave} we have
	$$ \left\vert \big< \I_{+}^{-1} \left( \alpha \right),\gamma \big> \right\vert \geq \underbrace{\left\vert \big< \I_{+}^{-1} \left( \alpha \right), \alpha \big>\right\vert}_{=\left\vert \N \left( \alpha \right) \right\vert} + \left\vert \big< \alpha,\gamma \big> \right\vert $$
which shows the inequality for $ n = 0 $.

Now we assume that the inequality above holds true for an $ n \in \mathbb{N}_{\geq 0} $ and for all $ m \in \mathbb{N}_{\leq n} $ we have that $ \big< \I_{+}^{-m} \left( \alpha \right), \gamma \big> \gtrless 0 $. By the induction hypothesis we get
	$$ \left\vert \big< \I_{+}^{-n} \left( \alpha \right),\gamma \big> \right\vert \geq \left\vert \big< \I_{+}^{-\left(n-1\right)} \left( \alpha \right),\gamma \big> \right\vert + \underbrace{\left\vert \N \left( \alpha \right) \right\vert}_{>0} > \left\vert \big< \I_{+}^{-\left(n-1\right)} \left( \alpha \right),\gamma \big> \right\vert. $$
By using the same arguments as above we see that
\begin{align*}
	\big< \I_{+}^{-\left(n+1\right)} \left( \alpha \right), \gamma \big> &= z \big< \I_{+}^{-n} \left( \alpha \right), \gamma \big> - \big< \I_{+}^{-\left(n-1\right)}\left( \alpha \right), \gamma \big> \\
	&= \underbrace{\left( z-1 \right)}_{>0} \underbrace{\big< \I_{+}^{-n} \left( \alpha \right), \gamma \big>}_{\gtrless 0} + \underbrace{\left( \big< \I_{+}^{-n} \left( \alpha \right), \gamma \big> - \big< \I_{+}^{-\left(n-1\right)}\left( \alpha \right), \gamma \big> \right)}_{\gtrless 0} \gtrless 0
\end{align*} 
which shows that the sign of $ \big< \I_{+}^{-\left(n+1\right)} \left( \alpha \right), \gamma \big> $ is the desired one. As before we have $ \I_{+}^{-n} \left( \alpha \right) \in B_{\gamma,\I_{+}^{-\left(n+1\right)} \left( \alpha \right) } $ and so we get
$$ \left\vert \big< \I_{+}^{-\left(n+1\right)} \left( \alpha \right),\gamma \big> \right\vert \geq \underbrace{\left\vert \big< \I_{+}^{-\left(n+1\right)} \left( \alpha \right),\I_{+}^{-n} \left( \alpha \right) \big>\right\vert}_{=\left\vert \N \left( \alpha \right) \right\vert} + \left\vert \big< \I_{+}^{-n} \left( \alpha \right),\gamma \big> \right\vert .$$
\end{proof}

\begin{lemma} \label{lem_dist1}
	Let $ z \in \mathbb{N} \setminus \left\{ 0,1 \right\}$, $ M \in \mathbb{Z} \setminus \left\{ 0 \right\} $, $ \alpha, \gamma \in B $ where $ B \subset S_M $ is a branch and $ \gamma \in B_{\alpha} $ . Then we have
	$$ \left\vert \N \left( \alpha \right) \right\vert > \left\vert \big< \alpha, \gamma \big> \right\vert $$
\end{lemma}

\begin{proof}
Let $ M \gtrless 0 $. Therefore $ \gamma \in B_{\alpha} $ implies that $ \big< \alpha, \gamma \big> \gtreqless 0 $ and $ \big< \I_{+} \left(\alpha\right), \gamma \big> \lessgtr 0 $. Since $ \gamma \in B_{\alpha} $, we conclude that $ \gamma \in B_{\alpha, \I_{+} \left(\alpha\right)} $ by \Cref{lem_sub_closed_branch} and therefore we get by \Cref{fact_ineq_concave}:
	$$ \left\vert \N \left( \alpha \right) \right\vert = \left\vert \big< \alpha, \I_{+} \left( \alpha \right) \big> \right\vert \geq \left\vert \big< \alpha, \gamma \big> \right\vert + \underbrace{\left\vert \big< \gamma,\I_{+} \left( \alpha \right) \big> \right\vert}_{>0} > \left\vert \big< \alpha, \gamma \big> \right\vert .$$
The last step follows because $ \I_{+} \left( \alpha \right) \notin B_{\alpha} $ and \Cref{lem_orient_prod_zero}.
\end{proof}

\begin{lemma} \label{lem_uniqueness}
	Let $ z \in \mathbb{N} \setminus \left\{ 0,1 \right\} $, $ M \in \mathbb{Z} \setminus \left\{ 0 \right\} $, $ n \in \mathbb{Z} $, $ m \in \left\{ 0,1 \right\} $, $ \alpha \in S_M $ and $ \gamma \in B_{\alpha} $. Then we have that $  \I_{+}^n \left( \left(-1 \right)^{m}\gamma \right) \in B_{\alpha } $ if and only if $ n = 0 = m $.
\end{lemma}

\begin{proof}
	Assume that there is $ n \in \mathbb{Z} $ such that $ \gamma \in B_{\alpha} $, $ \I_{+}^n \left(\left(-1 \right)^{m}\gamma \right) \in B_{\alpha} $ and let $ M\gtrless 0 $. Therefore we have $ \big< \alpha, \gamma \big> \gtreqless 0 $, $ \big< \I_{+}\left(\alpha\right)$, $ \gamma \big> \lessgtr 0 $, $ \big< \alpha, \I_{+}^n \left( \left(-1 \right)^{m}\gamma \right) \big> \gtreqless 0 $ and $ \big< \I_{+}\left(\alpha\right)$, $ \I_{+}^n \left(\left(-1 \right)^{m}\gamma \right) \big> \lessgtr 0 $. 

Since $ \big< \I_{+}^k \left( \gamma \right),\I_{+}^{k+1} \left( \gamma \right) \big> \gtreqless 0 $ for all $ k \in \mathbb{Z} $, we have that $ \I_{+}^k \left( \gamma \right) \in B_{\I_{+}^{k-1}\left( \gamma \right),\I_{+}^{k+1}\left( \gamma \right)} $ holds true for all $ k \in \mathbb{Z} $. Hence, we can use the inequality from \Cref{fact_ineq_concave} several times or \Cref{lem_orientation_and_approx} to get
\begin{align*}
	\left\vert \big< \gamma, \I_{+}^n \left( \left(-1 \right)^m \gamma \right) \big> \right\vert &=
	\left\vert \big< \gamma, \I_{+}^n \left( \gamma \right) \big> \right\vert \\ &\geq \sum_{k=0}^{\left\vert n \right\vert-1} \left\vert\big< \I_{+}^k \left( \gamma \right), \I_{+}^{k+1} \left( \gamma \right) \big>\right\vert = \left\vert n \right\vert \left\vert \N \left( \gamma \right) \right\vert = \left\vert n \right\vert \left\vert \N \left( \alpha \right) \right\vert
\end{align*}
where we can use 
	$$ \left\vert \big< \gamma, \I_{+}^n \left( \gamma \right) \big> \right\vert = \left\vert \big< \I_{+}^{-n} \left( \gamma \right), \gamma \big> \right\vert = \left\vert \big< \gamma, \I_{+}^{-n} \left( \gamma \right) \big> \right\vert $$
in case that $ n \in \mathbb{Z} $ is negative.


Now either $ \big< \I_{+}^n \left( \left(-1 \right)^m\gamma \right), \gamma \big> \lesseqgtr 0 $ or $ \big< \I_{+}^n \left( \left(-1 \right)^m\gamma \right), \gamma \big> \gtrless 0 $ which implies $ \gamma \in B_{\alpha,\I_{+}^n \left( \left(-1 \right)^m\gamma \right)} $ and $ \I_{+}^n \left( \left(-1 \right)^m\gamma \right) \in B_{\gamma,\I_{+}\left( \alpha \right)} $ or $ \I_{+}^n \left( \left(-1 \right)^m\gamma \right) \in B_{\alpha,\gamma} $ and $ \gamma \in B_{\I_{+}^n \left( \left(-1 \right)^m\gamma \right),\I_{+}\left( \alpha \right)} $, respectively. By \Cref{fact_ineq_concave} we get either
\begin{align*}
	\left\vert \N \left( \alpha \right) \right\vert &= \left\vert \big< \alpha, \I_{+} \left( \alpha \right) \big> \right\vert \\
	&\geq \left\vert \big< \alpha, \gamma \big> \right\vert + \underbrace{\left\vert \big< \gamma, \I_{+}^n \left( \left(-1 \right)^m\gamma \right)\big> \right\vert}_{\geq\left\vert n \right\vert \left\vert \N \left( \alpha \right) \right\vert} + \left\vert \big<\I_{+}^n \left( \left(-1 \right)^m\gamma \right),\I_{+} \left( \alpha \right) \big> \right\vert 
\end{align*}
or
\begin{align*}
	\left\vert \N \left( \alpha \right) \right\vert &= \left\vert \big< \alpha, \I_{+} \left( \alpha \right) \big> \right\vert \\
	&\geq \left\vert \big< \alpha, \I_{+}^n \left( \left(-1 \right)^m\gamma \right) \big> \right\vert + \underbrace{\left\vert \big< \I_{+}^n \left( \left(-1 \right)^m\gamma \right), \gamma\big> \right\vert}_{\geq\left\vert n \right\vert \left\vert \N \left( \alpha \right) \right\vert} + \left\vert \big<\gamma,\I_{+} \left( \alpha \right) \big> \right\vert
\end{align*}
where the entries of the oriented area can be exchanged if we take the absolute value of it.

Therefore we have $ n \in \left\{ -1,0,1 \right\} $. We show now that the case $ n \in \left\{ -1,1 \right\} $ is not possible. Otherwise we get that $ \left\vert \big<\gamma,\I_{+} \left( \alpha \right) \big> \right\vert = 0 $ by the second inequality which is not possible or by the first inequality $ \left\vert \big< \alpha, \gamma \big> \right\vert = 0 $, i.e. $ \gamma = \alpha $ or $ \gamma = -\alpha $ by \Cref{lem_orient_prod_zero}. However, $ \gamma = -\alpha $ is not possible as $ \big< \I_{+} \left(\alpha \right), \gamma \big> \lessgtr 0 $ would imply $ \big< \alpha, \I_{+} \left(\alpha \right) \big> \lessgtr 0 $ which is a contradiction to iv) of \Cref{proposition_prop_iz}. Hence, we need to discuss the case $ \gamma = \alpha $, i.e. we have to show that $ \I_{+} \left( \left(-1 \right)^m \gamma\right) = \I_{+} \left( \left(-1 \right)^m \alpha \right) $ and $ \I_{+}^{-1} \left( \left(-1 \right)^m \gamma \right)= \I_{+}^{-1} \left( \left(-1 \right)^m \alpha \right) $ are not contained in $ B_{\alpha} $. Observe that both of them lie on the same branch and by \Cref{proposition_prop_iz} we have that 
$$ \underbrace{\left\vert \big< \alpha,\I_{+}^{-1}\left( \left(-1 \right)^m\alpha \right) \big> \right\vert}_{= \left\vert \big< \alpha,\I_{+}^{-1}\left( \alpha \right) \big> \right\vert} = \left\vert \N \left( \alpha \right) \right\vert = \underbrace{\left\vert \big< \alpha,\I_{+}\left( \left(-1 \right)^m\alpha \right) \big> \right\vert}_{= \left\vert \big< \alpha,\I_{+}\left( \alpha \right) \big> \right\vert} $$
which is a contradiction to \Cref{lem_dist1} if we assume that either $ \I_{+}^{-1}\left( \left(-1 \right)^m \gamma\right) \in B_{\alpha} $ or $ \I_{+}\left( \left(-1 \right)^m \gamma\right) \in B_{\alpha} $. 

Hence, the remaining case is $ n = 0 $. If $ m = 0 $ there is nothing to show. Therefore we only need to show that $ m = -1 $ is not possible. In this case we have $ \pm \gamma \in B_{\alpha} $. However this is not possible as we cannot have $ \big< \alpha, \gamma \big> \gtreqless 0 $, $ \big< \I_{+}\left(\alpha\right)$, $ \gamma \big> \lessgtr 0 $ and $ \big< \alpha, -\gamma \big> \gtreqless 0 $, $ \big< \I_{+}\left(\alpha\right)$, $ -\gamma \big> \lessgtr 0 $ at the same time.
\end{proof}

\begin{lemma} \label{lem_dist2}
	Let $ z \in \mathbb{N} \setminus \left\{ 0,1 \right\} $, $ M \in \mathbb{Z} \setminus \left\{ 0 \right\} $ and $ \alpha, \gamma \in B $ where $ B \subseteq S_M $ is a branch. If $ \left\vert \big< \alpha,\gamma \big> \right\vert \geq \left\vert \N \left( \alpha \right) \right\vert $, we have that either
	$$ \left\vert \big< \I_{+}\left( \alpha \right), \gamma \big>  \right\vert \leq \left\vert \big< \alpha, \gamma \big> \right\vert - \left\vert \N \left( \alpha \right) \right\vert $$
	or
$$ \left\vert \big< \I_{+}^{-1}\left( \alpha \right), \gamma \big> \right\vert \leq \left\vert \big< \alpha, \gamma \big> \right\vert - \left\vert \N \left( \alpha \right) \right\vert .$$
\end{lemma}

\begin{proof}
	At first we will show that either $ \I_{+} \left( \alpha \right) \in B_{\alpha, \gamma} $ or $ \I_{+}^{-1} \left( \alpha \right) \in B_{\alpha, \gamma} $. If both is not the case, we have
	$$ \big< \alpha, \I_{+} \left( \alpha \right) \big> \big< \gamma, \I_{+} \left( \alpha \right) \big> > 0 $$
and
	$$ \big< \alpha, \I_{+}^{-1} \left( \alpha \right) \big> \big< \gamma, \I_{+}^{-1} \left( \alpha \right) \big> > 0 .$$
Now by applying iii) of \Cref{proposition_prop_iz} we see that the signs of $ \big< \alpha, \I_{+}^{-1} \left( \alpha \right) \big> $ and $ \big< \alpha, \I_{+} \left( \alpha \right) \big> $ must be different because
	$$ \big< \alpha, \I_{+}^{-1} \left( \alpha \right) \big> = \big< \I_{+} \left( \alpha \right), \I_{+} \left(\I_{+}^{-1} \left( \alpha \right) \right) \big> = \big< \I_{+} \left(\alpha\right), \alpha \big> = - \big< \alpha,\I_{+} \left(\alpha\right) \big> .$$
Hence, also the signs of $ \big< \I_{+}^{-1} \left( \alpha \right), \gamma \big> $ and  $\big< \I_{+} \left( \alpha \right), \gamma \big> $ must be different. Now $ \big< \alpha, \gamma \big> $ is either positive or negative and therefore we have that either $ \gamma \in B_{\alpha,\I_{+}^{-1} \left( \alpha \right)} $ or $ \gamma \in B_{\alpha,\I_{+} \left( \alpha \right)} $. By \Cref{fact_ineq_concave} we get in both cases
	$$ \left\vert \N \left( \alpha \right) \right\vert = \left\vert \big< \alpha, \I_{+} \left( \alpha \right) \big> \right\vert = \left\vert \big< \alpha, \I_{+}^{-1} \left( \alpha \right) \big> \right\vert \geq \left\vert \big< \alpha, \gamma \big> \right\vert $$
which is a contradiction to the assumption in the lemma.
Therefore either $ \I_{+} \left( \alpha \right) \in B_{\alpha, \gamma} $ or $ \I_{+}^{-1} \left( \alpha \right) \in B_{\alpha, \gamma} $ hold true. Again by \Cref{fact_ineq_concave} we get either
	$$ \left\vert \big< \alpha, \gamma \big> \right\vert \geq \left\vert \big< \alpha, \I_{+}\left( \alpha \right) \big> \right\vert + \left\vert \big< \I_{+}\left( \alpha \right), \gamma \big> \right\vert $$
or
	$$ \left\vert \big< \alpha, \gamma \big> \right\vert \geq \left\vert \big< \alpha, \I_{+}^{-1}\left( \alpha \right) \big> \right\vert + \left\vert \big< \I_{+}^{-1}\left( \alpha \right), \gamma \big> \right\vert $$
where the desired result follows by iii) and iv) of \Cref{proposition_prop_iz} because
	$$ \left\vert \big< \alpha, \I_{+} \left( \alpha \right) \big> \right\vert = \left\vert \N \left( \alpha \right) \right\vert = \left\vert \big< \I_{+}^{-1}\left( \alpha \right), \alpha \big> \right\vert .$$			
\end{proof}

\begin{lemma} \label{lem_construct_containing_subbranch}
	Let $ z \in \mathbb{N} \setminus \left\{ 0,1 \right\} $, $ M \in \mathbb{Z} \setminus \left\{ 0 \right\} $ and $ \alpha, \gamma \in S_M $. If there is $ n \in \mathbb{Z} $ such that 
	$$ \left\vert \big< \I_{+}^n \left( \alpha \right), \gamma \big> \right\vert < \left\vert \N \left( \alpha \right) \right\vert ,$$
then	 
$$ \big< \I_{+}^{n-1} \left( \alpha \right),\gamma \big> \big< \I_{+}^{n} \left( \alpha \right),\gamma \big> \leq 0 $$
or
	$$ \big< \I_{+}^{n+1} \left( \alpha \right),\gamma \big> \big< \I_{+}^{n} \left( \alpha \right),\gamma \big> \leq 0 $$
hold.	
\end{lemma}

\begin{proof}
	Assume $ \gamma \notin B_{\I_{+}^{n-1} \left( \alpha \right),\I_{+}^n \left( \alpha \right)} \cup B_{\I_{+}^{n} \left( \alpha \right),\I_{+}^{n+1} \left( \alpha \right)} $, then it must hold
	$$ \big< \I_{+}^{n-1} \left( \alpha \right),\gamma \big> \big< \I_{+}^{n} \left( \alpha \right),\gamma \big> > 0 $$
and	
	$$ \big< \I_{+}^{n+1} \left( \alpha \right),\gamma \big> \big< \I_{+}^{n} \left( \alpha \right),\gamma \big> > 0 .$$
This means that also the terms
$ \big< \gamma, \I_{+}^{n-1} \left( \alpha \right) \big> $	 and $ \big< \gamma, \I_{+}^{n+1} \left( \alpha \right) \big> $ have the same sign and so one of the following inequalities has to be satisfied: Either
	$$ \big< \gamma, \I_{+}^{n-1} \left( \alpha \right) \big> \big< \I_{+}^{n}\left( \alpha \right), \I_{+}^{n-1}\left( \alpha \right) \big> \leq 0 $$
or
	$$ \big<\gamma, \I_{+}^{n+1} \left( \alpha \right),  \big> \big< \I_{+}^{n}\left( \alpha \right), \I_{+}^{n+1}\left( \alpha \right) \big> \leq 0 $$
holds true because the terms 
$$ \big< \I^{n-1}\left( \alpha \right),\I^{n}\left( \alpha \right) \big> = \N \left( \alpha \right) = - \big< \I^{n+1}\left( \alpha \right),\I^{n}\left( \alpha \right) \big> $$
do not have the same signs by \Cref{proposition_prop_iz}. Without loss of generality, we can assume that $ \gamma \in B $ and otherwise we can work with $ -\gamma $ instead and replace $ \gamma $ everywhere by $ -\gamma $ without changing the assumptions of this lemma. Hence, we get
	$$ \I_{+}^{n-1} \left( \alpha \right) \in B_{\gamma,\I_{+}^{n} \left( \alpha \right)} $$
or
	$$ \I_{+}^{n+1} \left( \alpha \right) \in B_{\gamma,\I_{+}^{n} \left( \alpha \right)} $$
and so we can use \Cref{fact_ineq_concave} to get a contradiction 
	$$ \left\vert \big< \gamma, \I_{+}^n \left( \alpha \right) \big> \right\vert \geq \left\vert \big< \gamma, \I_{+}^{n-1} \left( \alpha \right) \big>\right\vert + \underbrace{\left\vert \big< \I_{+}^{n-1} \left( \alpha \right),\I_{+}^{n} \left( \alpha \right) \big> \right\vert}_{= \left\vert \N \left( \alpha \right) \right\vert} $$
or	
	$$ \left\vert \big< \gamma, \I_{+}^n \left( \alpha \right) \big> \right\vert \geq \left\vert \big< \gamma, \I_{+}^{n+1} \left( \alpha \right) \big>\right\vert + \underbrace{\left\vert \big< \I_{+}^{n+1} \left( \alpha \right),\I_{+}^{n} \left( \alpha \right) \big> \right\vert}_{= \left\vert \N \left( \alpha \right) \right\vert} .$$
\end{proof}

The next statement will be an important tool we use for the proof of the following theorem. 

\begin{proposition} \label{prop_part_of_branches}
	Let $ M \in \mathbb{Z} \setminus \{ 0 \} $, $ z \in \mathbb{N} $, $ S_M \subset \mathbb{R}[i_z] $ and $ \alpha \in S_M $. Then the following holds true:
	\begin{enumerate}
		\item[i)] If $ z = 0 $ and $ M > 0 $, then
			$ \coprod_{j=1}^{4}B_{\I_{+}^{j} \left(\alpha\right)} = S_M $ is a partition.
		\item[ii)] If $ z = 1 $ and $ M > 0 $, then
			$ \coprod_{j=1}^{6}B_{\I_{+}^{j} \left(\alpha\right)} = S_M $ is a partition.
		\item[iii)]	If $ z > 1 $ and $ B \subseteq S_M $ is a branch with $ \alpha \in B $, then $ \coprod_{j\in \mathbb{Z}}B_{\I_{+}^{j} \left(\alpha\right)} = B $ and $ \coprod_{j\in \mathbb{Z}, k \in \{ -1,1\}}B_{\I_{+}^{j} \left(k\alpha\right)} = S_M $ are partitions. 
	\end{enumerate}
\end{proposition}

\begin{proof}
That the subbranches are contained in $ S_M $ or in $ B $ is clear by definition and by \Cref{lem_sub_closed_branch}. We need to show that for each element in $ S_M $ or in $ B $ there is a unique subbranch as indicated in the disjoint union of $ S_M $ or in $ B $, respectively, containing this element. 

Let $ \alpha = a_1+a_2i_z \in \mathbb{R}[i_z] $. Now we can determine the intermediate angle defined by $ \alpha $ and $ \I_{+} \left( \alpha \right) $ considered as vectors in the complex plane. We have
	$$ \I_{+} \left( \alpha \right) = \left( a_1i_z + a_2i_z^2 \right) = -a_2+ \left( a_1+za_2 \right)i_z .$$
By using the scalar product we get that the angle $ \theta $ defined by $ \alpha, \I_{+} \left( \alpha \right) $ and the origin in between satisfies
	$$ \cos \left( \theta \right) = \frac{a_2^2 z}{\sqrt{a_1^2+a_2^2}\sqrt{a_2^2+\left(a_1+a_2z \right)^2}} .$$	
Hence, we clearly have that $ 0 \leq \cos \left( \theta \right) \leq 1 $ and so $ \theta $ is at most a right angle.

Let $ z = 0 $ and $ M > 0 $, then $ S_M $ is a circle of radius $ \sqrt{M} $ around the origin and $ \alpha, \I_{+} \left( \alpha \right), \I_{+}^2 \left( \alpha \right), \I_{+}^3 \left( \alpha \right) $ are distributed on the circle anticlockwise each by an angle of $ \tfrac{\pi}{2} $ to their neighbors (compare with \Cref{units_z0}). Hence, each $ \delta \in S_M $ lies exactly between two of the four elements $ \alpha, \I_{+} \left( \alpha \right), \I_{+}^2 \left( \alpha \right), \I_{+}^3 \left( \alpha \right) $ (observe that $ \I_{+}^{4} \left( \alpha \right) = \alpha $) or is exactly equal to one of them and so there is a unique subbranch containing $ \delta $.

If $ z = 1 $ and $ M > 0 $, then $ S_M $ is an ellipse. As seen in \Cref{ex_unit_ellipse} the multiplicative order of $ i_1 $ is $ 6 $ and so $ \I_{+}^j \left( \alpha \right) $ are different points on the ellipse for $ j = 0,1,2,3,4,5 $ distributed anticlockwise around an ellipse. Therefore an element $ \delta \in S_M $ is equal or located between two neighbored elements $ \I_{+}^j \left( \alpha \right), \I_{+}^{j+1}\left( \alpha \right)$ for some $ j $ and so there exist exactly one subbranch containing $ \delta $.


Now let $ z \geq 2 $ and $ M \gtrless 0 $. To prove $ \coprod_{j\in \mathbb{Z}, k \in \{ -1,1\}}B_{\I_{+}^{j} \left(k\alpha\right)} = S_M $ we need to show that for each $ \delta \in B $ or $ \delta \in S_M $, respectively, there exist a unique $ n \in \mathbb{Z} $ and a unique $ m \in \left\{ -1,1 \right\} $ such that $ \delta \in B_{\I_{+}^{n} \left(\alpha\right)} $ or $ \delta \in B_{\I_{+}^{n} \left(m\alpha\right)} $, respectively. 

Existence: Let $ \delta \in S_M $. At first we show that there is $ n \in \mathbb{Z} $ such that $ \left\vert \big< \I_{+}^n \left( \alpha \right), \delta \big> \right\vert < \left\vert \N \left( \alpha \right) \right\vert $ and then we find a subbranch which contains $ \delta $. If $ \left\vert \big< \alpha, \delta \big> \right\vert < \left\vert \N \left( \alpha \right) \right\vert $, then $ \left\vert \big< \I_{+}^n \left( \alpha \right), \delta \big> \right\vert < \left\vert \N \left( \alpha \right) \right\vert $ trivially holds true for $ n = 0 $. Assume now that $ \left\vert \big< \alpha, \delta \big> \right\vert \geq \left\vert \N \left( \alpha \right) \right\vert $. Then we can apply \Cref{lem_dist2} and either
$$ \left\vert \big< \I_{+}\left( \alpha \right), \delta \big>  \right\vert \leq \left\vert \big< \alpha, \delta \big> \right\vert - \left\vert \N \left( \alpha \right) \right\vert .$$
	or
$$ \left\vert \big< \I_{+}^{-1}\left( \alpha \right), \delta \big> \right\vert \leq \left\vert \big< \alpha, \delta \big> \right\vert - \left\vert \N \left( \alpha \right) \right\vert $$
hold true. Hence, if $ \left\vert \big< \I_{+}\left( \alpha \right), \delta \big>  \right\vert $ and $ \left\vert \big< \I_{+}^{-1}\left( \alpha \right), \delta \big> \right\vert $ are still larger than $ \left\vert \N \left( \alpha \right) \right\vert $, we can proceed with \Cref{lem_dist2} applied to the smaller term of both until we get the first $ n \in \mathbb{Z} $ such that $ \left\vert \big< \I_{+}^n \left( \alpha \right), \delta \big> \right\vert < \left\vert \N \left( \alpha \right) \right\vert $. Observe that 
	$$ \left\vert \big< \alpha, \delta \big> \right\vert \geq \left\vert \big< \I_{+} \left(\alpha\right), \delta \big> \right\vert + \left\vert \N \left( \alpha \right) \right\vert \geq \dots \geq \left\vert \big< \I_{+}^{\left\vert n \right\vert} \left(\alpha\right), \delta \big> \right\vert + \left\vert n \left\vert \right\vert \N \left( \alpha \right) \right\vert $$
or	
$$ \left\vert \big< \alpha, \delta \big> \right\vert \geq \left\vert \big< \I_{+}^{-1} \left(\alpha\right), \delta \big> \right\vert + \left\vert \N \left( \alpha \right) \right\vert \geq \dots \geq \left\vert \big< \I_{+}^{-\left\vert n \right\vert} \left(\alpha\right), \delta \big> \right\vert + \left\vert n \left\vert \right\vert \N \left( \alpha \right) \right\vert $$
must hold depending on whether $ n $ is positive or not (i.e. $ n = \left\vert n \right\vert $ or $ n = -\left\vert n \right\vert $). The reason why such an $ n \in \mathbb{Z} $ has to exist is that there is an $ m \in \mathbb{N} $ such that $ \left\vert \big< \alpha, \delta \big> \right\vert - m  \left\vert \N \left( \alpha \right) \right\vert < \left\vert \N \left( \alpha \right) \right\vert > 0$ and so $ \left\vert n \right\vert \leq m $.



By \Cref{lem_construct_containing_subbranch} we can assume that either 
$$ \big< \I_{+}^{n-1} \left( \alpha \right),\delta \big> \big< \I_{+}^{n} \left( \alpha \right),\delta \big> \leq 0 $$
or 
$$ \big< \I_{+}^{n} \left( \alpha \right),\delta \big> \big< \I_{+}^{n+1} \left( \alpha \right),\delta \big> \leq 0 .$$ Now we will discuss both cases. At first consider
$$ \big< \I_{+}^{n-1} \left( \alpha \right),\delta \big> \big< \I_{+}^{n} \left( \alpha \right),\delta \big> \leq 0 $$ and hence either $ \big< \I_{+}^{n-1} \left( \alpha \right),\delta \big> \gtreqless 0 $ and $ \big< \I_{+}^{n} \left( \alpha \right),\delta \big> \lesseqgtr 0 $ or both relations are exchanged. In case $ \big< \I_{+}^{n} \left( \alpha \right),\delta \big> = 0 $, then we have $ \delta \in \left\{ -\I_{+}^{n} \left( \alpha \right),\I_{+}^{n} \left( \alpha \right) \right\} $ by \Cref{lem_orient_prod_zero} and so $ \delta \in B_{\I_{+}^n \left( \alpha \right)} $ or $ \delta \in B_{\I_{+}^n \left( -\alpha \right)} $.
Otherwise we have $ \big< \I_{+}^{n} \left( \alpha \right),\delta \big> \lessgtr 0 $ and then $ \delta \in B_{\I_{+}^{n-1} \left( \alpha \right)} $. In case the relations are exchanged, i.e. $ \big< \I_{+}^{n-1} \left( \alpha \right),\delta \big> \lesseqgtr 0 $ and $ \big< \I_{+}^{n} \left( \alpha \right),\delta \big> \gtreqless 0 $, then we have that $ \big< \I_{+}^{n-1} \left( -\alpha \right),\delta \big> \gtreqless 0 $ and $ \big< \I_{+}^{n} \left( -\alpha \right),\delta \big> \lesseqgtr 0 $ and so we can do the same discussion as above where all $ \alpha $'s are exchanged by $ -\alpha $. Hence, we can show that either $ \delta \in B_{\I_{+}^{n-1} \left( -\alpha \right)} $ or $ \delta \in B_{\I_{+}^{n} \left( -\alpha \right)} $.

Secondly, if $$ \big< \I_{+}^{n} \left( \alpha \right),\delta \big> \big< \I_{+}^{n+1} \left( \alpha \right),\delta \big> \leq 0 ,$$ then we can assume that $ \delta \notin \left\{ - \I_{+}^{n+1} \left( \alpha \right), \I_{+}^{n+1} \left( \alpha \right) \right\} $ because $ \left\vert \big< \I_{+}^n \left( \alpha \right), \delta \big> \right\vert < \left\vert \N \left( \alpha \right) \right\vert $. Hence, we deduce either $ \big< \I_{+}^{n} \left( \alpha \right),\delta \big> \gtreqless 0 $ and $ \big< \I_{+}^{n+1} \left( \alpha \right),\delta \big> \lessgtr 0 $ or $ \big< \I_{+}^{n} \left( \alpha \right),\delta \big> \lesseqgtr 0 $ and $ \big< \I_{+}^{n+1} \left( \alpha \right),\delta \big> \gtrless 0 $, i.e. $ \big< \I_{+}^{n} \left( -\alpha \right),\delta \big> \gtreqless 0 $ and $ \big< \I_{+}^{n+1} \left( -\alpha \right),\delta \big> \lessgtr 0 $. Thus, we get either $ \delta \in B_{\I_{+}^{n} \left( \alpha \right)} $ or $ \delta \in B_{\I_{+}^{n} \left( -\alpha \right)} $.

This shows that $ \coprod_{j\in \mathbb{Z}, k \in \{ -1,1\}}B_{\I_{+}^{j} \left(k\alpha\right)} $ covers $ S_M $. Since $ S_M $ consists of two branches and one branch $ B $ does not contain $ -\alpha $ and all the elements which we get by applying $ \I_{+}^n $ to $ \alpha $ for $ n \in \mathbb{Z} $ we deduce that $ B_{\I_{+}^n \left(-\alpha \right)} \nsubseteq B $ and so we must have that $ B $ is covered by $ \coprod_{j\in \mathbb{Z}}B_{\I_{+}^{j} \left(\alpha\right)} $.

Uniqueness: We have to show that $ \delta \in S_M $ can be contained in at most one of these subbranches. Assume not, then we find $ n_1,n_1 \in \mathbb{Z} $ and $ m_1,m_2 \in \left\{ 0,1 \right\} $ such that $ \delta \in B_{\I_{+}^{n_j}\left( \left( -1\right)^{m_j}\alpha\right)}$ for $ j = 1,2 $. Define $ \gamma \coloneqq \I_{+}^{n_1}\Big( \left( -1\right)^{m_1}\alpha\Big) $, then we have that $ \alpha = \I_{+}^{-n_1}\left( \left( -1\right)^{-m_1}\gamma\right) $ and so $ \I_{+}^{n_2}\left( \left( -1\right)^{m_2}\alpha\right) = \I_{+}^{n_2-n_1}\left( \left( -1\right)^{m_2-m_1}\gamma\right) $. Therefore we can also say $ \delta \in B_{\gamma} $ and $ \delta \in B_{\I_{+}^{n_2-n_1}\left( \left( -1\right)^{m_2-m_1}\gamma\right)} $ where the latter is equivalent to $ \I_{+}^{n_1-n_2}\left( \left( -1 \right)^{m_1-m_2} \delta \right) \in B_{\gamma} $ by \Cref{lem_branch_rules}. By \Cref{lem_uniqueness} we conclude that $ n_1=n_2 $ and $ m_1=m_2 $ which shows that the subbranch containing $ \delta $ is unique. This shows that all branches of the form $ B_{\I_{+}^{j} \left(k\alpha\right)} $ for all $ j\in \mathbb{Z} $ and $ k \in \left\{ -1,1\right\} $ are pairwise disjoint.
\end{proof}


\begin{corollary} \label{coro_neg_partition}
	Let $ M \in \mathbb{Z} \setminus \{ 0 \} $, $ z \in \mathbb{N} $, $ S_M \subset \mathbb{R}[i_{-z}] $ and $ \alpha \in S_M $. Then the following holds true:
	\begin{enumerate}
		\item[i)] If $ z = 0 $ and $ M > 0 $, then
			$ \coprod_{j=1}^{4}B_{\I_{-}^{j} \left(\alpha\right)} = S_M $ is a partition.
		\item[ii)] If $ z = -1 $ and $ M > 0 $, then
			$ \coprod_{j=1}^{6}B_{\I_{-}^{j} \left(\alpha\right)} = S_M $ is a partition.
		\item[iii)]	If $ z < -1 $ and $ B \subseteq S_M $ is a branch with $ \alpha \in B $. Then $ \coprod_{j\in \mathbb{Z}}B_{\I_{-}^{j} \left(\alpha\right)} = B $ and $ \coprod_{j\in \mathbb{Z}, k \in \{ -1,1\}}B_{\I_{-}^{j} \left(k\alpha\right)} = S_M $ are partitions. 
	\end{enumerate}
\end{corollary}

\begin{proof}
	Use the isomorphism $ \Phi $ between $ z $-rings, \Cref{prop_part_of_branches}, \Cref{lem_branch_rules} and \Cref{coro_In} as well as its proof.
\end{proof}


We are finally ready for one of the main results of this section and its proof:

\begin{theorem}[Local Solution Theorem 1]\label{theo_local_branch}
	Let $ z \in \mathbb{Z} $ and $ M \in \mathbb{Z} \setminus \{ 0\} $. Then the Diophantine equation $ x^2+zxy+y^2 = M $ is solvable if and only if $ S_M \neq \emptyset $ and for all $ \alpha \in S_M $ we have that $ B_{\alpha} \cap \mathbb{Z}[i_z] \neq \emptyset $.	
\end{theorem}	
	

Recall that in case $ \vert z \vert \leq 1 $ and $ M < 0 $ we have that $ S_M = \emptyset $ and so it is clear that in this case $ x^2+zxy+y^2 = M $ is not solvable, see \Cref{ex_non_neg_sol}. However, if $ M > 0 $, $ S_M $ is not empty, so we can choose $ \alpha \in S_M \subseteq \mathbb{R}[i_z] $ ($ \alpha $ does not have to be an element of $ \mathbb{Z}[i_z] $) and reduce the problem of solvability of $ x^2+zxy+y^2 = M $ to the local problem whether $ B_{\alpha} $ does contain an integer solution of $ x^2+zxy+y^2 = M $ or not. If not, then $ x^2+zxy+y^2 = M $ is not solvable at all, compare with \Cref{example_z6_pm3_pm7}.	
	
	

\begin{proof}[Proof of \Cref{theo_local_branch}]	
	Assume that the Diophantine equation $ x^2+zxy+y^2 = M $ for $ z \in \mathbb{Z} $ and $ M \in \mathbb{Z} \setminus \{ 0\} $ can be solved by $ \gamma \in \mathbb{Z}[i_z] $. Then $ \gamma \in S_M \neq \emptyset $. Let $ \alpha \in S_M $ be arbitrary. We have to show that $ B_{\alpha} \cap \mathbb{Z}[i_z] \neq \emptyset $. To make the notation easier we will now denote $ \I_{+} $ or $ \I_{-} $ by $ \I $ depending whether $ z \geq 0 $ or $ z < 0 $, respectively. Let $ B $ be the branch containing $ \alpha $. Then either $ \gamma \in B $ or $ -\gamma \in B $ (or both if $ z \in \left\{-1,0,1 \right\} $). Hence, by \Cref{prop_part_of_branches} and \Cref{coro_neg_partition} we find $ n \in \mathbb{Z} $ such that either $ \gamma \in B_{\I^n \left( \alpha \right)} $ or $ -\gamma \in B_{\I^n \left( \alpha \right)} $. Since $ -\gamma $ also solves the Diophantine equation above, we can assume, without loss of generality, that the first case holds true (otherwise we exchange $ \gamma $ by $ -\gamma $). 
	
By \Cref{lem_branch_rules} $ \gamma \in B_{\I^n \left( \alpha \right)} = \I^{n}\left( B_{\alpha} \right) $ and this is equivalent to $ \I^{-n}\left( \gamma \right) \in B_{\alpha} $. Since $ \gamma \in \mathbb{Z}[i_z] $ we also have $ \I^{-n}\left( \gamma \right) \in \mathbb{Z}[i_z] $ by v) of \Cref{proposition_prop_iz} and \Cref{coro_In} and so $ B_{\alpha} \cap \mathbb{Z}[i_z] \neq \emptyset $.


The reverse direction is clear as an element of the set $ B_{\alpha} \cap \mathbb{Z}[i_z] \subseteq S_M $ satisfies the Diophantine equation $ x^2+zxy+y^2 = M $.
\end{proof}

\begin{example}
	We can verify the statement of \Cref{theo_local_branch} on \Cref{units_z0}. For example, we see that $ B_{\sqrt{6}-2i} \cap \mathbb{Z}[i] \neq \emptyset $. Indeed, $ 3 \pm i_z $ solves $ x^2+y^2 = 10 $. Also the other branches contain exactly two solutions to the above Diophantine equation (see the intersections of the blue circle with the $ \mathbb{Z} \times \mathbb{Z} $-grid) and so it does not matter which branch we consider. It would even work if we choose another $ \alpha \in S_{10} $.
\end{example}

In fact, a consequence of \Cref{theo_local_branch} is that if we find no positive solution (i.e. $ x,y \geq 0 $) to $ x^2 + zxy + y^2 = M $ for $ z \in \mathbb{N} $ and $ M \in \mathbb{N} \setminus \{ 0 \} $, then the Diophantine equation has no solution in general what we will show now.

\begin{corollary} \label{coro_pos_sol}
	If the Diophantine equation $ x^2 + zxy + y^2 = M $ is solvable for $ x,y \in \mathbb{Z} $ where $ z \in \mathbb{N} $ and $ M \in \mathbb{N} \setminus \{ 0 \} $, then there exist a solution for it where both, $ x,y $, are non-negative. Moreover, if $ \alpha \in \mathbb{Z}[i_z] $ is a solution to $ x^2 + zxy + y^2 = M $, then there is a unique unit $ \varepsilon \in \mathbb{Z}[i_z] $ such that $ \varepsilon \alpha $ is a positive solution to the Diophantine equation above.
\end{corollary}

To prove uniqueness of $ \varepsilon $ we need to know more about the units in $ \mathbb{Z}[i_z] $ if $ z \in \mathbb{N} $. Therefore we will postpone it to the next section and only prove the existence of $ \varepsilon $ in the following part.

\begin{proof}
Observe that $ \sqrt{M} \in S_M $ and $ \I_{+} \left( \sqrt{M}\right) = \sqrt{M}i_z $. Moreover, $ B_{\sqrt{M}} \cup \{ \sqrt{M}i_z \} $ is the part of $ S_M $ in the first quadrant. Hence, by \Cref{theo_local_branch} a solution to the Diophantine equation exists if and only if $ B_{\sqrt{M}}\cap \mathbb{Z}[i_z] \neq \emptyset $, i.e. we find a positive solution. Furthermore, if $ \alpha \in \mathbb{Z}[i_z] $ is any solution to the Diophantine equation $ x^2 + zxy + y^2 = M $, then we find unique $ n \in \mathbb{Z} $ and $ m \in \left\{ -1,1 \right\} $ such that $ \alpha \in B_{\I_{+}^{-n} \left( \left( -1 \right)^m \sqrt{M} \right)} $ by \Cref{prop_part_of_branches} which is equivalent to $ \I_{+}^n \left( \left( -1 \right)^m \alpha \right) = \left( -1 \right)^mi_z^n \alpha \in B_{\sqrt{M}} $, so $ \varepsilon \coloneqq \left( -1 \right)^mi_z^n $ is the desired unit such that $ \varepsilon \alpha $ is a positive solution to $ x^2 + zxy + y^2 = M $.
\end{proof}

Note that we cannot show now that $ \varepsilon $ is unique as there might also exist other units in $ \mathbb{Z}[i_z] $ such that $ \varepsilon \alpha $ will be a positive solution to $ x^2 + zxy + y^2 = M $.

\begin{example} \label{example_exist_sol}
We can show that the Diophantine equation 
	$$ x^2+6xy+y^2 = 7 $$
has no solution by considering its graph in the first quadrant of the complex plane and seeing that there is no intersection with the $ \mathbb{Z} \times \mathbb{Z} $-grid (see \Cref{example_pm3_pm7}). We could also argue in the follwoing way: If there is a positive solution and $ x = 0  $ or $ y = 0 $ does not work as $ 7 $ is not a square, we would have $ x,y >0 $ where 
	$$ x^2+6xy+y^2 \geq 1^2 + 6 \cdot 1 \cdot 1 + 1^2 > 7 .$$ 
This would be a contradiction to the existence of a positive solution by \Cref{coro_pos_sol} and so $ x^2+6xy+y^2 = 7 $ is not solvable.
\end{example}
	


So far we know that the existence or non-existence of a solution to the Diophantine equation $ x^2+zxy+y^2 = M $ for $ z \in \mathbb{Z} $, $ M \in \mathbb{Z} \setminus \{ 0 \} $ can be proved by considering any subbranch in $ S_M $, i.e. some bounded and connected subset of $ S_M $ which contains a solution if and only if the Diophantine equation is solvable. Our goal now is to develop another criterion for proving the non-existence of a solution to $ x^2+zxy+y^2 = M $ by considering a connected part of a branch which contains no solution to $ x^2+zxy+y^2 = M $, but a subbranch. To find out whether a subbranch is contained in the considered part of the branch we will “measure” the “length” of the part of the branch by using the oriented area. This only works if all our branches are concave. For this approach use closed branches as  closed branches are easy to work with (we can choose start and end points) and so we get another criterion simpler to handle for proving the non-existence of a solution.

\begin{theorem}[Local Solution Theorem 2] \label{theo_no_solution} 
	Let $ z \in \mathbb{Z} \setminus \{ -1,0,1 \} $, $ M \in \mathbb{Z} \setminus \{ 0 \} $ and $ B \subseteq S_M $ be a branch where $ \alpha_1, \alpha_2 \in B $. If $ B_{\alpha_1, \alpha_2} \cap \mathbb{Z}[i_z] = \emptyset $ and $ \vert \big< \alpha_1 , \alpha_2 \big> \vert \geq \vert M \vert $, then the Diophantine equation $ x^2+ zxy + y^2 = M $ has no solution.
	\end{theorem}
	
\begin{proof}
The idea of the proof is the following: We will show that $ B_{\alpha_1,\alpha_2} $ contains a subbranch and then we can apply the Local Solution Theorem.	Let $ \I $ denote $ \I_{+} $ if $ z \geq 0 $ and $ \I_{-} $ if $ z < 0 $. Recall that $ \alpha_j, \I \left( \alpha_j \right) $ for $ j = 1,2 $ are on the same branch. At first we will show that either $ \I \left( \alpha_1 \right) \in B_{\alpha_1, \alpha_2} $ or $ \I \left( \alpha_2 \right) \in B_{\alpha_1, \alpha_2} $.

More concretely, let $ M \gtrless 0 $ and assume without loss of generality that $ \big< \alpha_1, \alpha_2 \big> \gtreqless 0 $ (otherwise we can just exchange $ \alpha_1 $ and $ \alpha_2 $) and show that then $ \I \left( \alpha_1 \right) \in B_{\alpha_1, \alpha_2} $. If not, then we have
	$$ \big< \alpha_1, \I \left( \alpha_1 \right) \big> \big< \alpha_2, \I \left( \alpha_1 \right) \big> > 0 .$$
Since $ \big< \alpha_1, \I \left( \alpha_1 \right) \big> \gtrless 0 $ we also have that $ \big< \alpha_2, \I \left( \alpha_1 \right) \big> \gtrless 0 $ and hence
	$$ \big< \alpha_1, \alpha_2 \big> \big< \I \left(\alpha_1\right), \alpha_2 \big> \leq 0 ,$$
so we have that $  \alpha_2 \in B_{\alpha_1, \I \left( \alpha_1 \right)} $ and by \Cref{fact_ineq_concave} it follows
	$$ \left\vert M \right\vert = \vert \big< \alpha_1,\I \left(\alpha_1 \right) \big> \vert \geq \vert \big< \alpha_1,\alpha_2 \big> \vert + \vert \big<\alpha_2, \I \left( \alpha_1 \right)\big> \vert \geq \left\vert M \right\vert $$
which implies that $ \vert \big<\alpha_2, \I \left( \alpha_1 \right)\big> \vert = 0 $, so $ \alpha_2 = \I \left( \alpha_1 \right) $ which is a contradiction because we assumed that $ \I \left( \alpha_1 \right) \notin B_{\alpha_1, \alpha_2} $.

Hence, we can assume $ \I \left( \alpha_1 \right) \in B_{\alpha_1, \alpha_2} $ and we also have $ B_{\alpha_1} \subseteq B_{\alpha_1,\I \left( \alpha_1 \right)} \subseteq B_{\alpha_1,\alpha_2} $ by \Cref{lem_sub_closed_branch}. Moreover, $ B_{\alpha_1, \alpha_2} \cap \mathbb{Z}[i_z] = \emptyset $ by assumption and so also $ B_{\alpha_1} \cap \mathbb{Z}[i_z] = \emptyset $. By \Cref{theo_local_branch} we conclude.
\end{proof}

We will see concrete applications of the last statement in the next few sections.

\subsection{Unit group of $ \mathbb{Z}[i_z] $}

The aim of this section is to identify the set of units in each $ z $-ring. For this we would like to prove the following theorem.

\begin{theorem}[Characterization of unit groups of $ z $-rings] \label{theorem_units}
	Let $ z \in \mathbb{Z} $. Then the set of units in $ \mathbb{Z}[i_z] $ is isomorphic to the additive group
	\begin{itemize}
		\item $ \mathbb{Z}/4	\mathbb{Z} $ and generated by $ i_z,-i_z $ if $ z = 0 $
		\item $ \mathbb{Z}/6	\mathbb{Z} $ and generated by $ \pm i_z $ if $ z = \pm 1 $, respectively
		\item $ \mathbb{Z}/2	\mathbb{Z} \times \mathbb{Z} $ and generated by $ -1,\pm i_z $ if $ z = 2 \vee z \geq 4 $ or $ z = -2 \vee z \leq -4 $, respectively
		\item $ \mathbb{Z}/2	\mathbb{Z} \times \mathbb{Z} $ and generated by $ -1,-1 \pm i_z $ if $ z = \pm 3 $, respectively.
	\end{itemize}
\end{theorem}

By applying a theorem of Gauss \cite[p.57]{Mordell_dio_equ} we can deduce that the Diophantine equation $ x^2+zxy+y^2 = 1 $ for $ z \in \mathbb{Z} $ has infinitely many solutions if the discriminant $ D = z^2-4 > 0 $ is not a prefect square. I.e. the unit sets of $ z $-rings have infinite cardinality if $ z \geq 3 $ (and of course also for $ z \leq -3 $). For $ \vert z \vert \leq 1 $ we know that all level sets are bounded and so it is easy to see that the unit sets must be finite (compare with \Cref{eisenstein_integers}). For $ z \in \left\{ -2,2 \right\} $ we will see that there are also infinitely many units in $ \mathbb{Z}[i_z] $.

\begin{proof}[Proof of \Cref{theorem_units}]	
		At first let $ z \in \{0,1 \} $. In these cases $ S_1 $ is bounded and consists of one branch and $ S_{-1} = \emptyset $. Therefore we can count the units (compare with \Cref{ex_unit_ellipse}). There are $ 4 $ and $ 6 $ units in $ S_1\cap \mathbb{Z}[i_z] $ for $ z = 0 $ and $ z = 1 $, respectively. Moreover, the unit $ i_z $ has order $ 4 $ in $ \mathbb{Z}[i] $ and $ 6 $ in $ \mathbb{Z}[i_1] $. By the fundamental theorem of finitely generated abelian groups we deduce that the set of unit groups of $  \mathbb{Z}[i] $ and $  \mathbb{Z}[i_1] $ are isomorphic to $ \mathbb{Z}/4	\mathbb{Z} $ and $ \mathbb{Z}/6	\mathbb{Z} $, respectively.
	
	
Let now $ z \geq 2 $. At first we consider an arbitrary unit $ \varepsilon \in \mathbb{Z}[i_z] $. This means $ \N(\varepsilon) \in \left\{ -1,1\right\} $ by \Cref{Nz_lemma}. We will discuss both cases below.

Consider the subbranch $ B_{1} \subset S_{1} $. Observe that $ B_1 \subseteq B_{1,\I_{+}\left( 1\right)} = B_{1,i_z} $ is entirely contained in the first quadrant. By \Cref{prop_part_of_branches} for each unit $ \varepsilon \in S_1 $ there exist $ n \in \mathbb{Z} $ and $ k \in \left\{ 0,1\right\} $ such that $ \varepsilon \in B_{\I_{+}^n \left( \left(-1\right)^k \right)} $ which is equivalent to 
	$$ \I_{+}^{-n}\left(\left(-1\right)^k \varepsilon \right) \in B_{1} .$$
Now we would like to show that $ 1 $ is the only unit contained in $ B_{1} $. If we have a unit $ a+bi_z \in \mathbb{Z}[i_z] $ in the first quadrant, then clearly $ a,b \geq 0 $ and $ a,b $ are not zero at the same time. Moreover $ a,b \geq 1 $ is not possible as then
	$$ 1 = a^2 + zab + b^2 \geq 2 + z > 1 .$$
Hence, $ 1 $ and $ i_z $ are the only units in the first quadrant of $ \mathbb{Z}[i_z] $. Since $ \I_{+} \left( 1\right) = i_z \notin B_1 $ we conclude 
	$$ \I_{+}^{-n}\left(\left(-1\right)^k \varepsilon \right) = 1 $$
and so each unit in $ \mathbb{Z}[i_z] $ with norm equal to $ 1 $ is of the form
	$$ \varepsilon = \I_{+}^{n} \left( \left( -1\right)^{k} \right)= \left( -1\right)^{k}i_z^n. $$
Observe that this holds for all units with norm $ 1 $ in $ \mathbb{Z}[i_z] $ if $ z \geq 2 $. If $ z = 2 $, then $ S_{-1} = \emptyset $, so there are no units with norm equal to $ -1 $. Hence, the units of $ \mathbb{Z}[i_2] $ are generated by $ -1 $ and $ i_z $. Observe that $ i_z \in \mathbb{Z}[i_z] $ must have infinite order for $ z \geq 2 $ because otherwise the subbraches in \Cref{prop_part_of_branches} would not define a partition. Hence, the units of $ \mathbb{Z}[i_2] $ are isomorphic to $ \mathbb{Z}/2	\mathbb{Z} \times \mathbb{Z}$.


Let now $ z \geq 3 $ and $ \varepsilon \in \mathbb{Z}[i_3] $ be an arbitrary unit with $ \N \left( \varepsilon \right) = -1 $. Observe that $ \I_{+} \left(-1+i_3\right) = -1+2i_3 $ and so $ -1+2i_3 \notin B_{-1+i_3} $. Let $ B $ be the branch which contains $ B_{-1+i_3} $. Since $ B $ is concave, we have 
	$$ B_{-1+i_3} \subseteq B_{-1+i_3,-1+2i_3} \subseteq [-1,0] \times [1,2]i_3 .$$
Moreover, $ B $ does not intersect the axis of the complex plane and so $ B_{-1+i_3} \cap \mathbb{Z}[i_3] = \left\{ -1+i_3 \right\} $. Now if $ \varepsilon \in \mathbb{Z}[i_z] $ is a unit with norm equal to $ -1 $, then we find $ n \in \mathbb{Z} $ and $ k \in \left\{ 0,1\right\} $ such that $ \varepsilon \in B_{\I_{+}^n \left( -1+i_3 \right)} $ by \Cref{prop_part_of_branches}. With the same argument as above we deduce 
	$$ \varepsilon = \I_{+}^{n} \left( \left( -1\right)^{k}\left( -1+i_3 \right) \right)= \left( -1\right)^{k}i_3^n \left( -1+i_3 \right). $$
However, since 
	$$ \left( -1+i_3 \right)^2 = 1 - 2i_3 + i_3^2 = 1 - 2i_3 + 3i_3-1 = i_3 $$
we have in fact 
	$$ \varepsilon = \left( -1 \right)^k {\left( -1+i_3 \right)}^{2n+1}. $$
Hence, the unit group of $ \mathbb{Z}[i_3] $ is generated by $ -1, -1+i_3 $ and all the units with norm equal to $ 1 $ are generated by $ -1+i_3 $ with an even exponent whereas odd exponents are used to generate units with norm equal to $ -1 $. Thus, the unit group of $ \mathbb{Z}[i_3] $ is isomorphic to $ \mathbb{Z}/2	\mathbb{Z} \times \mathbb{Z} $. 

Now we consider the case $ z > 3 $. It remains to show that $ S_{-1} \cap \mathbb{Z}[i_z] $ is empty because we already know that $ S_{1} \cap \mathbb{Z}[i_z] $ is isomorphic to $ \mathbb{Z}/2\mathbb{Z} \times \mathbb{Z} $. For this we consider the branch of $ S_{-1} $ entirely in the fourth quadrant. We will denote it by $ B $ and it is enough to show that $ B $ contains no unit of $ \mathbb{Z}[i_z] $ because then the other branch does neither by symmetry reasons. At first we show that there is an element $ \gamma = c\left( 1-i_z \right) \in B $ for some $ c \in \mathbb{R} $. Since $ \gamma \in B $ is in the fourth quadrant, we have that $ c>0 $ and
	$$ c^2-zc^2+c^2=-1. $$
Hence, $ c = \sqrt{\frac{1}{z-2}}<1 $. 

Consider now
$$ G \coloneqq \{ \beta \in B \mid 0 < \mathrm{Re}\left( \beta \right) < 1 \vee -1 < \mathrm{Im}\left( \beta \right) < 0 \} .$$ 
Since $ 0 < c < 1 $, we have $ \gamma \in G $ and so $ G $ is not empty. We try to estimate the coordinates of the elements on the boundary of $ G $ (they are not contained in $ G $). For this consider $ x^2+zxy+y^2 =-1 $ and let $ x = 1 $. We get
$$ y = \frac{-z \pm \sqrt{z^2-8}}{2} $$
We would like to study both of these solutions and call $ y_{+} $ the solution with the plus sign and $ y_{-} $ the other one. Clearly for both of them it holds $ y < 0 $ because $ B $ lies entirely in the fourth quadrant. Moreover, we have
	\begin{align*}
		y_{+} &= \frac{-z + \sqrt{z^2-8}}{2} \\
			&= \frac{-z + \sqrt{ \left( z-2 \right)^2+4z-12 }}{2} \\
			&> \frac{-z +  z-2 }{2} \\
			&= -1
	\end{align*}
and
	\begin{align*}
		y_{-} &= \frac{-z - \sqrt{z^2-8}}{2} \\
			&= \frac{-z - \sqrt{ \left( z-2 \right)^2+4z-12 }}{2} \\
			&< \frac{-z - \left( z-2 \right)}{2} \\
			&= -z+1.
	\end{align*}
Hence, $ \alpha_1 \coloneqq 1+y_{-}i_z $ is an element on the boundary of $ G $. By symmetry we can conclude that there must be another such element $ \alpha_2 \coloneqq  x_{+}-i_z $ where $ x_{+}>z-1 $. Consider now the closed branch $ B_{\alpha_1,\alpha_2} $. Then we have $ B_{\alpha_1,\alpha_2} \setminus \left\{\alpha_1, \alpha_2 \right\} \subset G $ and $ G \cap \mathbb{Z}[i_z] = \emptyset $. Moreover, $ \alpha_j \notin \mathbb{Z}[i_z] $ for $ j = 1,2 $ and $ z \geq 4 $. Thus, we have $ B_{\alpha_1,\alpha_2} \cap \mathbb{Z}[i_z] = \emptyset $.

Furthermore,
	\begin{align*}
		\left\vert \big< \alpha_1 , \alpha_2 \big> \right\vert &= \left\vert -1-y_{-}x_{+} \right\vert \\
		&> -1+\left( z-1 \right)^2 \\
		&> 1 
	\end{align*}
and so we have that the Diophantine equation $ x^2+zxy+y^2 = -1 $ is not solvable for $ z \geq 4 $ by \Cref{theo_no_solution}. This means that there are no units in $ \mathbb{Z}[i_z] $ with norm equal to $ -1 $ and so the group structure of the units of $ \mathbb{Z}[i_z] $ as well as its generators are in this case the same as for $ z = 2 $ (compare with \Cref{proof_units} for $ z = 7 $).
	
We will now consider the unit group of $ \mathbb{Z}[i_z] $ if $ z < 0 $. Observe that the ring isomorphism $ \Phi $ between $ \mathbb{Z}[i_z] $ and $ \mathbb{Z}[i_{-z}] $ respects the norm and so also the units. Hence, the group structures of the unit groups are the same for $ z $ as for $ -z $. The only thing which changes are the generators of the units because $ \Phi $ changes the imaginary parts, i.e. the imaginary parts of all generators of units in $ \mathbb{Z}[i_{z}] $ are a multiple of $ -1 $ compared to the imaginary parts of the generators uf units in $ \mathbb{Z}[i_{-z}] $. Thus, we conclude.
\end{proof}

\vspace{5mm}
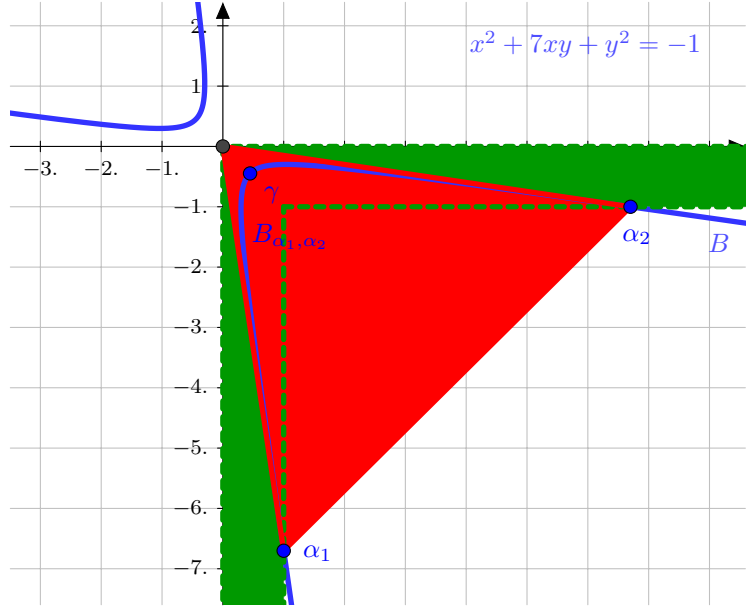
\begin{figure}[h]  
\begin{center}
\pagestyle{empty}

\definecolor{ffqqqq}{rgb}{1.,0.,0.}
\definecolor{qqzzqq}{rgb}{0.,0.6,0.}
\definecolor{ududff}{rgb}{0.30196078431372547,0.30196078431372547,1.}
\definecolor{xdxdff}{rgb}{0.49019607843137253,0.49019607843137253,1.}
\definecolor{uuuuuu}{rgb}{0.26666666666666666,0.26666666666666666,0.26666666666666666}
\definecolor{qqqqff}{rgb}{0.,0.,1.}
\definecolor{ttttff}{rgb}{0.2,0.2,1.}
\definecolor{cqcqcq}{rgb}{0.7529411764705882,0.7529411764705882,0.7529411764705882}
\begin{tikzpicture}[line cap=round,line join=round,>=triangle 45,x=0.8cm,y=0.8cm]
\draw [color=cqcqcq,, xstep=0.8cm,ystep=0.8cm] (-3.4943516949643376,-7.5966201167001035) grid (8.589838948901445,2.3890135412555527);
\draw[->,color=black] (-3.4943516949643376,0.) -- (8.589838948901445,0.);
\foreach \x in {-3.,-2.,-1.,1.,2.,3.,4.,5.,6.,7.,8.}
\draw[shift={(\x,0)},color=black] (0pt,2pt) -- (0pt,-2pt) node[below] {\footnotesize $\x$};
\draw[->,color=black] (0.,-7.5966201167001035) -- (0.,2.3890135412555527);
\foreach \y in {-7.,-6.,-5.,-4.,-3.,-2.,-1.,1.,2.}
\draw[shift={(0,\y)},color=black] (2pt,0pt) -- (-2pt,0pt) node[left] {\footnotesize $\y$};
\draw[color=black] (0pt,-10pt) node[right] {\footnotesize $0$};
\clip(-3.4943516949643376,-7.5966201167001035) rectangle (8.589838948901445,2.3890135412555527);
\fill[line width=2.pt,color=qqzzqq,fill=qqzzqq,fill opacity=0.10000000149011612] (0.,0.) -- (0.,-9.) -- (1.,-9.) -- (1.,-1.) -- (14.,-1.) -- (14.,0.) -- cycle;
\fill[line width=2.pt,color=ffqqqq,fill=ffqqqq,fill opacity=0.10000000149011612] (0.,0.) -- (1.,-6.701562118716424) -- (6.701562118716424,-1.) -- cycle;
\draw [samples=50,domain=-0.99:0.99,rotate around={-45.:(0.,0.)},xshift=0.cm,yshift=0.cm,line width=2.pt,color=ttttff] plot ({0.6324555320336759*(1+(\x)^2)/(1-(\x)^2)},{0.4714045207910317*2*(\x)/(1-(\x)^2)});
\draw [samples=50,domain=-0.99:0.99,rotate around={-45.:(0.,0.)},xshift=0.cm,yshift=0.cm,line width=2.pt,color=ttttff] plot ({0.6324555320336759*(-1-(\x)^2)/(1-(\x)^2)},{0.4714045207910317*(-2)*(\x)/(1-(\x)^2)});
\draw [line width=2.pt,dash pattern=on 3pt off 3pt,color=qqzzqq] (0.,0.)-- (0.,-9.);
\draw [line width=2.pt,color=qqzzqq] (0.,-9.)-- (1.,-9.);
\draw [line width=2.pt,dash pattern=on 3pt off 3pt,color=qqzzqq] (1.,-9.)-- (1.,-1.);
\draw [line width=2.pt,dash pattern=on 3pt off 3pt,color=qqzzqq] (1.,-1.)-- (14.,-1.);
\draw [line width=2.pt,color=qqzzqq] (14.,-1.)-- (14.,0.);
\draw [line width=2.pt,dash pattern=on 3pt off 3pt,color=qqzzqq] (14.,0.)-- (0.,0.);
\draw [line width=2.pt,color=ffqqqq] (0.,0.)-- (1.,-6.701562118716424);
\draw [line width=2.pt,color=ffqqqq] (1.,-6.701562118716424)-- (6.701562118716424,-1.);
\draw [line width=2.pt,color=ffqqqq] (6.701562118716424,-1.)-- (0.,0.);
\begin{scriptsize}
\draw [fill=qqqqff] (0.447213595499958,-0.447213595499958) circle (2.5pt);
\draw[color=qqqqff] (0.75,-0.8) node {\fontsize{10}{0} $ \gamma $};
\draw [fill=uuuuuu] (0.,0.) circle (2.5pt);
\draw [fill=xdxdff] (0.,-9.) circle (2.5pt);
\draw [fill=ududff] (1.,-9.) circle (2.5pt);
\draw[color=qqqqff] (1.05,-1.5) node {\fontsize{10}{0} $B_{\alpha_1,\alpha_2}$};
\draw[color=qqzzqq] (0.5,-6.5) node {\fontsize{10}{0} $G$};
\draw[color=ududff] (8.1,-1.6) node {\fontsize{10}{0} $B$};
\draw[color=ududff] (5.9,1.7) node {\fontsize{10}{0} $x^2+7xy+y^2=-1$};
\draw [fill=ududff] (14.,-1.) circle (2.5pt);
\draw [fill=qqqqff] (6.701562118716424,-1.) circle (2.5pt);
\draw[color=qqqqff] (6.75,-1.5) node {\fontsize{10}{0} $\alpha_2 $};
\draw [fill=qqqqff] (1.,-6.701562118716424) circle (2.5pt);
\draw[color=qqqqff] (1.5,-6.75) node {\fontsize{10}{0} $\alpha_1 $};
\draw[color=ffqqqq] (2.9,-2.6) node {\fontsize{10}{0} $ \tfrac{1}{2} \left\vert \big< \alpha_1 , \alpha_2 \big> \right\vert > 1 $};

\end{scriptsize}
\end{tikzpicture}
\caption{Construction to show that $ S_{-1}\cap\mathbb{Z}[i_7] $ is empty}
\label{proof_units}
\end{center}
\end{figure}
\vspace{5mm}


\begin{example}
We showed in \Cref{theorem_units} that the unit group of $ \mathbb{Z}[i_3] $ is generated by the two elements $ -1+i_3 $ and $ -1 $. In fact, the elements in $ S_{1}\cap\mathbb{Z}[i_3] $ are generated by an even power of $ g \coloneqq -1+i_3 $ whereas the elements in $ S_{-1}\cap\mathbb{Z}[i_3] $ are generated by an odd power of $ -1+i_3 $. Multiplying with $ -1 $ has the effect of a mirror reflection on the origin as we can see in \Cref{units_z3}.
\end{example}

We postponed the proof of uniqueness of $ \varepsilon \in \mathbb{Z} $ in \Cref{coro_pos_sol}. Indeed, the set of units in $ \mathbb{Z}[i_z] $ on one branch with norm equal to one is generated by the unit $ i_z $ if $ z \in \mathbb{N} $ by \Cref{theorem_units}. Hence, if $ \alpha \in \mathbb{Z} $ satisfies the Diophantine equation $ x^2 + zxy+ y^2 = M > 0 $, then all associated solutions $ \varepsilon \alpha $ can be described by $ \pm \alpha i_z^n $ for $ n \in \mathbb{Z} $, i.e. $ \varepsilon \in \left\{ \I_+^{n} \left( 1 \right), -\I_+^{n} \left( 1\right)\right\} $. However, there exists exactly one such $ n \in \mathbb{Z} $ such that either $ \I_+^{n} \left( 1 \right) $ or $ -\I_+^{n} \left( 1 \right) $ (not both) lies in $ B_1 $ and is a positive solution to the Diophantine equation $ x^2 + zxy+ y^2 = M $. This finishes the proof of \Cref{coro_pos_sol}.

\vspace{5mm}
\begin{figure}[h]  
\begin{center}
\pagestyle{empty}

\definecolor{qqqqff}{rgb}{0.,0.,1.}
\definecolor{xfqqff}{rgb}{0.4980392156862745,0.,1.}
\definecolor{cqcqcq}{rgb}{0.7529411764705882,0.7529411764705882,0.7529411764705882}
\begin{tikzpicture}[line cap=round,line join=round,>=triangle 45,x=0.7cm,y=0.7cm]
\draw [color=cqcqcq,, xstep=0.7cm,ystep=0.7cm] (-13.395201486745927,-2.2835759690255717) grid (3.5591822156666373,8.411860438245636);
\draw[->,color=black] (-13.395201486745927,0.) -- (3.5591822156666373,0.);
\foreach \x in {-13.,-12.,-11.,-10.,-9.,-8.,-7.,-6.,-5.,-4.,-3.,-2.,-1.,1.,2.,3.}
\draw[shift={(\x,0)},color=black] (0pt,2pt) -- (0pt,-2pt) node[below] {\footnotesize $\x$};
\draw[->,color=black] (0.,-2.2835759690255717) -- (0.,8.411860438245636);
\foreach \y in {-2.,-1.,1.,2.,3.,4.,5.,6.,7.,8.}
\draw[shift={(0,\y)},color=black] (2pt,0pt) -- (-2pt,0pt) node[left] {\footnotesize $\y$};
\draw[color=black] (0pt,-10pt) node[right] {\footnotesize $0$};
\clip(-13.395201486745927,-2.2835759690255717) rectangle (3.5591822156666373,8.411860438245636);
\draw [samples=50,domain=-0.99:0.99,rotate around={-45.:(0.,0.)},xshift=0.cm,yshift=0.cm,line width=2.pt,color=xfqqff] plot ({1.4142135623730951*(1+(\x)^2)/(1-(\x)^2)},{0.6324555320336759*2*(\x)/(1-(\x)^2)});
\draw [samples=50,domain=-0.99:0.99,rotate around={-45.:(0.,0.)},xshift=0.cm,yshift=0.cm,line width=2.pt,color=xfqqff] plot ({1.4142135623730951*(-1-(\x)^2)/(1-(\x)^2)},{0.6324555320336759*(-2)*(\x)/(1-(\x)^2)});
\draw [samples=50,domain=-0.99:0.99,rotate around={-135.:(0.,0.)},xshift=0.cm,yshift=0.cm,line width=2.pt,color=qqqqff] plot ({0.6324555320336759*(1+(\x)^2)/(1-(\x)^2)},{1.4142135623730951*2*(\x)/(1-(\x)^2)});
\draw [samples=50,domain=-0.99:0.99,rotate around={-135.:(0.,0.)},xshift=0.cm,yshift=0.cm,line width=2.pt,color=qqqqff] plot ({0.6324555320336759*(-1-(\x)^2)/(1-(\x)^2)},{1.4142135623730951*(-2)*(\x)/(1-(\x)^2)});
\begin{scriptsize}
\draw [fill=qqqqff] (1.,0.) circle (2.5pt);
\draw[color=qqqqff] (1.375,0.475) node {\fontsize{10}{0} $ g^0 $};
\draw [fill=qqqqff] (0.,1.) circle (2.5pt);
\draw[color=qqqqff] (0.425,1.425) node {\fontsize{10}{0} $ g^2 $};
\draw [fill=qqqqff] (-3.,8.) circle (2.5pt);
\draw[color=qqqqff] (-2.45,8.) node {\fontsize{10}{0} $ g^6 $};
\draw [fill=qqqqff] (-1.,3.) circle (2.5pt);
\draw[color=qqqqff] (-0.6,3.45) node {\fontsize{10}{0} $ g^4 $};
\draw [fill=xfqqff] (-1.,1.) circle (2.5pt);
\draw[color=xfqqff] (-1.3,1.375) node {\fontsize{10}{0} $ g^{1} $};
\draw [fill=xfqqff] (-1.,2.) circle (2.5pt);
\draw[color=xfqqff] (-1.5,2.3) node {\fontsize{10}{0} $ g^3 $};
\draw [fill=xfqqff] (-2.,5.) circle (2.5pt);
\draw[color=xfqqff] (-2.55,4.85) node {\fontsize{10}{0} $ g^5 $};
\draw [fill=xfqqff] (-2.,1.) circle (2.5pt);
\draw[color=xfqqff] (-2,1.55) node {\fontsize{10}{0} $ g^{-1} $};
\draw [fill=xfqqff] (-5.,2.) circle (2.5pt);
\draw[color=xfqqff] (-4.5,2.5) node {\fontsize{10}{0} $ g^{-3} $};
\draw [fill=xfqqff] (-13.,5.) circle (2.5pt);
\draw[color=xfqqff] (-12.5,5.55) node {\fontsize{10}{0} $ g^{-5} $};
\draw [fill=qqqqff] (-8.,3.) circle (2.5pt);
\draw[color=qqqqff] (-8.25,2.45) node {\fontsize{10}{0} $ -g^{-4} $};
\draw [fill=qqqqff] (3.,-1.) circle (2.5pt);
\draw[color=qqqqff] (3.1,-1.7) node {\fontsize{10}{0} $ g^{-2} $};
\draw [fill=qqqqff] (-1.,0.) circle (2.5pt);
\draw[color=qqqqff] (-0.7,0.4) node {\fontsize{10}{0} $- g^{0} $};
\draw [fill=qqqqff] (0.,-1.) circle (2.5pt);
\draw[color=qqqqff] (-0.6,-1.4) node {\fontsize{10}{0} $ -g^{2} $};
\draw [fill=xfqqff] (1.,-1.) circle (2.5pt);
\draw[color=xfqqff] (0.5,-0.8) node {\fontsize{10}{0} $ -g^{1} $};
\draw [fill=xfqqff] (2.,-1.) circle (2.5pt);
\draw[color=xfqqff] (1.9,-1.5) node {\fontsize{10}{0} $ -g^{-1} $};
\draw [fill=xfqqff] (1.,-2.) circle (2.5pt);
\draw[color=xfqqff] (1.8,-1.95) node {\fontsize{10}{0} $ -g^{-3} $};
\draw[color=qqqqff] (-10.95,-1.7) node {\fontsize{10}{0} $x^2+3xy+y^2 = 1$};
\draw[color=xfqqff] (-10.775,7.325) node {\fontsize{10}{0} $x^2+3xy+y^2 = -1$};

\end{scriptsize}
\end{tikzpicture}
\caption{Units in $ \mathbb{Z}[i_3] $}
\label{units_z3}
\end{center}
\end{figure}
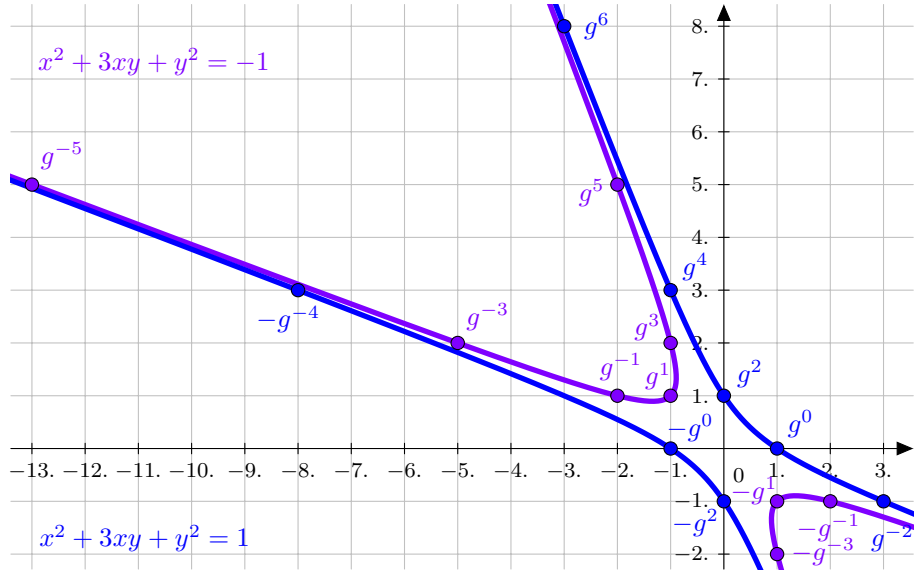
\vspace{5mm}

The next statement is a consequence of the proof of \Cref{theorem_units}.

\begin{corollary} \label{coro_z=3_sol_-1}
	The Diophantine equation $ x^2 + zxy + y^2 = -1 $ can only be solved for $ x,y,z \in \mathbb{Z} $ if and only if $ z \in \left\{ -3,3 \right\} $. Moreover, if $ z \in \left\{ -3,3 \right\} $, then $ x^2 + zxy + y^2 = M $ is solvable if and only if $ x^2 + zxy + y^2 = -M $ is solvable.
\end{corollary}

\begin{proof}
	We showed in the proof of \Cref{theorem_units} that $ S_{-1} = \emptyset $ if and only if $ z \notin \left\{ -3,3 \right\} $. Moreover, $ \mathbb{Z}[i_3] $ and $ \mathbb{Z}[i_{-3}] $ are isomorphic (recall that the isomorphism between them preserves $ \mathbb{Z} $ and changes the sign of the imaginary part). If $ z = \pm 3 $, then $ 1 \mp i_z \in S_{-1} $ is a unit in the corresponding $ z $-ring. Hence, if $ z = \pm 3 $, then $ \alpha \in \mathbb{Z}[i_z] $ solves $ x^2 + zxy + y^2 = M $ if and only if $ \left( 1 \mp i_z \right)\alpha $ solves  $ x^2 + zxy + y^2 = -M $.
\end{proof}

We can ask what happens if $ z \in \mathbb{Z} \setminus \left\{  -3,3 \right\} $ and $ M \in \mathbb{Z} \setminus \left\{ 0 \right\} $. Can we still find solutions to $ x^2 + zxy + y^2 = M $ and $ x^2 + zxy + y^2 = -M $? Sometimes yes, what we will see in the next example. However, it does not work if $ z \in \left\{-2, -1,0,1,2 \right\} $, $ M \in \mathbb{Z} $ is a prime (see \Cref{coro_solvability_pm_p}) or if $ \mathbb{Z}[i_z] $ is a unique factorization domain (see \Cref{primes_if_uniq_factor_domain}).

\begin{example} \label{exa_z39}
	Consider $ 5-i_{39} \in \mathbb{Z}[i_{39}] $, then $ \N \left( 5-i_{39} \right)= -169 $, i.e. the Diophantine equation 
	$$ x^2+ 39xy + y^2 = -169 $$ can be solved. 
Moreover, the Diophantine equation 
	$$ x^2+ 39xy + y^2 = 169 $$
can be solved, too, for example, if $ x= 13 $ and $ y = 0 $.
\end{example}

\subsection{Primes in $ \mathbb{Z} $ with respect to $ \mathbb{Z}[i_z] $}

To deal with the Diophantine equations of the form $ x^2+zxy+ y^2 = M $ we will see that the prime elements in $ \mathbb{Z} $ play an important role. Considering them in $ \mathbb{Z}[i_z] $, we will split them up into the following categories:



\begin{definition}
	Let $ p \in \mathbb{Z} $ be prime. Then we call $ p $ considered as an element of $ \mathbb{Z}[i_{z}] $ {\em regular (element)} if it is irreducible in $ \mathbb{Z}[i_{z}] $. Otherwise we call $ p $ {\em irregular (element)}. If $ p $ is irregular and we can solve the Diophantine equation $ x^2+zxy + y^2 = p $, then we say that $ p $ is {\em of type I}. Otherwise we say that $ p $ is {\em of type II}. If $ p = \alpha \overline{\alpha} $ is irregular for $ \alpha \in \mathbb{Z}[i_z] $ such that $ \alpha $ and $ \overline{\alpha} $ are associated, then we call $ p $ {\em special (element)}.
\end{definition}

Note that special elements are always of type I as they are equal to the norm of their associated irreducible factors. Recall that irreducible and prime elements are not the same in general if the ring we consider is not a unique factorization domain. However, we will see later that all irregular elements are prime elements whereas there are regular elements (so irreducible) which are not prime with respect to the corresponding $ z $-ring (compare with the ring $ \mathbb{Z}[i_{39}] $ we will discuss in \Cref{ex_prime__neq_irreducible}). Also note that $ p \in \mathbb{Z}[i_z] $ is regular/irregular/special if and only if the same holds true for $ p \in \mathbb{Z}[i_{-z}] $ as these rings are isomorphic and the corresponding isomorphism preserves $ \mathbb{Z} $.

\begin{example} \label{example_primes_Gaussian}
	Let us consider the Gaussian integers $ \mathbb{Z}[i] $. Then we know that the positive, regular primes $ p \in \mathbb{Z}[i] $ are of the form $ p \equiv 3 \pmod{4} $ and the positive, irregular primes are either of the form $ p \equiv 1 \pmod{4} $ or $ p = 2 $. Clearly the positive irregular primes $ p $ are of type I and the negative ones of type II as the Diophantine equation $ x^2+y^2 = p $ is not solvable if $ p $ is negative.
	 In fact, $ p = 2 $ is the only special prime in $ \mathbb{Z}[i] $ as we saw in \cite{Miniatur}. Its factors $ 1+i,1-i $ are associated as $ i \in \mathbb{Z}[i] $ is a unit and $ i \left( 1 - i \right) = i + 1 $. Note that $ -2 $ is not special as it cannot be written as a product of two conjugated elements in $ \mathbb{Z}[i_z] $.
\end{example}

\begin{lemma} \label{lemma_ir_and_reg_primes}
	Let $ p \in \mathbb{Z} $ be prime. Then $ p \in \mathbb{Z}[i_z] $ is regular if and only if both Diophantine equations $ x^2+zxy+ y^2 = p $ and $ x^2+zxy+ y^2 = -p $ are not solvable. Furthermore, if $ p \in \mathbb{Z}[i_z] $ is irregular, then either $ p = \alpha \overline{\alpha} = \N \left( \alpha \right) $ or $ p = -\alpha \overline{\alpha} = -\N \left( \alpha \right) $ for some $ \alpha \in \mathbb{Z}[i_z] $.
\end{lemma}

\begin{proof}
	If $ x^2+zxy+ y^2 = p $ or $ x^2+zxy+ y^2 = -p $ is solvable, then either $ p = \left( x+i_zy \right) \overline{\left( x+i_zy \right)} $ or $ -p = \left( x+i_zy \right) \overline{\left( x+i_zy \right)} $, so $ p $ is reducible. Conversely, if $ p $ is reducible, then there exist $ \alpha \in \mathbb{Z}[i_z] $ with $ \alpha \mid p $ and $ \N \left( \alpha \right) \notin \{ \pm 1,\pm p^2\} $. By \Cref{Nz_lemma} we have that $ \N \left( \alpha \right) \mid \N \left( p \right) = p^2 $ and hence $ \N \left( \alpha \right) \in \left\{ -p,p \right\} $ which shows that $ \alpha $ solves the Diophantine equation $ x^2+zxy+ y^2 = M $ for either $ M = p $ or $ M = -p $. Moreover, in this case we have either $ \N \left( \alpha \right) = \alpha \overline{\alpha} = p $ or $ \N \left( \alpha \right) = \alpha \overline{\alpha} = -p $ by \Cref{Nz_lemma}.
\end{proof}



\begin{example} \label{example_z6_pm3_pm7}
Recall \Cref{example_exist_sol} where we showed that $ x^2+6xy+y^2 = 7 $ has no solution. Since $ x^2+6xy+y^2 = -7 $ is solvable by $ x = 4 $ and $ y = -1 $, we clearly have that $ -7, 7 \in \mathbb{Z}[i_6] $ are irregular by \Cref{lemma_ir_and_reg_primes} where $ -7 $ is of type I and $ 7 $ of type II. On the other hand, both equations $ x^2 + 6xy + y^2 = \pm 3 $ are not solvable. That $ x^2 + 6xy + y^2 = 3 $ is not solvable follows by \Cref{coro_pos_sol} (see \Cref{example_pm3_pm7}, there is no intersection of the light blue line and the $ \mathbb{Z} \times \mathbb{Z} $-grid in the first quadrant). Moreover, that $ x^2 + 6xy + y^2 = -3 $ cannot be solved can be seen in the following way: Clearly there is an element $ -2+y_{*}i_6 \in S_{-3} $ (marked with a cross in \Cref{example_pm3_pm7}). By calculation it follows that $ y_{*} \in \{ 6\pm \sqrt{29} \} $. We can choose $ y_{*} = 6- \sqrt{29} $ and calculate
	\begin{align*}
		\I_{+} \left( -2+ \left(6-  \sqrt{29} \right) i_6\right) &= -2i_6+ \left(6-  \sqrt{29} \right) \left(6i_6-1 \right) \\
		&= \sqrt{29}-6 + \left( 34-6 \sqrt{29} \right)i_6
	\end{align*}	
Now we have that 
	$$ \mathrm{Im} \left( \I_{+} \left( -2+y_{*} \right) \right) = 34- 6\sqrt{25}<4.$$
And so it is enough to consider just a part of the branch, namely, we have to check whether the intersection of the $ \mathbb{Z} \times \mathbb{Z} $-grid and the dark blue line within the green part in \Cref{example_pm3_pm7} is empty which is clearly true. Since the considered green part of $ S_{-3} $ contains the subbranch $ B_{-2+y_{*}} $, we get by \Cref{theo_local_branch} that there is no solution to $ x^2 + 6xy + y^2 = -3 $. Therefore $ -3,3 \in \mathbb{Z}[i_6] $ are regular by \Cref{lemma_ir_and_reg_primes}.
\end{example}

\vspace{5mm}
\begin{figure}[h]  
\begin{center}
\pagestyle{empty}

\definecolor{qqqqcc}{rgb}{0.,0.,0.8}
\definecolor{ttffqq}{rgb}{0.2,1.,0.}
\definecolor{qqqqff}{rgb}{0.,0.,1.}
\definecolor{wwccff}{rgb}{0.4,0.8,1.}
\definecolor{ffzztt}{rgb}{1.,0.6,0.2}
\definecolor{ffqqtt}{rgb}{1.,0.,0.2}
\definecolor{cqcqcq}{rgb}{0.7529411764705882,0.7529411764705882,0.7529411764705882}
\begin{tikzpicture}[line cap=round,line join=round,>=triangle 45,x=1.0cm,y=1.0cm]
\draw [color=cqcqcq,, xstep=1.0cm,ystep=1.0cm] (-2.9687499955146395,-1.4939158301683642) grid (4.682373906970594,4.613791698884489);
\draw[->,color=black] (-2.9687499955146395,0.) -- (4.682373906970594,0.);
\foreach \x in {-2.,-1.,1.,2.,3.,4.}
\draw[shift={(\x,0)},color=black] (0pt,2pt) -- (0pt,-2pt) node[below] {\footnotesize $\x$};
\draw[->,color=black] (0.,-1.4939158301683642) -- (0.,4.613791698884489);
\foreach \y in {-1.,1.,2.,3.,4.}
\draw[shift={(0,\y)},color=black] (2pt,0pt) -- (-2pt,0pt) node[left] {\footnotesize $\y$};
\draw[color=black] (0pt,-10pt) node[right] {\footnotesize $0$};
\clip(-2.9687499955146395,-1.4939158301683642) rectangle (4.682373906970594,4.613791698884489);
\fill[line width=2.pt,color=ttffqq,fill=ttffqq,fill opacity=0.10000000149011612] (0.,4.) -- (-2.,4.) -- (-2.,0.) -- (0.,0.) -- cycle;
\draw [samples=50,domain=-0.99:0.99,rotate around={-135.:(0.,0.)},xshift=0.cm,yshift=0.cm,line width=2.pt,color=ffqqtt] plot ({1.3228756555322954*(1+(\x)^2)/(1-(\x)^2)},{1.8708286933869707*2*(\x)/(1-(\x)^2)});
\draw [samples=50,domain=-0.99:0.99,rotate around={-135.:(0.,0.)},xshift=0.cm,yshift=0.cm,line width=2.pt,color=ffqqtt] plot ({1.3228756555322954*(-1-(\x)^2)/(1-(\x)^2)},{1.8708286933869707*(-2)*(\x)/(1-(\x)^2)});
\draw [samples=50,domain=-0.99:0.99,rotate around={-45.:(0.,0.)},xshift=0.cm,yshift=0.cm,line width=2.pt,color=ffzztt] plot ({1.8708286933869707*(1+(\x)^2)/(1-(\x)^2)},{1.3228756555322954*2*(\x)/(1-(\x)^2)});
\draw [samples=50,domain=-0.99:0.99,rotate around={-45.:(0.,0.)},xshift=0.cm,yshift=0.cm,line width=2.pt,color=ffzztt] plot ({1.8708286933869707*(-1-(\x)^2)/(1-(\x)^2)},{1.3228756555322954*(-2)*(\x)/(1-(\x)^2)});
\draw [samples=50,domain=-0.99:0.99,rotate around={-135.:(0.,0.)},xshift=0.cm,yshift=0.cm,line width=2.pt,color=wwccff] plot ({0.8660254037844386*(1+(\x)^2)/(1-(\x)^2)},{1.224744871391589*2*(\x)/(1-(\x)^2)});
\draw [samples=50,domain=-0.99:0.99,rotate around={-135.:(0.,0.)},xshift=0.cm,yshift=0.cm,line width=2.pt,color=wwccff] plot ({0.8660254037844386*(-1-(\x)^2)/(1-(\x)^2)},{1.224744871391589*(-2)*(\x)/(1-(\x)^2)});
\draw [samples=50,domain=-0.99:0.99,rotate around={-45.:(0.,0.)},xshift=0.cm,yshift=0.cm,line width=2.pt,color=qqqqff] plot ({1.224744871391589*(1+(\x)^2)/(1-(\x)^2)},{0.8660254037844386*2*(\x)/(1-(\x)^2)});
\draw [samples=50,domain=-0.99:0.99,rotate around={-45.:(0.,0.)},xshift=0.cm,yshift=0.cm,line width=2.pt,color=qqqqff] plot ({1.224744871391589*(-1-(\x)^2)/(1-(\x)^2)},{0.8660254037844386*(-2)*(\x)/(1-(\x)^2)});
\begin{scriptsize}
\draw [fill=ffzztt] (-2.,1.) circle (2.5pt);
\draw [fill=ffzztt] (-1.,2.) circle (2.5pt);
\draw [fill=ffzztt] (2.,-1.) circle (2.5pt);
\draw [fill=ffzztt] (4.,-1.) circle (2.5pt);
\draw [fill=ffzztt] (-1.,4.) circle (2.5pt);
\draw [color=qqqqff] 
(-2,0.6148300063564661)-- ++(-2.5pt,-2.5pt) -- ++(5.0pt,5.0pt) ++(-5.0pt,0) -- ++(5.0pt,-5.0pt);
\draw[color=ffqqtt] (2.4,3.75) node {\fontsize{10}{0} $x^2+6xy+y^2 = 7$};
\draw[color=ffzztt] (2.55,3.25) node {\fontsize{10}{0} $x^2+6xy+y^2 = -7$};
\draw[color=wwccff] (2.425,2.75) node {\fontsize{10}{0} $x^2+6xy+y^2 = 3$};
\draw[color=qqqqcc] (2.575,2.25) node {\fontsize{10}{0} $x^2+6xy+y^2 = -3$};


\end{scriptsize}
\end{tikzpicture}
\caption{Some level sets in $ \mathbb{R}[i_6] $}
\label{example_pm3_pm7}
\end{center}
\end{figure}
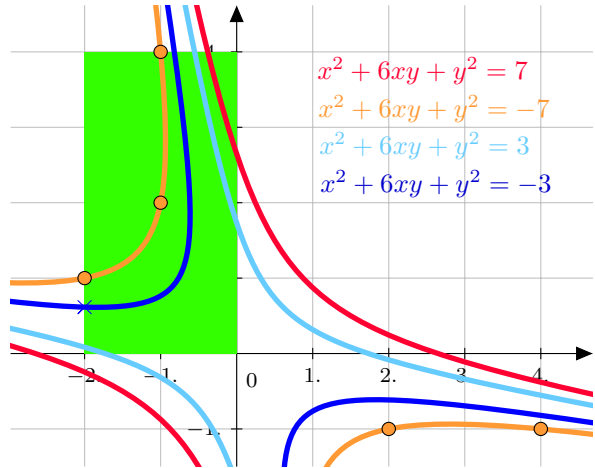
\vspace{5mm}

\begin{example} \label{example_primes_z123}
	Let $ z \in \left\{ -3,3 \right\} $, then we have that an element $ p \in \mathbb{Z}[i_z] $ which is prime in $ \mathbb{Z} $ is of type I if and only if $ -p $ is also of type I. This is a consequence of the fact that $ \mathbb{Z}[i_z] $ contains elements with norm $ -1 $ (compare with the proof of \Cref{coro_z=3_sol_-1}). Conversely, the inverse statement also holds true, i.e. the existence of elements being prime in $ \mathbb{Z} $ such that $ \pm p $ are of type I is true if and only if $ z \in \left\{ -3,3 \right\} $, see \Cref{coro_solvability_pm_p} in the next section. 
For $ z \in \{ 0, \pm 1 \pm 2 \} $ it is obvious that there are no primes $ \pm p \in \mathbb{Z} $ both of type I since $ x^2+zxy+y^2 = M $ has no solution if $ M < 0 $ as mentioned \Cref{ex_non_neg_sol}. In particular, for $ z \in \left\{ -2,2 \right\} $ there exist no irregular elements in $ \mathbb{Z}[i_z] $ as a prime $ p \in \mathbb{Z} $ is never a square and so $ p $ cannot be represented by $ x^2 \pm 2xy + y^2 = \left( x \pm y \right)^2 $.
\end{example}







In the next few sections we would like to find out more about the rings $ \mathbb{Z}[i_z] $ i.e. which one of them are no unique factorization domains, about properties of their regular, irregular (both types), special and non-special primes and the connection to the Diophantine equation $ x^2 + zxy + y^2 = M $.

\subsection{The irregular elements in $ \mathbb{Z}[i_z] $}


Recall that the irregular elements can be factorized as stated in \Cref{lemma_ir_and_reg_primes}. Moreover, it is clear that these two factors are irreducible because their norm is equal to a prime in $ \mathbb{Z} $ and \Cref{Nz_lemma}. However, it is a priori not clear that they are also prime in the corresponding $ z $-ring. The goal of this section will be to show this. At first we start with a weaker form of the above statement. For proving both statements we use the following lemma:

\begin{lemma} \label{lem_trinity_divisability}
	Let $ p \in \mathbb{Z} $ be prime, $ \alpha \in \mathbb{Z}[i_z] $ such that $ p $ divides (in $ \mathbb{Z} $) two of the following terms:
	$$ \mathrm{Re} \left( \alpha \right), \mathrm{Im} \left( \alpha \right), \N \left( \alpha \right) $$
Then $ p $ divides $ \alpha $ (in $ \mathbb{Z}[i_z] $). Moreover, if $ p $ divides $ \mathrm{Re} \left( \alpha \right) $ and $ \mathrm{Im} \left( \alpha \right) $ (in $ \mathbb{Z} $), then $ p^2 $ divides $ \N \left( \alpha \right) $ (in $ \mathbb{Z} $). Conversely, if $ n \in \mathbb{Z} $ divides $ \alpha $ (in $ \mathbb{Z}[i_z] $), then $ n $ divides $ \mathrm{Re} \left( \alpha \right) $ and $ \mathrm{Im} \left( \alpha \right) $ (in $ \mathbb{Z} $).
\end{lemma}

\begin{proof}
	We have
		$$ \mathrm{Re} \left( \alpha \right)^2 + z\mathrm{Re} \left( \alpha \right)\mathrm{Im} \left( \alpha \right) + \mathrm{Im} \left( \alpha \right)^2 = \N \left( \alpha \right) $$
and hence we see that $ p $ dividing two of the terms $ \mathrm{Re} \left( \alpha \right), \mathrm{Im} \left( \alpha \right), \N \left( \alpha \right) $ also implies that it divides all of them. Additionally, we have that $ p^2 $ divides $ \N \left( \alpha \right) $ if $ p $ divides $ \mathrm{Re} \left( \alpha \right) $ and $ \mathrm{Im} \left( \alpha \right) $ because each summand on the left-hand side of the equation consists of terms divisible by $ p^2 $. Moreover, if $ p \mid \mathrm{Re} \left( \alpha \right) $ and $ p \mid \mathrm{Im} \left( \alpha \right) $, then we find $ a_1,a_2 \in \mathbb{R} $ such that $ \mathrm{Re} \left( \alpha \right) = a_1p $ and $ \mathrm{Im} \left( \alpha \right) = a_2p $. Thus, we have
	$$ \alpha = p \left( a_1 + a_2i_z \right) $$
where the product here is the $ z $-product.	
Conversely, if $ n \in \mathbb{Z} \subseteq \mathbb{Z}[i_z] $ divides $ \alpha $, then there is $ b_1 + b_2i_z \in \mathbb{Z}[i_z] $ such that 
	$$ n \left( b_1 + b_2i_z \right) = \alpha .$$
On the other hand, we have 
	$$ n \left( b_1 + b_2i_z \right) = nb_1 + nb_2i_z $$
and so $ \mathrm{Re} \left( \alpha \right) = nb_1 $ and $ \mathrm{Im} \left( \alpha \right) = nb_2 $ which shows that the real and the imaginary part of $ \alpha $ are divisible by $ n $.
\end{proof}

\begin{proposition} \label{prop_norm_prime_argument}
	Let $ p \in \mathbb{Z} $ be irregular, $ p = \alpha \overline{\alpha} $ or $ p = -\alpha \overline{\alpha} $ with  $ \alpha \in \mathbb{Z}[i_z] $ and $ \beta \in \mathbb{Z}[i_z] $ with $ p \mid \N \left( \beta \right) $. Then either $ \alpha \mid \beta $ or $ \overline{\alpha} \mid \beta $. 
\end{proposition}

\begin{proof}
	Assume $ \alpha = a_1 + a_2i_z $, $ \beta = b_1 + b_2 i_z $ with $ \N \left( \beta \right) = pM $ for $ M \in \mathbb{Z} $, then we have:
	\begin{align*}
		\alpha \beta &= \left( a_1+a_2i_z \right)_{p} \left( b_1+b_2i_z \right)_{pM} = \left( a_1b_1-a_2b_2 + \mathrm{Im} \left( \alpha \beta \right)i_z\right)_{p^2M} \\
		\widetilde{\alpha} \beta &= \left( a_2+a_1i_z \right)_{p} \left( b_1+b_2i_z \right)_{pM} = \left( a_2b_1-a_1b_2 + \mathrm{Im} \left( \widetilde{\alpha} \beta \right) i_z\right)_{p^2M}
	\end{align*}
At first we would like to show that one of the real parts of the above products is divisible by $ p $:
	\begin{align*}
		\mathrm{Re}\left( \alpha \beta \right) \mathrm{Re}\left( \widetilde{\alpha} \beta \right) 
		&= \left( a_1b_1-a_2b_2 \right) \left( a_2b_1-a_1b_2 \right) \\
		&\equiv a_1a_2b_1^2 - a_1^2b_1b_2 - a_2^2b_1b_2 + a_1a_2b_2^2 \pmod{p} \\
		&\equiv a_1 a_2 \left( b_1^2 + b_2^2 \right) - \left( a_1^2 + a_2^2 \right)b_1b_2 \pmod{p} \\
		&\equiv a_1 a_2 \left( b_1^2 +zb_1b_2 + b_2^2 \right) - \left( a_1^2 + za_1a_2 + a_2^2 \right)b_1b_2 \pmod{p} \\
		&\equiv a_1 a_2 \N \left( \beta \right)- b_1b_2 \N \left( \alpha \right) \pmod{p} \\
		&\equiv 0 \pmod{p}
	\end{align*}
where the last step follows because $ \N \left( \alpha \right) $ and $ \N \left( \beta \right) $ are divisible by $ p $. Therefore either $ \mathrm{Re}\left( \alpha \beta \right) $ or $ \mathrm{Re}\left( \widetilde{\alpha} \beta \right) $ is divisible by $ p $. Since $ p \mid \N \left( \alpha \beta \right) = p \N \left( \beta \right) $ and $ p \mid \N \left( \widetilde{\alpha} \beta \right) = p \N \left( \beta \right) $ we deduce that either $ p \mid \alpha \beta $ or $ p \mid \widetilde{\alpha} \beta $ by \Cref{lem_trinity_divisability}. Observe that $ p \mid \widetilde{\alpha} \beta $ and $ p \mid \overline{\alpha} \beta $ is equivalent because $ \overline{\alpha} $ and $ \widetilde{\alpha} $ are associated by \Cref{lemma_mirror}. Hence, if $ p \mid \alpha \beta $, then 
	$$ \tfrac{\beta}{\overline{\alpha}} = \tfrac{\alpha\beta}{\alpha\overline{\alpha}} = \tfrac{\alpha\beta}{p} \in \mathbb{Z}[i_z] $$
and if $ p \mid \widetilde{\alpha} \beta $, we have
	$$ \tfrac{\beta}{\alpha} = \tfrac{\overline{\alpha}\beta}{\overline{\alpha}\alpha} = \tfrac{\overline{\alpha}\beta}{p} \in \mathbb{Z}[i_z] .$$
	\end{proof}

With the last proposition we can conclude a fact which we already mentioned before:

\begin{corollary} \label{coro_solvability_pm_p}
Let $ p \in \mathbb{Z} $ be prime and irregular in $ \mathbb{Z}[i_z] $. Then $ p $ and $ -p $ are of type I if and only if $ z \in \{ -3,3 \} $.
\end{corollary}

\begin{proof}
	If both $ \pm p $ are of type I we can find $ \alpha_1, \alpha_2 \in \mathbb{Z}[i_z] $ such that $ p = \alpha_1 \overline{\alpha_1} = -\alpha_2 \overline{\alpha_2} $ where $ \N \left( \alpha_1\right) = p = -\N \left( \alpha_2\right) $. Without loss of generality, we can assume that $ \alpha_1 \mid \alpha_2 $ by \Cref{prop_norm_prime_argument}. I.e.we find $ \varepsilon \in \mathbb{Z}[i_z] $ such that $ \alpha_2 = \varepsilon \alpha_1 $. Then we have
	$$ -p= \N \left( \alpha_2\right) = \N \left( \varepsilon\right) \N \left( \alpha_1\right) = \N \left( \varepsilon\right)p $$
and so we deduce that $ \N \left( \varepsilon \right) = -1 $. Thus $ z \in \{ -3,3 \} $ by \Cref{coro_z=3_sol_-1}.

The converse statement is a consequence of \Cref{coro_z=3_sol_-1} because if $ p \in \mathbb{Z}[i_z] $ is irregular, then either $ x^2+zxy+y^2 = p $ or $ x^2+zxy+y^2 = -p $ can be solved by \Cref{lemma_ir_and_reg_primes} and hence both Diophantine equations. Therefore $ -p,p \in \mathbb{Z}[i_z] $ are both of type I.
\end{proof}





We will come now to the main theorem of this section:

\begin{theorem} \label{theo_prime_type_I}
The irreducible factors of irregular elements in $ \mathbb{Z}[i_z] $ are prime elements.
\end{theorem} 

\begin{proof}
	Let $ p = \alpha \overline{\alpha} $ be an irregular element in $ \mathbb{Z}[i_z] $ with irreducible factors $ \alpha, \overline{\alpha} \in \mathbb{Z}[i_z] $. Let $ \beta = b_1+b_2i_z \in \mathbb{Z}[i_z] $ and $ \gamma = c_1+c_2i_z \in \mathbb{Z}[i_z] $ with $ \alpha \mid \beta \gamma $. We show that either $ \alpha \mid \beta $ or $ \alpha \mid \gamma $.
	
At first we calculate
	\begin{align*}
		\widetilde{\alpha} \beta \gamma &= \left( a_2+a_1 i_z \right)  \left( b_1+b_2 i_z \right) \left( c_1+a_2 i_z \right) \\
		&= \left( a_2+a_1 i_z \right) \big( b_1c_1-b_2c_2 + \left( b_1c_2+b_2c_1+zb_2c_2 \right)i_z \big) \\
		&= a_2 \left( b_1 c_1 - b_2 c_2 \right) -a_1 \left( b_1 c_2 + b_2 c_1 + z b_2 c_2 \right) +\mathrm{Im} \left(  \widetilde{\alpha} \beta \gamma \right)i_z 
	. \end{align*}
Since $ \alpha \mid \beta \gamma $ we also have that $ p = \overline{\alpha} \alpha \mid \overline{\alpha} \beta \gamma $	 and hence $ p \mid i_z\overline{\alpha} \beta \gamma = \widetilde{\alpha} \beta \gamma$ which implies $ p \mid \mathrm{Re} \left(\widetilde{\alpha} \beta \gamma \right) $ by \Cref{lem_trinity_divisability}. Additionally, we also have that $ \N \left( \alpha \right) = a_1^2+za_1a_2+a_2^2 = p $ and so $ a_1^2 \equiv -\left( za_1a_2+a_2^2 \right)\pmod{p} $. With this we get:
	\begin{align*}
		\mathrm{Re}\left( \widetilde{\alpha}\beta \right) \mathrm{Re}\left( \widetilde{\alpha}\gamma \right) &= \left( a_2b_1 - a_1b_2 \right) \left( a_2c_1 - a_1c_2 \right) \\
		&\equiv a_2^2b_1c_1+a_1^2b_2c_2-a_1a_2 \left( b_1c_2 + b_2c_1 \right) \pmod{p} \\
		&\equiv a_2^2b_1c_1-\left( za_1a_2 + a_2^2 \right) b_2c_2-a_1a_2 \left( b_1c_2 + b_2c_1 \right) \pmod{p} \\
		&\equiv a_2 \big( a_2 \left( b_1c_1-b_2c_2 \right)-a_1\left(b_1c_2+b_2c_1+zb_2c_2 \right) \big) \pmod{p} \\
		&\equiv a_2 \mathrm{Re}\left( \widetilde{\alpha} \beta \gamma \right) \pmod{p} \\
		&\equiv 0
	\end{align*}
Since $ p \mid p \N \left( \beta \right) = \N \left( \widetilde{\alpha} \beta \right) $, $ p \mid p \N \left( \gamma \right) = \N \left( \widetilde{\alpha} \gamma \right) $ and either $ p \mid \mathrm{Re}\left( \widetilde{\alpha}\beta \right) $ or $ p \mid \mathrm{Re}\left( \widetilde{\alpha}\gamma \right) $ we deduce that either $ p \mid \widetilde{\alpha}\beta $ or $ p \mid \widetilde{\alpha}\gamma $ again by \Cref{lem_trinity_divisability}. Dividing by $ \overline{\alpha} $ implies that either $ \alpha \mid i_z\beta $ or $ \alpha \mid i_z\gamma $ which is equivalent to $ \alpha \mid \beta $ or $ \alpha \mid \gamma $.
\end{proof}

With \Cref{theo_prime_type_I} we can easily prove \Cref{prop_norm_prime_argument}: Indeed, since $ \alpha \mid p $ and $ p \mid \N \left( \beta \right)= \beta \overline{\beta} $, we conclude that $ \alpha \mid \beta $ or $ \alpha \mid \overline{\beta} $ where the latter is equivalent to $ \overline{\alpha} \mid \beta $.  


If we assume $ \mathbb{Z}[i_z] $ to be a unique factorization domain, then we can conclude the follwoing statement by using \Cref{theo_prime_type_I}:

\begin{corollary} \label{primes_if_uniq_factor_domain}
	Let $ \mathbb{Z}[i_z] $ be a unique factorization domain. 
If $ M \in \mathbb{Z} $ and $ z \notin \left\{ -3,3 \right\} $, then at most one of the Diophantine equations $ x^2 + zxy + y^2 = M $ and $ x^2 + zxy + y^2 = -M $ can be solved.
\end{corollary}


\begin{proof}
	Assume $ \alpha, \beta \in \mathbb{Z}[i_z] $ such that $ \alpha $ solves $ x^2 + z x y + y^2 = M $ and $ \beta $ solves $ x^2 + z x y + y^2 = -M $. Let $ M $ be divisible by a prime $ p \in \mathbb{Z} $. Then $ p \in \mathbb{Z}[i_z] $ can either be regular or irregular. In the following we will discuss both cases.

Assume at first that $ p $ is regular. Then $ p $ is irreducible and also prime in $ \mathbb{Z}[i_z] $ because $ \mathbb{Z}[i_z] $ is a unique factorization domain. Since $ p \mid M = \alpha \overline{\alpha}$ this means that $ p \mid \alpha $ or $ p \mid \overline{\alpha} $. If $ p \mid \overline{\alpha} $, then $ p = \overline{p} \mid \overline{\overline{\alpha}} = \alpha $ and so always $ p \mid \alpha $. Moreover, $ p \mid \beta $ holds by the same arguments. Define $ \alpha' \coloneqq \tfrac{\alpha}{p} $, $ \beta' \coloneqq \tfrac{\beta}{p} $ and $ M' \coloneqq \tfrac{M}{p^2} $. 

In case $ p $ is irregular, then $ p = \gamma \overline{\gamma} $ or $ p = -\gamma \overline{\gamma} $ for $ \gamma \in \mathbb{Z}[i_z] $ prime. Then we have $ \gamma \mid \alpha $ or $ \gamma \mid \overline{\alpha} $ and also $ \gamma \mid \beta $ or $ \gamma \mid \overline{\beta} $. Without loss of generality, we can assume that $ \gamma \mid \alpha $ and $ \gamma \mid \beta $. Hence, we can set $ \alpha' = \tfrac{\alpha}{\gamma} $, $ \beta' = \tfrac{\beta}{\gamma} $ and $ M' = \tfrac{M}{p} $. 

In both cases we see that $ \alpha' $ solves $ x^2 + z x y + y^2 = M' $ and $ \beta' $ solves $ x^2 + z x y + y^2 = -M' $. We can iterate this process for all prime factors of $ M' $ until we come to $ M' \in \{ \pm 1 \} $. However, we know by \Cref{coro_z=3_sol_-1} that this is only possible if $ z \in \left\{ -3,3 \right\} $. 
\end{proof}

The question whether a given $ z $-ring is a unique factorization domains or not not is not easy to answer in general. We will see later that most of them cannot be unique factorization domains. For example, if $ z = 3,5,9,21 $, then $ \mathbb{Z}[i_z] $ is a unique factorization domain. However, for $ z = 7,11,13,15 $ it is not.

\end{section}

\subsection{Special elements in $ \mathbb{Z}[i_z] $}	

Sometimes it happens that elements in $ \mathbb{Z}[i_z] $ and their conjugates are associated. For example, this holds true for $ 1 + i_z $ and $ 1 - i_z $:
	\begin{align*}
		\overline{\left( 1+i_z \right)} i_z &= \left( 1+z-i_z \right) i_z = 1 + i_z \\
		\overline{\left( 1-i_z \right)} \left(-i_z\right) &= \left( 1-z+i_z \right) \left(-i_z\right) = 1 - i_z.
	\end{align*}
To construct the positive, primitive solutions of a Diophantine equation $ x^2 + zxy + y^2 = M $ for $ M \in \mathbb{Z} $ being a product of primes in $ \mathbb{Z} $ of type I, we will distinguish whether these primes have associated factors, i.e if they are special or not. The goal of this section is to characterize the special elements in $ \mathbb{Z}[i_z] $. For this we would like to prove the following statement:

\begin{theorem}[Characterization of Special Primes] \label{theorem_special_primes}
Let $ z \in \mathbb{Z} \setminus \{ \pm 3,\pm 4 \} $. Then $ \mathbb{Z}[i_z] $ can have at most two special elements of the form $ 2 \pm z \in \mathbb{Z}[i_z] $. Each of them is special in $ \mathbb{Z}[i_z] $ if and only if it is prime in $ \mathbb{Z} $. Otherwise the special elements are 
	\begin{itemize}
		\item $ \pm 5 \in \mathbb{Z}[i_z] $ if $ z = \pm 3 $
		\item $ -2,-3 \in \mathbb{Z}[i_z] $ if $ z = \pm 4 $.	
	\end{itemize}	
\end{theorem}

The following statements will be needed to finally prove \Cref{theorem_special_primes}:

\begin{lemma} \label{lemma_associate}
	Let $ p = \alpha \overline{\alpha} \in \mathbb{Z}[i_z] $ be irregular (i.e. of type I). Then $ p $ is special if and only if $ p \mid 2-z $ or $ p \mid 2+z $.
\end{lemma}

\begin{proof}
	Let $ \alpha = a + bi_z \in \mathbb{Z}[i_z] $. Then $ p = a^2 + zab + b^2 $ and therefore we have that $ p \nmid a $ and $ p \nmid b $ (otherwise we have that $ p $ divides $ a $ and $ b $, so $ p^2 $ divides $ p $ which is a contradiction). 
	
	We will first assume that $ p $ is special, i.e. $ \alpha $ and $ \overline{\alpha} $ are associated. Hence, we can find a unit $ \varepsilon \in \mathbb{Z}[i_z] $ such that $ \alpha =  \overline{\alpha}\varepsilon $. And so we conclude $ \alpha^2 = \alpha \overline{\alpha} \varepsilon = p \varepsilon $, i.e. 
	$$ \varepsilon = \tfrac{1}{p} \alpha^2 = \tfrac{1}{p} \left( a^2-b^2 + \left( 2ab + zb^2 \right)i_z \right) \in \mathbb{Z}[i_z]$$
	This means $ p \mid a^2-b^2 $ and $ p \mid 2ab + zb^2 = b \left( 2a + zb\right) $ by \Cref{lem_trinity_divisability}. Therefore we have 
	$$ p \mid z \left( a^2-b^2 \right) + \left( 2ab + zb^2  \right) = a \left( za + 2b \right) .$$
	From the above we deduce $ p \mid 2a + zb $ and $ p \mid za + 2b $ because $ p \nmid a $ and $ p \nmid b $. Thus, $ p $ also divides linear combinations of terms divisible by $ p $, namely
	$$ p \mid 2 \left( 2a+zb \right) - z\left( za + 2b \right) = a \left( 4-z^2 \right) = a \left( 2-z \right) \left( 2 + z \right) $$
and therefore we get $ p \mid 2-z $ or $ p \mid 2 + z $.	

On the other hand, let us assume that $ p \mid 2 + z $ or $ p \mid 2 - z $ what we will denote by $ p \mid 2 \pm z $ to show both cases in one. Then we also have that $ p \mid \left( z \pm 2 \right)ab $ and so we get that
	$$ p \mid p - \left( 2 \pm z \right)ab = a^2 + zab + b^2 -\left( z \pm 2 \right)ab .$$
Hence, $ p \mid \left( a \mp b \right)^2 $ and so $ p \mid a-b $ or $ p \mid a+b $. Thus, in both cases we have $ p \mid a^2-b^2 $. Moreover, we also have  that 
	$$ p \mid a^2 + zab + b^2 - \left( a^2-b^2 \right) = zab + 2b^2 $$
and so we can conclude that 
$$ \tfrac{1}{p}\alpha^2 = \tfrac{1}{p} \left( a^2-b^2 + \left( 2ab + zb^2 \right)i_z \right) \in \mathbb{Z}[i_z] .$$
Since $ \N(\tfrac{1}{p}\alpha^2) = \tfrac{1}{p^2}\N(\alpha^2) = \tfrac{1}{p^2}\N(\alpha)^2 = 1 $
we observe that $ \varepsilon \coloneqq \tfrac{1}{p}\alpha^2 \in \mathbb{Z}[i_z] $ is a unit and
	$$ \varepsilon \overline{\alpha} = \tfrac{1}{p}\alpha^2\overline{\alpha} = \alpha $$
holds which shows that $ \alpha $ and $ \alpha $ are associated.
\end{proof}


Now we would like to show that apart from some few exceptions all special elements are of the form $ 2 \pm z $. For this we need the following technical lemma which will also give us some information about a range in $ \mathbb{Z} $ where we can only find regular elements in the corresponding $ z $-ring.

\begin{lemma} \label{lem_represented_M}
	If $ z,M \in \mathbb{Z} $ with $ 2-\vert z \vert < M < 2+ \vert z \vert $, then $ M $ is represented by $ x^2+zxy+y^2 $ if and only if $ \sqrt{M} \in \mathbb{N} $. Moreover, primes $ p \in \mathbb{Z} $ with $ 2-\vert z \vert < p < \vert z \vert + 2 $ are not of type I in $ \mathbb{Z}[i_z] $ and if $ \vert p \vert < \vert z \vert -2 $, then $ p \in \mathbb{Z}[i_z] $ is regular.
\end{lemma}

\begin{proof}

We only need to show the statement for $ z \geq 0 $ as the isomorphism between the $ z $- and $ \left(-z\right) $-ring preserves $ \mathbb{Z} $. 

	If $ M $ is a square in $ \mathbb{Z} $, then we can set $ x = \sqrt{M} $ and $ y=0 $ and we see that $ x^2+zxy+y^2 = M $. Now assume that $ M \in \mathbb{Z} $ is not a square and represented by $ x^2+zxy+y^2 $ with $ 2-z < M < 2+z $. Then we find $ a,b \in \mathbb{Z} $ such that $ a^2 + zab + b^2 = M $. We will consider now the cases when $ M $ is positive or negative separately.
	
Let $ M > 0 $. Then by \Cref{coro_pos_sol} we can assume that $ a $ and $ b $ are non-negative. Moreover, neither $ a = 0 $ nor $ b = 0 $ as otherwise $ M $ would be a square. Therefore, we have $ a,b \geq 1 $ and therefore we get the contradiction 
	$$ 2 + z > M = a^2 + zab + b^2 \geq 1^2+ z + 1^2 = 2 + z .$$
	
It remains to discuss the case $ M < 0 $. In this case $ z \geq 4 $ holds. We need to show that $ M $ is not represented by $ x^2 + zxy + y^2 $. For this we consider the branch $ B \subseteq S_{M} $ in the fourth quadrant. The idea is to show that a closed branch in $ S_M $ contains no solution to $ x^2+zxy+y^2 = M $. Therefore consider 
$$ G \coloneqq S_{M} \cap \left( \left(0,1 \right) \times \left(-\infty,0 \right)i_z \cup \left(0,\infty \right) \times \left( -1,0 \right)i_z \right) .$$
Clearly $ G $ is connected and 
$ G \cap \mathbb{Z}[i_z] = \emptyset $. Now would like to find a connected part of a branch lying entirely in $ G $. For this let $ \epsilon > 0 $ be small enough. We will show now the existence of elements on $ B $ which we can use to define a closed branch on it. Let one of these elements have real part $ 1- \epsilon $ and the other one imaginary part $ \epsilon -1 $. We will denote the corresponding elements  by $ \alpha_1 $ and $ \alpha_2 $, respectively. Let us determine them such that they lie on $ B $. Clearly $ \alpha_1 $ has to satisfy the equation 
	$$ \left( 1- \epsilon \right)^2 + \left( 1- \epsilon \right)zy + y^2 = M .$$
Therefore we get
	$$ y_{\pm} = \frac{-\left( 1- \epsilon \right)z \pm \sqrt{\left( 1- \epsilon \right)^2z^2-4\left( \left(1-\epsilon \right)^2-M\right)}}{2} .$$ 	
Analogously, $ \alpha_2 $ satisfies the equation
	$$ x^2 -\left( 1- \epsilon \right)zx + \left( 1- \epsilon \right)^2 = M $$
and so we deduce	
	$$ x_{\pm} = \frac{\left( 1- \epsilon \right)z \pm \sqrt{\left( 1- \epsilon \right)^2z^2-4\left( \left(1-\epsilon \right)^2-M\right)}}{2} .$$ 
	
We have to make sure that term under the square root is positive. Observe that the condition of the lemma implies $ 3-z \leq M $. If $ \epsilon \leq \tfrac{1}{z} $, then we get
	\begin{align*}
		\left( 1- \epsilon \right)^2z^2-4\left( \left(1-\epsilon \right)^2-M\right) &= \left( 1- \epsilon \right)^2\left(z^2-4\right)+4M \\
		&\geq \left( 1- \epsilon \right)^2\left(z^2-4\right) + 4\left( 3-z \right) \\ 
		&= \left( 1-\epsilon \right)^2z^2-4z + 8 + \underbrace{4\left( 1- \left( 1- \epsilon\right)^2\right)}_{>0}    \\
		&> \left( 1-\tfrac{1}{z} \right)^2z^2-4z + 8 \\
		&= z^2-6z+9 \\
		&= \left(z-3 \right)^2\\
		&\geq 0.
	\end{align*}
Now we can define $ \alpha_1 \coloneqq 1-\epsilon +y_{-}i_z \in \mathbb{R}[i_z] $ and $ \alpha_2 \coloneqq x_{+}- \left(1-\epsilon\right)i_z \in \mathbb{R}[i_z] $. Then $ 1-\epsilon +y_{+}i_z,x_{-} - \left(1-\epsilon\right)i_z \in B_{\alpha_1,\alpha_2} \subseteq B $ and since $ B $ is concave we clearly get that $ B_{\alpha_1,\alpha_2} \subseteq G $ which implies $ B_{\alpha_1,\alpha_2} \cap \mathbb{Z}[i_z] = \emptyset $ (compare with \Cref{proof_lem_z-2_M_0} where $ z = 5 $, $ M = -2 $ and $ \epsilon = \tfrac{1}{5} $).
	
We would like to estimate the absolute value of the oriented area defined by $ \alpha_1 $ and $ \alpha_2 $. Observe that for non-negative $ a,b \in \mathbb{R} $ the following inequality holds true
	$$ \left( a+b \right)^2 \geq \left( a+ b \right)\left( a - b \right) = a^2-b^2 .$$
With this inequality we get	
	\begin{align*}
		\left\vert \big< \alpha_1,\alpha_2 \big> \right\vert &= \left\vert -\left( 1-\epsilon \right)^2 -x_{+}y_{-} \right\vert \\
		&\geq -\left( 1-\epsilon \right)^2 + \frac{1}{4}\left( \left( 1- \epsilon \right)z + \sqrt{\left( 1- \epsilon \right)^2z^2-4\left( \left(1-\epsilon \right)^2-M\right)}\right)^2 \\	
		&\geq -\left( 1-\epsilon \right)^2 + \frac{1}{4} \left( \left( 1-\epsilon \right)^2z^2 - \left( \left( 1-\epsilon \right)^2z^2 -4 \left( \left( 1- \epsilon \right)^2 -M \right) \right)\right) \\
		&= -M.
	\end{align*} 
However, since $ M $ is negative we conclude by \Cref{theo_no_solution} that 
	$$ x^2+ zxy + y^2 = M $$
has no solution.

The remaining part is a consequence of the above since $ p $ is never a square in $ \mathbb{N} $. Moreover, if $ \vert p \vert < \vert z \vert -2 $ is satisfied, then $ 2-\vert z \vert < \pm p < \vert z \vert - 2 <\vert z \vert + 2 $ which means that both $ -p,p \in \mathbb{Z}[i_z] $ are not of type I. By \Cref{lemma_ir_and_reg_primes} we have that $ -p,p \in \mathbb{Z}[i_z] $ are regular.

\end{proof}
\vspace{5mm}
\begin{figure}[h]  
\begin{center}
\pagestyle{empty}

\definecolor{ffqqqq}{rgb}{1.,0.,0.}
\definecolor{qqqqff}{rgb}{0.,0.,1.}
\definecolor{qqzzqq}{rgb}{0.,0.6,0.}
\definecolor{ududff}{rgb}{0.30196078431372547,0.30196078431372547,1.}
\definecolor{xdxdff}{rgb}{0.49019607843137253,0.49019607843137253,1.}
\definecolor{uuuuuu}{rgb}{0.26666666666666666,0.26666666666666666,0.26666666666666666}
\definecolor{ttttff}{rgb}{0.2,0.2,1.}
\definecolor{cqcqcq}{rgb}{0.7529411764705882,0.7529411764705882,0.7529411764705882}
\begin{tikzpicture}[line cap=round,line join=round,>=triangle 45,x=1.0cm,y=1.0cm]
\draw [color=cqcqcq,, xstep=1.0cm,ystep=1.0cm] (-1.5,-4.5) grid (5.5,1.5);
\draw[->,color=black] (-1.5,0.) -- (5.5,0.);
\foreach \x in {-1.,1.,2.,3.,4.,5.}
\draw[shift={(\x,0)},color=black] (0pt,2pt) -- (0pt,-2pt) node[below] {\footnotesize $\x$};
\draw[->,color=black] (0.,-4.5) -- (0.,1.5);
\foreach \y in {-4.,-3.,-2.,-1.,1.}
\draw[shift={(0,\y)},color=black] (2pt,0pt) -- (-2pt,0pt) node[left] {\footnotesize $\y$};
\draw[color=black] (0pt,-10pt) node[right] {\footnotesize $0$};
\clip(-1.5,-4.5) rectangle (5.5,1.5);
\fill[line width=2.pt,color=qqzzqq,fill=qqzzqq,fill opacity=0.10000000149011612] (0.,0.) -- (0.,-9.) -- (1.,-9.) -- (1.,-1.) -- (14.,-1.) -- (14.,0.) -- cycle;
\fill[line width=2.pt,color=ffqqqq,fill=ffqqqq,fill opacity=0.10000000149011612] (0.,0.) -- (0.8,-3.166190378969061) -- (3.16619037896906,-0.8) -- cycle;
\draw [samples=50,domain=-0.99:0.99,rotate around={-45.:(0.,0.)},xshift=0.cm,yshift=0.cm,line width=2.pt,color=ttttff] plot ({1.1547005383792515*(1+(\x)^2)/(1-(\x)^2)},{0.7559289460184544*2*(\x)/(1-(\x)^2)});
\draw [samples=50,domain=-0.99:0.99,rotate around={-45.:(0.,0.)},xshift=0.cm,yshift=0.cm,line width=2.pt,color=ttttff] plot ({1.1547005383792515*(-1-(\x)^2)/(1-(\x)^2)},{0.7559289460184544*(-2)*(\x)/(1-(\x)^2)});
\draw [line width=2.pt,dash pattern=on 3pt off 3pt,color=qqzzqq] (0.,0.)-- (0.,-9.);
\draw [line width=2.pt,color=qqzzqq] (0.,-9.)-- (1.,-9.);
\draw [line width=2.pt,dash pattern=on 3pt off 3pt,color=qqzzqq] (1.,-9.)-- (1.,-1.);
\draw [line width=2.pt,dash pattern=on 3pt off 3pt,color=qqzzqq] (1.,-1.)-- (14.,-1.);
\draw [line width=2.pt,color=qqzzqq] (14.,-1.)-- (14.,0.);
\draw [line width=2.pt,dash pattern=on 3pt off 3pt,color=qqzzqq] (14.,0.)-- (0.,0.);
\draw [line width=2.pt,color=ffqqqq] (0.,0.)-- (0.8,-3.166190378969061);
\draw [line width=2.pt,color=ffqqqq] (0.8,-3.166190378969061)-- (3.16619037896906,-0.8);
\draw [line width=2.pt,color=ffqqqq] (3.16619037896906,-0.8)-- (0.,0.);
\begin{scriptsize}
\draw[color=ttttff] (3.55,1.2) node {\fontsize{10}{0} $ x^2+5xy+y^2 = -2 $};
\draw [fill=uuuuuu] (0.,0.) circle (2.5pt);
\draw [fill=xdxdff] (0.,-9.) circle (2.5pt);
\draw [fill=ududff] (1.,-9.) circle (2.5pt);
\draw[color=qqzzqq] (4.45,-0.5) node {\fontsize{10}{0} $G$};
\draw [fill=ududff] (14.,-1.) circle (2.5pt);
\draw [fill=xdxdff] (14.,0.) circle (2.5pt);
\draw [fill=ududff] (0.8,3.) circle (2.5pt);
\draw [fill=qqqqff] (0.8,-3.166190378969061) circle (2.5pt);
\draw[color=qqqqff] (0.45,-3.35) node {\fontsize{10}{0} $\alpha_1 $};
\draw[color=qqqqff] (0.75,-0.4) node {\fontsize{10}{0} $B_{\alpha_1,\alpha_2} $};
\draw[color=qqqqff] (4.7,-1.38) node {\fontsize{10}{0} $ B $};
\draw [fill=qqqqff] (3.16619037896906,-0.8) circle (2.5pt);
\draw[color=qqqqff] (3.4,-0.55) node {\fontsize{10}{0} $\alpha_2$};
\draw[color=ffqqqq] (1.68,-1.3) node {\fontsize{10}{0} $ \tfrac{1}{2}\big< \alpha_1, \alpha_2 \big> $};

\end{scriptsize}
\end{tikzpicture}
\caption{The case $ M < 0 $ with respect to the proof of \Cref{lem_represented_M}}
\label{proof_lem_z-2_M_0}
\end{center}
\end{figure}
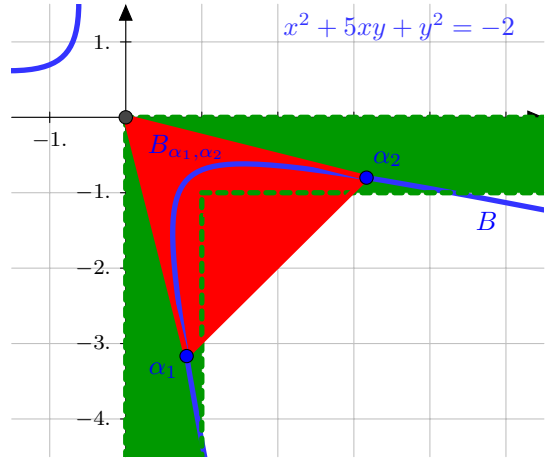
\vspace{5mm}


\begin{example}
	Consider \Cref{example_pm3_pm7} again. There we have that $ z = 6 $. Hence, by \Cref{lem_represented_M} we immediately get that all $ M \in \mathbb{Z} \setminus \left\{ 0,1,4 \right\} $ with $ -4 < M < 8 $ cannot be represented by $ x^2 + 6xy+ y^2 $, i.e. both Diophantine equations
	$$ x^2 + 6xy+ y^2 = 3 $$
and 	
	$$ x^2 + 6xy+ y^2 = -3 $$
have no solution and $ -3,3 \in \mathbb{Z}[i_6] $ are regular. 	
\end{example}

For $ z \geq 6 $ we can show the following consequence of \Cref{lem_represented_M} which will be useful for the proof of \Cref{theorem_special_primes}:

\begin{corollary} \label{coro_regular_p}
	If $ z \geq 6 $ and $ p \in \mathbb{Z} $ is prime with $ p \mid 2+z $ or $ p \mid 2-z $, but $ p \notin \left\{ 2 - z, 2+z \right\} $, then $ p \in \mathbb{Z}[i_z] $ is not special.
\end{corollary}

\begin{proof}
	If $ p $ divides either one of the the numbers $ 2+z $ or $ 2-z $ and $ p $ is not equal to them, then $ -\tfrac{2+z}{2} \leq p \leq \tfrac{2+z}{2} $. Moreover, $ z \geq 6 $ is equivalent to $ 2-z \leq -\tfrac{2+z}{2} $ and hence we deduce $ 2-z <  p  < 2+z $ because $ p \neq 2-z $. By \Cref{lem_represented_M} we have that $ p \in \mathbb{Z}[i_z] $ is not of type I and so we conclude that $ p \in \mathbb{Z}[i_z] $ is not special.
\end{proof}



\begin{proof}[Proof of \Cref{theorem_special_primes}]
Since the rings $ \mathbb{Z}[i_z] $, $ \mathbb{Z}[i_{-z}] $ are isomorphic such that $ \mathbb{Z} $ is preserved, both rings have the same special elements. Therefore we only have to check $ z \geq 0 $.

By \Cref{lemma_associate} we know that a prime $ p \in \mathbb{Z} $ of type I (in $ \mathbb{Z}[i_z] $) is special if and only if it divides either $ 2-z $ and/or $ 2+z $. Observe that a prime $ p  $ which divides either $ 2-z $ and/or $ 2+z $ must satisfy $-2-z \leq p \leq 2+z $. Moreover, primes in $ \mathbb{Z} $ of the form $ 2-z < p < 2+z $ are not of type I by \Cref{lem_represented_M}. Hence, only the candidates 
$$ -2-z,-1-z,-z,1-z,2-z,2+z \in \mathbb{Z}[i_z] $$
can be special and each of them is special if and only if it is of type I and divides either $ 2-z $ or $ 2+z $.

Observe that $ 2-z $ and $ 2+z $ can be represented by $ x^2+zxy+y^2 $ for $ x=1 $ and $ y = 1 $ or $ y = -1 $, respectively, and so all irregular elements of either one of these forms are of type I. Moreover, if $ z \geq 6 $, then such a $ p $ is of type I in the corresponding $ z $-ring if and only if $ p $ is equal to $ 2-z $ or $ 2+z $ by \Cref{coro_regular_p}. Hence, if $ z \geq 6 $, then the special elements of $ \mathbb{Z}[i_z] $ are exactly the primes in $ \mathbb{Z} $ of the form $ 2-z $ or $ 2+z $ and for all $ z \in \left\{ 0,1,2,3,4,5 \right\} $ we have to check all candidates above whether they are special or not separately.

We start with $ z \in \left\{ 0,1,2 \right\} $. Then the above candidates without $ 2-z,2+z $ are all either negative or equal to $ 0 $ or $ 1 $. Hence, they cannot be special as special elements in these $ z $-rings have to be positive prime numbers in $ \mathbb{Z} $ (compare with \Cref{ex_non_neg_sol}, \Cref{example_primes_Gaussian} and \Cref{example_primes_z123}). 

Let $ z = 3 $, then the only candidates being primes in $ \mathbb{Z} $ are $ -2-z = -5 $, $ -z = -3 $, $ 1-z = -2 $, $ 2+z = 5 $. Since $ 2+z = 5 \in \mathbb{Z}[i_3] $ is prime in $ \mathbb{Z} $ we clearly have that it is special. By \Cref{coro_solvability_pm_p} we clearly get that $ -5 \in \mathbb{Z}[i_3] $ is also special because it is of type I, too, and it divides $ z+2 $. However, the other candidates $ -3,-2 $ do neither divide $ 2-z=-1 $ or $ 2+z=5 $, so they cannot be special.

If $ z = 4 $, then the only candidates being primes in $ \mathbb{Z} $ are $ -1-z = -5 $, $ 1-z = -3 $ and $ 2-z = -2 $. Then clearly $ 2-z \in \mathbb{Z}[i_4] $ is special. However, $ -5 $ is not as it does not divide either $ 2-z = -2 $ nor $ 2+z = 6 $. It remains to show that $ -3 \in \mathbb{Z}[i_4] $ is of type I. This follows from
	$$ x^2+4xy+y^2 = -3 $$
if we set $ x = 1 $ and $ y = -2 $ and so $ -3 \in \mathbb{Z}[i_4] $ is special, too.

If $ z = 5 $, then the candidates to check are $ -2-z = -7 $, $ -z = -5 $ and $ 2-z = -3 $ which must be special. However, $ -5 $ cannot be special because it does not divide $ 2-z = -3 $ or $ 2+z = 7 $, nor $ -7 $ is because $ 2+z = 7 \in \mathbb{Z}[i_5] $ is special as prime in $ \mathbb{Z} $ and so $ -7 $ cannot also be of type I by \Cref{coro_solvability_pm_p}.
\end{proof}

\subsection{Many $ z $-rings are not unique factorization domains}

In \Cref{primes_if_uniq_factor_domain} we assumed that $ \mathbb{Z}[i_z] $ is a unique factorization domain. However, in general it is difficult to decide which of these $ z $-rings are unique factorization domains and which not. For example, it is known that $ \mathbb{Z}[i_z] $ is a unique factorization domain, if $ \vert z \vert \leq 5 $. In this section we would like to show that most of the $ \mathbb{Z}[i_z] $ are not unique factorization domains. More concretely, whenever $ 2-z,2+z \in \mathbb{Z} $ are not both primes for $ \left\vert z \right\vert \geq 6 $, then $ \mathbb{Z}[i_z] $ is not a unique factorization domain. At the end of this section we will discuss that the reverse statement does not hold true, i.e. there are $ z $-rings where $ z \pm 2 \in \mathbb{Z} $ are both primes and $ \left\vert z \right\vert \geq 6 $, but $ \mathbb{Z}[i_z] $ is not a unique factorization domain.

\begin{lemma} \label{lemma_irred_elem}
If $ z \in \mathbb{Z} $ with $ \left\vert z \right\vert \geq 6 $, then the elements 
	$$ \left( -1 \right)^n  + \left( -1 \right)^m i_z \in \mathbb{Z}[i_z] $$
are irreducible for all $ n,m \in \left\{ 0,1 \right\} $.
\end{lemma}

\begin{proof}
It is enough to consider $ 6 \geq z $ as the isomorphism 			$$ \Phi: \mathbb{Z}[i_z] \to \mathbb{Z}[i_{-z}] $$  changes only the sign of the imaginary unit and hence we can conclude as irreducibility is preserved by ring isomorphisms.

To simplify the notation let
$$ \gamma_{n,m} \coloneqq \left( -1 \right)^n + \left( -1 \right)^m i_z $$
for $ n,m \in \left\{ 0,1 \right\} $ arbitrary and observe that
$$ \N \left( \gamma_{n,m} \right) = \left\{ 
\begin{array}{ll}
	2+z & n+m \equiv 0 \pmod 2 \\ [1ex]
	2-z & n+m \equiv 1 \pmod 2
\end{array}	
\right. .$$
We assume now that $ \gamma_{n,m} $ is reducible. Then we find $ \alpha, \beta \in \mathbb{Z}[i_z] $ with $ \gamma_{n,m} = \alpha \beta $ such that $ \alpha $ and $ \beta $ are no units. Therefore we have $  \N \left( \gamma_{n,m} \right) = \N \left( \alpha \right) \N \left( \beta \right) $ and 
	$$ 2 \leq \left\vert \N \left( \alpha \right) \right\vert, \left\vert \N \left( \beta \right) \right\vert < z + 2 $$
by \Cref{Nz_lemma}.

At first we will consider the case $ z = 6 $. Then $ \N \left( \gamma_{n,m} \right) \in \left\{ -4,8 \right\} $. Hence, either $ \left\vert \N \left( \alpha \right) \right\vert = 2 $ or $ \left\vert \N \left( \beta \right) \right\vert = 2 $. However, we have that
	$$ 2 = \left\vert \pm 2 \right\vert <  z  -2 = 4 $$
and so $ -2,2 \in \mathbb{Z}[i_6] $ are regular by \Cref{lem_represented_M} which is a contradiction. Therefore $ \gamma_{n,m} \in \mathbb{Z}[i_6] $ is irreducible.

Now let us assume that $ z > 6 $. We have
	$$ \left\vert \N \left( \alpha \right) \right\vert = \left\vert \frac{\N \left( \gamma_{n,m} \right)}{\N \left( \beta \right)} \right\vert \leq \frac{2+z}{2} < z-2 < z + 2 $$
where the second last inequality is equivalent to $ z > 6 $. Thus, we clearly have
	$$ 2-z < \N \left( \alpha \right) < 2+z $$
and by \Cref{lem_represented_M} we conclude that $ \sqrt{\N \left( \alpha \right)} \in \mathbb{N} $. 
Hence, $ \N \left( \alpha \right) $ is positive which allows us to use \Cref{coro_pos_sol} and so we find a unit $ \varepsilon \in \mathbb{Z}[i_z] $ such that $ \varepsilon\alpha \in \mathbb{Z}[i_z] $ has non-negative real and imaginary part. Assume $ \mathrm{Re} \left( \varepsilon\alpha \right), \mathrm{Im} \left( \varepsilon\alpha \right) \geq 1 $, then
	$$ \N \left( \varepsilon\alpha \right) \geq 1^2 + z + 1^2 \geq z + 2 $$
and we get a contradiction. Therefore either the real or the imaginary part of $ \varepsilon\alpha $ is equal to zero and so we clearly have 
	$$ \varepsilon\alpha \in \left\{ \sqrt{\N \left( \alpha \right)}, \sqrt{\N \left( \alpha \right)}i_z \right\} .$$
Since both elements in the set above are associated, we get that $ \alpha $ is associated to $ \sqrt{\N \left( \alpha \right)} \in \mathbb{Z}[i_z] $. Therefore $ \alpha \mid \gamma_{n,m} $ also implies that $ \sqrt{\N \left( \alpha \right)} \mid \gamma_{n,m} $. However, $ \sqrt{\N \left( \alpha \right)} \in \mathbb{Z} $ and so we have that $$ \sqrt{\N \left( \alpha \right)} \mid \mathrm{Re} \left( \gamma_{n,m} \right) \in \left\{ -1,1 \right\} $$
by \Cref{lem_trinity_divisability}. Finally, we conclude that $ \alpha \in \mathbb{Z}[i_z] $ is a unit which is a contradiction and so $ \gamma_{n,m} \in \mathbb{Z}[i_z] $ is also irreducible for $ z > 6 $.
\end{proof}




\begin{theorem} \label{theo_not_uniqe_factor_domain}
	Let $ z \in \mathbb{Z} $ with $ \left\vert z \right\vert \geq 6 $ and $ 2 \pm z \in \mathbb{Z} $ are not both primes. Then $ \mathbb{Z}[i_z] $ is not a unique factorization domain.
\end{theorem}

\begin{proof}
By isomorphy of $ z $- rings it is enough to prove the statement for $ z \geq 6 $. To show that an integral domain is not a unique factorization domain we simply need to show that there exist irreducible elements which are not prime. Let us consider the case $ z = 6 $ separately. Then both, $ 2 - z = 4 $ and $ 2+z = 8 $ are no primes. For example, we have
	$$ \left( 1 + i_6 \right)^2 = 8i_6 = 2 \cdot 4i_6 $$
where we showed that $ 2 \in \mathbb{Z}[i_6] $ is regular (recall the proof of \Cref{lemma_irred_elem}), i.e. irreducible. By the same lemma we also have that $ 1+i_6  \in \mathbb{Z}[i_6] $ is irreducible and therefore $ 2 \mid 1 + i_6 $ if we assume that $ \mathbb{Z}[i_z] $ is a unique factorization domain which lead us to contradiction as $ 2 \nmid 1 $ by \Cref{lem_trinity_divisability}. Hence, $ 2 \in \mathbb{Z}[i_6] $ is irreducible, but not prime which shows that $ \mathbb{Z}[i_6] $ is not a unique factorization domain.

Let now $ z > 6 $ and assume that either one of $ 2+z,2-z \in \mathbb{Z} $ is not a prime. To consider both cases in one we will say that $ 2 \pm z \in \mathbb{Z} $ is not a prime. Then we find a prime number $ p \in \mathbb{N} $ such that $ p \mid 2 \pm z $ and $ p^2 \leq z+2 $. Moreover, we have that $ p, \tfrac{z \pm 2}{p} \in \mathbb{Z}[i_z] $ and 
	\begin{align*}
		\left( 1 \pm i_z \right)^2 = \left(z \pm 2 \right)i_z = p \cdot \frac{z \pm 2}{p}i_z.
	\end{align*}
We will now show that $ p \in \mathbb{Z}[i_z] $ is irreducible. We have that 
	$$ \left\vert p \right\vert \leq \frac{z+2}{2}< z-2 $$
where the above inequality is equivalent to $ z > 6 $. By \Cref{lem_represented_M} this means that $ p \in \mathbb{Z}[i_z] $ is irreducible. By \Cref{lemma_irred_elem} we know that $ 1 \pm i_z \in \mathbb{Z}[i_z] $ is irreducible, too. Hence, $  p \in \mathbb{Z}[i_z] $ and $ 1 \pm i_z $ must be associated if $ \mathbb{Z}[i_z] $ is a unique factorization domain. However, if they are associated, then $ p \mid 1 + i_z $ in $ \mathbb{Z}[i_z] $ and so $ p \mid 1 $ in $ \mathbb{Z} $ by \Cref{lem_trinity_divisability} which is a contradiction.
\end{proof}



\begin{example}
	In fact, the assumption $ \vert z \vert \geq 6 $ in \Cref{theo_not_uniqe_factor_domain} is necessary. Consider the case $ z = 4 $, then both, $ 2-z = -4 \in \mathbb{Z} $ and	$ 2+z = 8 \in \mathbb{Z} $ are not primes. For example, we have
	$$ \left( 1+i_4 \right)^2 = 6i_4 = 2 \cdot 3 i_4$$
However, the problem here is that all the factors above are not irreducible in $ \mathbb{Z}[i_4] $. In fact, we have
	\begin{align*}
		\left(1 + i_4\right)_{6} &= \left( 1-i_4 \right)_{-2} \left( 2-i_4 \right)_{-3} = \left( 3- i_4 \right)_{-2} \left( 1-2i_4 \right)_{-3} \\
		2 &= \left( 1-i_4 \right)_{-2} \left( 3- i_4 \right)_{-2} \\
		3i_4 &= \left( 2-i_4 \right)_{-3}\left( 1-2i_4 \right)_{-3}
	\end{align*}
where all the factors on the right-hand side must be irreducible because their norm is prime in $ \mathbb{Z} $.
\end{example}


\begin{example} \label{ex_prime__neq_irreducible}
Observe that the reverse statement of \Cref{theo_not_uniqe_factor_domain} is not true, i.e. there are also $ z $-rings with both $ 2-z,2+z \in \mathbb{Z} $ primes, but they are still not unique factorization domains. Recall \Cref{exa_z39} where
	$$ \left( 5-i_{39} \right)_{-169} \left( -34 + i_{39} \right)_{-169} = -13^2 $$
with $ \left( -34 + i_{39} \right)_{-169} = \overline{\left( 5-i_{39} \right)_{-169}} $. Note that
	$$ \left\vert \pm 13 \right\vert < z-2 = 37 $$
and so we get that $ -13, 13 \in \mathbb{Z}[i_{39}] $ are of type II by \Cref{lem_represented_M} and so all elements with norm $ \pm 169 $ as $ 5-i_{39}, -34 + i_{39}, 13 \in \mathbb{Z}[i_{39}] $ must be irreducible. However, $ 13 \in \mathbb{Z}[i_{39}] $ is not a prime element as it divides neither $ 5-i_{39} \in \mathbb{Z}[i_{39}] $ nor $ 34 + i_{39} \in \mathbb{Z}[i_{39}] $ by \Cref{lem_trinity_divisability}. Hence, $ \mathbb{Z}[i_{39}] $ is not a unique factorization domain even if $ 2-z,2+z \in \mathbb{Z} $ are prime numbers.
\end{example}





We could ask whether there are infinitely many unique factorization domains of the form $ \mathbb{Z}[i_z] $ for $ z \in \mathbb{Z} $ or not. A necessary condition for the existence of infinitely many of them is the existence of infinitely many prime pairs $ p,p+4 \in \mathbb{Z} $. Such pairs are called cousin primes and indeed there are infinitely many of these cousin primes, see\cite{Cousin_primes}.

\begin{section}{Positive, primitive solutions of the Diophantine equation $ x^2 + zxy + y^2 = M $ for $ M $ being a product of irregular primes}

\begin{subsection}{The general case}

With the tools we have from the previous sections we can now deal with the question about the number of positive solutions to Diophantine equations of the form $ x^2 + zxy + y^2 = M $ for $ z,M \in \mathbb{Z} $ (particularly, $ M,z \in \mathbb{N} $) if $ M $ is a product of irregular elements in $ \mathbb{Z}[i_z] $. Note that all the statements in this section hold trivially true for $ z \in \left\{ -2,2 \right\} $ as there do not exist irregular elements in these $ z $-rings by \Cref{example_primes_z123}. The next statement is a generalization of Proposition $ 3 $ from \cite{Miniatur}. 

\begin{proposition} \label{prop_3_miniatur}
	Let $ z \in \mathbb{N} $, $ k \in \mathbb{N} \setminus \{ 0 \} $ and $ p = \alpha \overline{\alpha} \in \mathbb{Z}[i_z] $ be irregular, but not special such that $ p^k > 0 $. Then there exists a unit $ \varepsilon \in \mathbb{Z}[i_z] $ such that $ \varepsilon \alpha^k $ is the unique positive, primitive solution to the equation $ x^2 + zxy + y^2 = p^k $.
\end{proposition}

\begin{proof}
	Observe that we have
	$$ \mathrm{Re} \left( \alpha^k \right)^2 + z \mathrm{Re} \left( \alpha^k \right) \mathrm{Im} \left( \alpha^k \right) + \mathrm{Im} \left( \alpha^k \right)^2 = \N \left( \alpha^k \right) = \N \left( \alpha \right)^k = p^k > 0 $$
and hence $ \alpha^k $	
satisfies the equation $ x^2 + zxy + y^2 = p^k $. By	\Cref{coro_pos_sol} we find a unit $ \varepsilon \in \mathbb{Z}[i_z] $ such that the real an the imaginary part of $ \varepsilon \alpha^k $ are positive. Since $ \N \left( \varepsilon \alpha^k \right) = \N \left( \alpha^k \right) = p^k $, we conclude that $ \varepsilon \alpha^k $ also satisfies $ x^2 + zxy + y^2 = p^k $ which shows the existence of a positive solution.

Now we would like to show that our solution is primitive. Assume not, then there is $ \lambda \in \mathbb{Z} \setminus \{ -1,1 \} $ such the real and the imaginary part of $ \varepsilon \alpha^k $ are divided by $ \lambda $. By \Cref{lem_trinity_divisability} this means that also the norm of $ \varepsilon \alpha^k $,
	$$ \N \left( \varepsilon\alpha^k \right) = \N \left( \varepsilon \right)\N \left( \alpha^k \right) = p^k,$$
is divided by $ \lambda $ and hence $ p \mid \lambda $ which means $ p $ also divides the real and imaginary part of $ \varepsilon \alpha^k $ so $ p = \alpha \overline{\alpha} \mid \varepsilon \alpha^k $ again by \Cref{lem_trinity_divisability}. Now $ k = 1 $ would imply $ \overline{\alpha} \mid \varepsilon $ and this is a contradiction. Hence, $ k > 1 $ and then $ \overline{\alpha} \mid \varepsilon \alpha^{k-1} $ finally implies $ \overline{\alpha} \mid \alpha $ which is a contradiction as $ p $ is not special.


Now we will show that $ \varepsilon \alpha^k $ is the unique positive, primitive solution to $ x^2 + zxy + y^2 = p^k $. We can use a similar trick as in the proof of Proposition 3 in \cite{Miniatur}: Assume that there is another positive, primitive solution $ a,b \in \mathbb{Z} $ with $ a^2 + zab + b^2 = p^k $. Then we have
	$$ \left( a+bi_z\right) \overline{\left( a+bi_z\right)} = p^k = \alpha^k \overline{\alpha}^k .$$
Since $ \alpha, \overline{\alpha} \in \mathbb{Z}[i_z] $ are prime we get that each of them divide one of the factors on the left-hand side. However, non of them divides the same factor because then our solution $ a + bi_z \in \mathbb{Z}[i_z] $ would not be primitive. Therefore, without loss of generality, we can assume that $ \alpha^k \mid \left( a+bi_z\right) $.

Thus, we have
	$$ \frac{a+bi_z}{\alpha^k} \overline{\left(\frac{a+bi_z}{\alpha^k}\right)} = \N \left( \frac{a+bi_z}{\alpha^k} \right) = 1 $$
and so both factors of the left-hand side are units, i.e. there exist a unit $ \varepsilon \in \mathbb{Z}[i_z] $ such that $ \varepsilon \alpha^k = a + bi_z $. By \Cref{coro_pos_sol} there exist only one associated positive, primitive solution to $ x^2 + zxy + y^2 = p^k $ and so we conclude that $ \varepsilon = 1 $ which shows uniqueness.
\end{proof}


If we compare the proof above with the proof of Proposition $ 3 $ in \cite{Miniatur} we see that they look similar, but the more general case here needs other tools as we cannot use Niven's theorem any more because we do not have the link to trigonometric functions which we have if we work with complex numbers. Moreover, it is in general more difficult transform a primitive solution of $ x^2 + zxy + y^2 = M $ to a positive, primitive solution (for the Gaussian integers this was way more simple since we could just take the absolute value of $ x $ and $ y $).

\begin{example} \label{example_pos_sol_49}
We would like to find the unique positive, primitive solution of the Diophantine equation 
	$$ x^2+6xy+y^2 = 49 .$$
By \Cref{example_z6_pm3_pm7}	we already know that $ -7 \in \mathbb{Z}[i_6] $ is of type I and the Diophantine equation 
	$$ x^2+6xy+y^2 = -7 $$
can be solved by $ x = 4 $ and $ y = -1 $. Hence, we have 
	$$ -7 = \left( 4 -i_6 \right) \overline{\left( 4 -i_6 \right)} = \left( 4 -i_6 \right)\left( -2 +i_6 \right) .$$
We set $ \alpha \coloneqq 4-i_6 $, then 
	$$ \alpha^2 = 16-8i_6+i_6^2 = 15-2i_6 $$
must solve the Diophantine equation on the top and it must be primitive (what we can see easily). However, our solution is not positive. Since our solution is on the branch which intersects the first quadrant, there must be $ n \in \mathbb{Z} $ such that $ i_6^n \alpha $ is positive and so in the first quadrant. Recall \Cref{proposition_prop_iz} and/or \Cref{Separating_lines_positive} to see that $ n > 0 $. 
Here we have
	$$ i_6\alpha = 15i_6-2i_6^2 = 2 + 3i_6 $$
which is the positive, primitive solution of the considered Diophantine equation. By \Cref{prop_3_miniatur} we know that it is unique up to interchanging the order of $ x $ and $ y $ what we can see in \Cref{example_p7_sol_49}: Indeed, $ S_{49} $ intersects the $ \mathbb{Z} \times \mathbb{Z}i_z $-grid in the first quadrant only four times where two and two of them are symmetric with respect to their real and imaginary parts. Moreover, the intersection on the axes is not a primitive solution. If we work with $ \overline{\alpha} $ instead of $ \alpha $ we also get that 
	$$ \overline{\alpha}^2 = \left( -2+i_6 \right)^2 = 4-4i_6+i_6^2 = 3+2i_6 $$
and so $ \overline{\alpha}^2 $ is already the primitive, positive solution to the above Diophantine where just $ x $ and $ y $ are interchanged.
\end{example}

\vspace{5mm}
\begin{figure}[h]  
\begin{center}
\pagestyle{empty}

\definecolor{ffqqqq}{rgb}{1.,0.,0.}
\definecolor{qqqqff}{rgb}{0.,0.,1.}
\definecolor{qqzzqq}{rgb}{0.,0.6,0.}
\definecolor{cqcqcq}{rgb}{0.7529411764705882,0.7529411764705882,0.7529411764705882}
\begin{tikzpicture}[line cap=round,line join=round,>=triangle 45,x=0.6cm,y=0.6cm]
\draw [color=cqcqcq,, xstep=0.6cm,ystep=0.6cm] (-2.7673595955934123,-3.2674063907421544) grid (15.795523955113431,7.4737280080347075);
\draw[->,color=black] (-2.7673595955934123,0.) -- (15.795523955113431,0.);
\foreach \x in {-2.,-1.,1.,2.,3.,4.,5.,6.,7.,8.,9.,10.,11.,12.,13.,14.,15.}
\draw[shift={(\x,0)},color=black] (0pt,2pt) -- (0pt,-2pt) node[below] {\footnotesize $\x$};
\draw[->,color=black] (0.,-3.2674063907421544) -- (0.,7.4737280080347075);
\foreach \y in {-3.,-2.,-1.,1.,2.,3.,4.,5.,6.,7.}
\draw[shift={(0,\y)},color=black] (2pt,0pt) -- (-2pt,0pt) node[left] {\footnotesize $\y$};
\draw[color=black] (0pt,-10pt) node[right] {\footnotesize $0$};
\clip(-2.7673595955934123,-3.2674063907421544) rectangle (15.795523955113431,7.4737280080347075);
\draw [samples=50,domain=-0.99:0.99,rotate around={-45.:(0.,0.)},xshift=0.cm,yshift=0.cm,line width=2.pt,color=qqzzqq] plot ({1.8708286933869707*(1+(\x)^2)/(1-(\x)^2)},{1.3228756555322954*2*(\x)/(1-(\x)^2)});
\draw [samples=50,domain=-0.99:0.99,rotate around={-45.:(0.,0.)},xshift=0.cm,yshift=0.cm,line width=2.pt,color=qqzzqq] plot ({1.8708286933869707*(-1-(\x)^2)/(1-(\x)^2)},{1.3228756555322954*(-2)*(\x)/(1-(\x)^2)});
\draw [samples=50,domain=-0.99:0.99,rotate around={-135.:(0.,0.)},xshift=0.cm,yshift=0.cm,line width=2.pt,color=qqqqff] plot ({3.5*(1+(\x)^2)/(1-(\x)^2)},{4.949747468305833*2*(\x)/(1-(\x)^2)});
\draw [samples=50,domain=-0.99:0.99,rotate around={-135.:(0.,0.)},xshift=0.cm,yshift=0.cm,line width=2.pt,color=qqqqff] plot ({3.5*(-1-(\x)^2)/(1-(\x)^2)},{4.949747468305833*(-2)*(\x)/(1-(\x)^2)});
\begin{scriptsize}
\draw [fill=qqzzqq] (4.,-1.) circle (2.5pt);
\draw[color=qqzzqq] (4.025,-2.675) node {\fontsize{10}{0} $ x^2+6xy+y^2 = -7 $};
\draw [fill=qqzzqq] (-2.,1.) circle (2.5pt);
\draw[color=qqzzqq] (-2.1,1.6) node {\fontsize{10}{0} $ \overline{\alpha} $};
\draw[color=qqzzqq] (3.9,-1.575) node {\fontsize{10}{0} $ \alpha $};
\draw [fill=qqqqff] (15.,-2.) circle (2.5pt);
\draw[color=qqqqff] (3.95,6.325) node {\fontsize{10}{0} $x^2+6xy+y^2 = 49$};
\draw[color=qqqqff] (15.2,-1.4) node {\fontsize{10}{0} $ \alpha^2 $};
\draw [fill=ffqqqq] (2.,3.) circle (2.5pt);
\draw[color=ffqqqq] (2.7,3.5) node {\fontsize{10}{0} $ i_6\alpha^2 $};
\draw [fill=ffqqqq] (3.,2.) circle (2.5pt);
\draw[color=ffqqqq] (3.575,2.475) node {\fontsize{10}{0} $ \overline{\alpha}^2 $};

\end{scriptsize}
\end{tikzpicture}
\caption{Positive, primitive solution to $ x^2+ 6xy + y^2 = 7^2 $}
\label{example_p7_sol_49}
\end{center}
\end{figure}
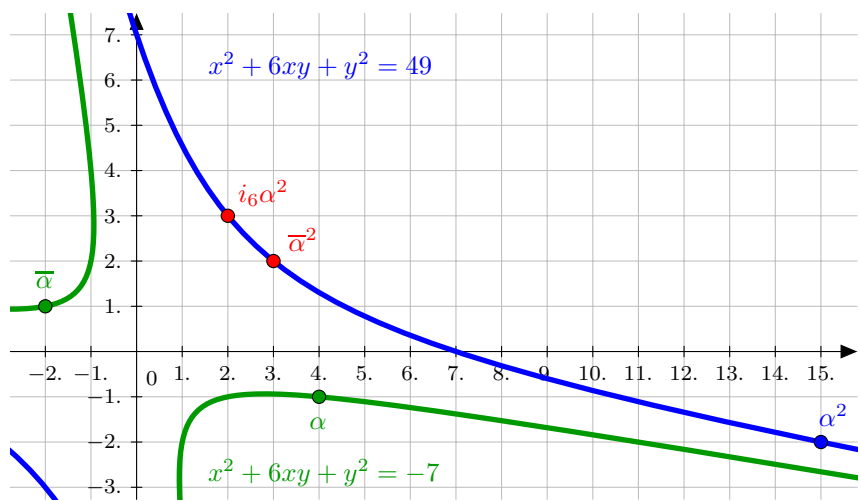
\vspace{5mm}

In \Cref{prop_3_miniatur} we considered the positive, primitive solution in the cases where $ z \geq 0 $ and $ p^k>0 $ if $ p \in \mathbb{Z}[i_z] $ is of type I. However, with the help of this proposition, the concept of subbranches and the isomorphism between the $ z $-rings and other observations we did in the sections before we can also discuss the question about the number of solutions of the Diophantine equation $ x^2+zxy+y^2 = p^k $ and their construction for $ z \in \mathbb{Z} $ and $ k \in \mathbb{Z} \setminus \left\{ 0 \right\} $ in general (even if the solution is not positive and/or not primitive) as long as $ p = \alpha \overline{\alpha} \in \mathbb{Z}[i_z] $ is of type I, but it might also be special. We will consider two solutions of the form $ \left\{ x,y \right\}, \left\{ -x,-y \right\} $ as the same if they are in the same quadrant. Note that associated solutions to Diophantine equations are either both primitive or not by vi) of \Cref{proposition_prop_iz} and \Cref{coro_In}. We will treat now all the different cases:

We start with the case $ p^k > 0 $. Let $ z \geq 0 $ and assume that $ p \in \mathbb{Z}[i_z] $ is not special. Then for each element in the list
	$$ \alpha^k, \overline{\alpha}{\alpha}^{k-1}, \overline{\alpha}^2{\alpha}^{k-2}, \dots, \overline{\alpha}^k $$
there is an associated element on the subbranch $ B_{\sqrt{p^k}} $ (actually on every choice of subbranch) where all the elements in the list cannot be associated as $ p $ is not special and so all of them are representatives of different equivalence classes with respect to association, but not necessarily of different solution classes of the Diophantine equation as some of them might be associated to elements in $ B_{\sqrt{p^k}} $ where they have just exchanged real and imaginary parts. In fact, this happens if and only if solutions of the Diophantine equation are conjugated to each other as
	$ \widetilde{\alpha} = i_z \overline{\alpha} $
where $ \widetilde{\alpha} = \mathrm{Im}\left( \alpha\right) + \mathrm{Re}\left( \alpha\right)i_z $. Hence, only the first $ \lceil \tfrac{k+1}{2} \rceil $ elements in the above list are associated to different solutions in $ B_{\sqrt{p^k}} $ of the Diophantine equation $ x^2 + zxy + y^2 = p^k $ and $ \alpha^k $ is associated to the unique primitive solution to the above Diophantine equation and all the other solutions are not primitive (if $ p \in \mathbb{Z}[i_z] $ is special, then there is no primitive solution for $ k > 1 $).

Let us consider now the case $ z = 0 $. In this case the first quadrant is equal to $ B_{\sqrt{p^k}} \cup \left\{ \sqrt{p^k}i\right\} $, but we do not need to consider $ \sqrt{p^k}i $ as it is the same solution as $ \sqrt{p^k} $ for the Diophantine equation. By \Cref{prop_3_miniatur} and what we discussed before we know that there is only one positive, primitive solution. Additionally, we get that there must be $ \lceil \tfrac{k+1}{2} \rceil -1 $ non-primitive solutions in $ B_{\sqrt{p^k}} $ and so also in the first quadrant. In all the other quadrants we have the same story by symmetry reasons. This means that there is exactly the same amount of primitive and non-primitive solutions to $ x^2+y^2 = p^k $ if $ x,y \geq 0 $ as, for example, for $ x \geq 0 $ and $ y \leq 0 $ (or another choice of $ \leq, \geq $ for both). In case $ p = 2 $, i.e. $ p \in \mathbb{Z}[i_z] $ is special, then all the above representatives of solutions are associated. Hence, there is only one solution to the Diophantine equation and this solution is primitive if and only if $ k = 1 $.

Let $ z = 1 $, then the first quadrant of $ \mathbb{Z} \times \mathbb{Z}i_1 $ is also covered by $ B_{\sqrt{p^k}} \cup \left\{ \sqrt{p^k}i_1\right\} $ where $ \sqrt{p^k} $ and $ \sqrt{p^k}i_1 $ are associated as well and so they are the same solution for the Diophantine equation $ x^2+xy+y^2 = p^k $. Hence, the number of solutions for this Diophantine equation in the first and the third quadrant remains the same by symmetry. However, the second quadrant is covered by two branches as well as a further element, namely $ B_{\I_{+}\left( \sqrt{p^k}\right)} \cup B_{\I_{+}^2\left( \sqrt{p^k}\right)} \cup \left\{ -\sqrt{p^k} \right\} $, and both branches are symmetric in the second quadrant with respect to the diagonal going through the origin and the second and fourth quadrant, respectively, so all associated representatives of the above list in $ B_{\I_{+}\left( \sqrt{p^k}\right)} $ give us a different solution to the above Diophantine equation. Moreover, if $ \I_{+} \left( \sqrt{p^k} \right) $ and $ \I_{+}^2 \left( \sqrt{p^k} \right) $ solves the equation (this happens if and only if $ k $ is even) then they are associated, but not the same solution of the equation and so both of them should be counted as different solutions. In total we get $ 2 \lceil \tfrac{k+1}{2} \rceil $ solutions (i.e. $ k+1 $ and $ k+2 $ if $ k $ is odd or even, respectively) of the Diophantine equation in the second and fourth quadrant. Two of them are primitive and the rest is non-primitive. In case $ p = 3 $, then there is only one solution in the first and third quadrant and two in the second and fourth quadrant which are all primitive if $ k = 1 $ and if $ k > 1 $, then the amount of solutions is the same, but all of them are non-primitive.

If $ z = 2 $, then there are no irregular primes and so there is nothing to show.

If $ z > 2 $, then the amount of solutions in the first and third quadrant is still the same, but there are infinitely many primitive solutions in the second and fourth quadrant to the Diophantine equation $ x^2 + zxy + y^2 = p^k $ as there are infinitely many subbranches contained in both of these quadrants. If $ p \in \mathbb{Z}[i_z] $ is special, we will again have just one primitive solution in the first and third quadrant and infinitely many primitive solutions in the second and fourth quadrant if $ k = 1 $ and if $ k > 1 $ the number of solutions in the quadrants remains the same, but all of them are non-primitive.

Now we can consider the cases if $ z < 0 $. Clearly the isomorphism between $ \mathbb{Z}[i_z] $ and $ \mathbb{Z}[i_{-z}] $ changes the quadrants i.e. what was true for the first/third quadrant for $ z > 0 $ is now true for the second/fourth quadrant and also the other way round.

We discuss what happen if $ p^k < 0 $. Note that we do not have to treat the cases $ z \in \left\{ 0,\pm 1, \pm 2 \right\} $ as there are no negative elements of type I in $ \mathbb{Z}[i_z] $.

Let $ z > 2 $, then there are no solutions in the first and third quadrant to the Diophantine equation $ x^2 + zxy + y^2 = p^k $ and infinitely many primitive solutions in the second and fourth quadrant as these quadrants contain infinitely many subbranches for $ p $ being non-special. This is true even if $ p \in \mathbb{Z}[i_z] $ is special for $ k = 1 $ and all these solutions must be non-primitive if $ k > 1 $.

For $ z < 2 $ we have the same story as for $ z > 2 $ just with the difference that the roles of the first/third and the second/fourth quadrant are exchanged.



The next proof will be similar to Theorem 4 in \cite{Miniatur} (note that we cannot just take absolute values to make the solution positive and so we will multiply our solution with a unit to reach that):

\begin{proposition} \label{theo_4_miniatur}
	Let $ z \in \mathbb{N} $ and $ n, k_l > 0 $ be integers, $ p_l = \alpha_l \overline{\alpha_l} \in \mathbb{Z}[i_z] $ be pairwise distinct non-special elements with different absolute values for all $ l = 1, \dots, n $ and let $ M = \prod_{l=1}^{n}p_l^{k_l} $. Then there exist a unit $ \varepsilon \in \mathbb{Z}[i_z] $ such that $ \varepsilon \prod_{l=1}^{n}\alpha_l^{k_l} $ is a positive, primitive solution to $ x^2+zxy+y^2 = M $.
\end{proposition}

\begin{proof}
	First of all,
		$$ \N \left( \prod_{l=1}^{n}\alpha_l^{k_l} \right) = \prod_{l=1}^{n}\alpha_l^{k_l}  \overline{\prod_{l=1}^{n}\alpha_l^{k_l}} = M $$
holds and by \Cref{coro_pos_sol} we find a unit $ \varepsilon \in \mathbb{Z}[i_z] $ such that $ \varepsilon \prod_{l=1}^{n}\alpha_l^{k_l} $ is a positive solution to $ x^2+zxy+y^2 = M $.

It remains to show that this solution is primitive. If not, then there must exist $ \lambda \in \mathbb{Z} \setminus \left\{ -1,1 \right\} $ such that $ \lambda \mid M $ and so $ \lambda $ must be divisible by at least one of the $ p_l $'s. Without loss of generality, let us assume that $ p_1 \mid \lambda $. Hence, $ p_1 $ also divides the real and the imaginary part of $ \varepsilon \prod_{l=1}^{n}\alpha_l^{k_l} $ which implies $ \alpha_1 \overline{\alpha_1} = p_1 \mid \varepsilon \prod_{l=1}^{n}\alpha_l^{k_l} $ by \Cref{lem_trinity_divisability}. Hence, there are $ l_1,l_2 \in \left\{ 1,2, \dots, n \right\} $ such that $ \alpha_1 \mid \alpha_{l_1} $ and $ \overline{\alpha_1} \mid \alpha_{l_2} $ because $ \alpha_1, \overline{\alpha_1} \in \mathbb{Z}[i_z] $ are prime. Therefore we deduce $ p_1 = \N \left( \alpha_1 \right) \mid \N \left( \alpha_{j} \right) = p_j $ for $ j = l_1,l_2 $ which implies $ l_1 = 1 = l_2 $. 
This means $ p_1 \mid \alpha_1^{k_1} $. Now we can proceed as in the proof of \Cref{prop_3_miniatur} i.e. we deduce the contradiction that $ p_1 $ is special.
\end{proof}


Now we would like to generalize Proposition 5 from \cite{Miniatur} for $ z $-rings:

\begin{theorem} \label{theo_main}
	Let $ z,n \in \mathbb{N} $ and $ M = q_1^{r_1}q_2^{r_2}  \prod_{l=1}^{n} {p_l}^{k_l}  \in \mathbb{N} \setminus \left\{ 0,1 \right\} $ be factorized where $ r_1,r_2 \in \{ 0,1 \} $, $ k_j \in \mathbb{N} \setminus \left\{ 0 \right\} $, $ p_j = \alpha_j \overline{\alpha_j} $ are non-special, irregular elements with different absolute values for $ j = 1,2, \dots, n $ and $ q_1,q_2 \in \mathbb{Z}[i_z]$ are each either a special element or equal to $ 1 $ such that their absolute values are also different from each other. Then there are $ \lceil 2^{n-1} \rceil $ positive, primitive solutions to $ x^2 + zxy + y^2 = M $. Moreover, if there is a $ q_j \neq 1 $ such that $ r_j \in \mathbb{N} $ would be at least equal to $ 2 $, then there would be no primitive solution. Also if $ \mathbb{Z}[i_z] $ is a unique factorization domain and if we allow $ M >0 $ to be divisible by any regular element, then there is no primitive solution to $ x^2 + zxy + y^2 = M $.
\end{theorem}	
	

Observe that the irregular and special elements do not have to be positive.

\begin{proof}
At first let $ n > 0 $. We will assume that two such special elements with different absolute values $ q_1,q_2 \in \mathbb{Z}[i_z] $ exist as for the other cases we can just ignore them and their factors. Then we can find two (associated) prime elements $ \beta_j \in \mathbb{Z}[i_z] $ such that $ q_j = \beta_j \overline{\beta_j} $ for $ j = 1,2 $. 

Let $ I,I' $ be a partition of the set $ \{ 1,2, \dots,n \} $. We can factorize
	\begin{align*}
		M = q_1^{r_1}q_2^{r_2} \prod_{l=1}^{n} {{p_l}^{k_l}} &= \left(\beta_1^{r_1}\beta_2^{r_2}\prod_{l=1}^{n} {\alpha_l}^{k_l} \right)\overline{\left(\beta_1^{r_1}\beta_2^{r_2}\prod_{l=1}^{n} {\alpha_l}^{k_l} \right)} \\
		&= \left(\underbrace{\beta_1^{r_1}\beta_2^{r_2}\prod_{l \in I} {\alpha_l}^{k_l}\prod_{l \in I'} {\overline{\alpha_l}}^{k_l}}_{\eqqcolon \alpha_I} \right) \left(\underbrace{ \overline{\beta_1^{r_1}\beta_2^{r_2}}\overline{\prod_{l \in I} {\alpha_l}^{k_l}\prod_{l \in I'} {\overline{\alpha_l}}^{k_l}}}_{\eqqcolon \overline{\alpha_I}} \right)
	\end{align*}	
and for each $ I $ we find a unit $ \varepsilon_I $ such that $ \varepsilon\alpha_I $ is a positive, primitive solution to $ x^2 +zxy + y^2 = M $ 
for $ r_1 = 0 = r_2 $ by \Cref{theo_4_miniatur}. In case $ r_1 $ or $ r_2 $ are not both zero, then we might have to adjust $ \varepsilon_I $ by \Cref{coro_pos_sol} such that our solution is still positive. Moreover, it is easy to see that the solution remains primitive because if $ q_j $ is a special element, then it cannot happen that $ q_j $ divides the real and imaginary part of $ \alpha_I $ because then $ q_j^2 \mid M $ by \Cref{lem_trinity_divisability} which is a contradiction to $ r_j \leq 1 $.

On the other hand, if $ \left\{ a,b \right\} $ is a positive, primitive solution to $ x^2 + zxy + y^2 = M $, then $ \left(a+bi_z\right)\overline{\left(a+bi_z\right)} = M $. Since $ a,b $ are coprime, we find $ I \subset \{ 1,2, \dots,n \} $ such that $ a + bi_z = \varepsilon \alpha_I $ for a unit $ \varepsilon \in \mathbb{Z}[i_z] $. This works because $ \N \left( a+bi_z \right)= M $ and
$ a+bi_z $ is only divisible by irregular elements which divides $ M $. Moreover, by \Cref{coro_pos_sol} we find a unique unit $ \varepsilon $ such that $ \varepsilon \alpha_I $ has positive real and imaginary part.

Now we would like to show that $ x^2 + zxy + y^2 = M $ has exactly $ 2^{n-1} $ solutions. Let $ I_1,I_2 \subseteq \{ 1,2, \dots,n \} $ and assume that $ \varepsilon_1
\alpha_{I_1} $ and $ \varepsilon_2
\alpha_{I_2} $ represent the same positive, primitive solution for units $ \varepsilon_1, \varepsilon_2 \in \mathbb{Z}[i_z] $. Then we have
	$$ \left\{ \mathrm{Re}\left( \varepsilon_1
\alpha_{I_1} \right), \mathrm{Im}\left( \varepsilon_1
\alpha_{I_1} \right) \right\} = \left\{ \mathrm{Re}\left( \varepsilon_2
\alpha_{I_2} \right), \mathrm{Im}\left( \varepsilon_2
\alpha_{I_2} \right) \right\} $$
and so either 
$$ \varepsilon_1
\alpha_{I_1} = \varepsilon_2
\alpha_{I_2} $$
if $ \mathrm{Re}\left( \varepsilon_1
\alpha_{I_1} \right) = \mathrm{Re}\left( \varepsilon_2
\alpha_{I_2} \right) $ and $ \mathrm{Im}\left( \varepsilon_1
\alpha_{I_1} \right) = \mathrm{Im}\left( \varepsilon_2
\alpha_{I_2} \right) $ or 
$$ \varepsilon_1
\alpha_{I_1} = \varepsilon_2
\widetilde{\alpha_{I_2}} $$
if $ \mathrm{Re}\left( \varepsilon_1
\alpha_{I_1} \right) = \mathrm{Im}\left( \varepsilon_2
\alpha_{I_2} \right) $ and $ \mathrm{Im}\left( \varepsilon_1
\alpha_{I_1} \right) = \mathrm{Re}\left( \varepsilon_2
\alpha_{I_2} \right) $.	
 		
In the case 	$ \varepsilon_1
\alpha_{I_1} = \varepsilon_2
\alpha_{I_2} $ we have that $ I_1 = I_2 $ because $ \alpha_{I_1} $ and $ \alpha_{I_2} $ have a unique prime factorization. If $ \varepsilon_1
\alpha_{I_1} = \varepsilon_2
\widetilde{\alpha_{I_2}} = \varepsilon_2 i_z\overline{\alpha_{I_2}} $, then we conclude that $ I_1 $ and $ I_2 $ are a partition of $ \{ 1,2, \dots,n \} $. 

On the other hand, if $ I_1 $ and $ I_2 $ equal, then trivially $ \alpha_{I_1} = \alpha_{I_2} $ and there is a unique unit $ \varepsilon \in \mathbb{Z}[i_z] $ such that $ \varepsilon\alpha_{I_j} $ has positive real and imaginary part for each $ j = 1,2 $. If $ I_1 $ and $ I_2 $ are a partition of $ \{ 1,2, \dots,n \} $, then $ \alpha_{I_1} = \overline{\alpha_{I_2}} $. Moreover, there are unique units $ \varepsilon_j \in \mathbb{Z}[i_z] $ such that $ \varepsilon_j\alpha_{I_j} $ are positive solutions for $ j = 1,2 $ by \Cref{coro_pos_sol}. Observe that
$$ \widetilde{\varepsilon_2\alpha_{I_2}} = i_z \overline{\varepsilon_2\alpha_{I_2}} = i_z \overline{\varepsilon_2}\alpha_{I_1} $$
and by \Cref{lemma_mirror} we deduce that $ \varepsilon_2\alpha_{I_2} $ and $ i_z \overline{\varepsilon_2}\alpha_{I_1} $ is the same positive solution for the above Diophantine equation. By the uniqueness of the unit $ \varepsilon_1 $ we conclude that $ \varepsilon_1 = i_z \overline{\varepsilon_2} $ and so we have 
$$ \widetilde{\varepsilon_2\alpha_{I_2}} = \varepsilon_1 \alpha_{I_1} $$
which shows that the unique associated positive solutions to $ \alpha_{I_1} $ and $ \alpha_{I_2} $ are the same.   Thus, we have exactly $ 2^{n-1} $ choices of $ I $ such that the resulting positive, primitive solutions are different form each other.



Now we consider the case $ n = 0 $. Then for at least one $ j $ we have $ r_j > 0 $ because $ M \in \mathbb{N} \setminus \left\{ 0,1 \right\} $. We have to show that there exist exactly one positive, primitive solution. Observe that $ \beta_1^{r_1}\beta_2^{r_2} $ satisfies the Diophantine equation and there exist a unit $ \varepsilon \in \mathbb{Z}[i_z] $ such that $ \varepsilon \beta_1^{r_1}\beta_2^{r_2} $ is a positive solution. Moreover, this solution must be primitive because otherwise a prime $ p \in \mathbb{Z} $ would divide $ M $ by \Cref{lem_trinity_divisability} so $ p \in \left\{ \pm q_1, \pm q_2 \right\} $, but then either $ q_1^2 \mid M $ or $ q_2^2 \mid M $ again by \Cref{lem_trinity_divisability} which is a contradiction because $ M = \N \left( \beta_1^{r_1}\beta_2^{r_2} \right) = q_1^{r_1}q_2^{r_2} $.

Conversely, let $ x + yi_z \in \mathbb{Z}[i_z] $ be a positive, primitive solution to the above Diophantine equation. Then 
	$$ \beta_j^{r_j} \mid \N \left( x + yi_z \right) = q_1^{r_1}q_2^{r_2} $$
which implies $ \beta_j^{r_j} \mid x + yi_z $ by \Cref{prop_norm_prime_argument} and the fact that $ \beta_j, \overline{\beta_j} $ are associated for $ j = 1,2 $. Hence, $ \beta_1^{r_1}\beta_2^{r_2} $ and $ x + yi_z \in \mathbb{Z}[i_z] $ are associated positive, primitive solutions and they must be equal by \Cref{coro_pos_sol}.

For the rest of the proof we will assume $ n \in \mathbb{N} $ without any restriction. Now we would like to show that if some $ r_j > 1 $, then there is no primitive solution to $ x^2+zxy+y^2 = M $. If so we have
$$ \beta_j^2\overline{\beta_j}^2 = q_j^2 \mid M = x^2 + zxy + y^2 =  \left( x+yi_z \right) \overline{\left(x+yi_z\right)}, $$
which implies that at least one of the factors on the right-hand side can be divided by two of the factors on the left-hand side. Since these factors on the left-hand side are all associated, we find a unit $ \varepsilon_j \in \mathbb{Z}[i_z] $ such that their product is equal to $ \varepsilon_jq_j $. Without loss of generality we can now assume that $ \varepsilon_jq_j \mid \left( x+yi_z \right) $, i.e. also $ q_j \in \mathbb{Z} $ divides $ x + yi_z $ in $ \mathbb{Z}[i_z] $. By \Cref{lem_trinity_divisability} this means that $ q \mid x $ and $ q \mid y $ which is a contradiction to our assumption that $ x+yi_z $ is a primitive solution to the Diophantine equation above.



Assume that $ p \in \mathbb{Z}[i_z] $ is regular and $ \mathbb{Z}[i_z] $ is a unique factorization domain. Then $ p \in \mathbb{Z}[i_z] $ is irreducible and therefore prime. If 
	$$ p \mid M = x^2 + zxy + y^2 =  \left( x+yi_z \right) \overline{\left(x+yi_z\right)}, $$
then again, without loss of generality, $ p \mid x+yi_z $ which implies $ p \mid x $ and $ p \mid y $ by \Cref{lem_trinity_divisability} and so $ x+yi_z $ is not a primitive solution.
\end{proof}

Observe that the discussion after  \Cref{example_pos_sol_49} about the number and the construction of solutions in a chosen quadrant to a Diophantine equation $ x^2 + zxy + y^2 = M $ if $ M \in \mathbb{Z} $ is a product of irregular primes in $ \mathbb{Z}[i_z] $ and $ z \in \mathbb{Z} $ works analogously. The system of different association equivalence classes is generalized in the notation from \Cref{theo_main} to all possible elements we can produce in the product $ \left(\beta_1^{r_1}\beta_2^{r_2}\prod_{l=1}^{n} {\alpha_l}^{m_l}{\overline{\alpha_l}}^{k_l-m_l} \right) $
for all choices of $ m_l \in \left\{ 0,1, \dots, k_l \right\} $. Of course we should not forget to take the symmetry into consideration, i.e. some of the generated solutions in different quadrants might essentially not be different from each other. Observe that elements in the system of representatives are primitive if and only if we have $ m_l \in \left\{ 0,k_l \right\} $ for all $ l = 1,2, \dots, n $ and $ r_1,r_2 \in \left\{ 0,1 \right\} $.
\end{subsection}

\begin{subsection}{The ring $ \mathbb{Z}[i_3] $ and solutions to $ x^2 + 3xy + y^2 = M $}

In this section we will consider a concrete example, namely the $ z $-ring $ \mathbb{Z}[i_3] $ where we can apply what we especially discussed in the last section. The goal is to understand how many positive, primitive solutions the Diophantine equation 
	$$ x^2 + 3xy + y^2 = M $$
has for any $ M \in \mathbb{N} $. As mentioned before it is known that this ring is a unique factorization domain. Recall that the special elements are $ -5,5 \in \mathbb{Z}[i_3] $ and that there exist also units with norm equal to $ -1 $. At first we would like to determine the regular and irregular elements of $ \mathbb{Z}[i_3] $. For the next statement we use a proof method similar to \cite[p. 21-29]{Aigner_Ziegler}.


\begin{theorem} \label{theo_char_primes_3}
	A prime $ p \in \mathbb{Z} $ is of the form $ 5n \pm 1 $ for $ n \in \mathbb{Z} $ if and only if $ p \in \mathbb{Z}[i_3] $ is irregular, but non-special. Furthermore, the regular elements in $ \mathbb{Z}[i_3] $ are prime.
\end{theorem}

\begin{proof}
Observe that there are no $ x,y \in \mathbb{Z} $ such that 
	$$ x^2 + 3xy + y^2 \equiv 2 \pmod{5} $$
or 
	$$ x^2 + 3xy + y^2 \equiv 3 \pmod{5} $$
hold. Therefore the primes in $ \mathbb{Z} $ for which we can find $ x,y \in \mathbb{Z} $ such that 
	$$ x^2 + 3xy + y^2 = p $$
are either of the form  $ 5n \pm 1 $ for $ n \in \mathbb{Z} $ or equal to $ \pm 5 \in \mathbb{Z} $ where the latter ones are the special elements. The goal is now to show that for all primes of the above form we really find $ x,y \in \mathbb{Z} $ such that $ x^2 + 3xy + y^2 = p $.

At first we will show that for each positive prime $ p \in \mathbb{Z} $ such that $ p \equiv \pm 1 \pmod{5} $ we find an element $ s_p \in \mathbb{N} $ such that 
	$$ s_p^2+3s_p+1 \equiv 0 \pmod{p} .$$
This is equivalent of showing the existence of an element $ s_p \in \mathbb{N} $ such that $$ \left( 2s_p+3 \right)^2 \equiv 5 \pmod{p} $$	
and this is equivalent to finding $ X_p \in \mathbb{N} $ such that $ X_p^2 \equiv 5 \pmod{p} $ holds. By quadratic reciprocity we get that the answer of this question is equivalent of finding $ X_p \in \mathbb{N} $ such that 
	$$ X_p^2 \equiv p \equiv \pm 1 \pmod{5} .$$
And this is clearly possible for $ X_p \equiv 1 \pmod{5} $ and $ X_p \equiv 2 \pmod{5} $. Therefore the existence of such an $ s_p \in \mathbb{N} $ is showed.


Let $ p \in \mathbb{Z} $ of the form $ p = n^2 \pm 1 $ be arbitrary and $ s_p \in \mathbb{N} $ such that 
	$$ s_p^2+3s_p+1 \equiv 0 \pmod{p}. $$ 
Consider the pairs $ \left(x,y\right) \in \mathbb{N} \times \mathbb{N} $ with $ 0 \leq x,y < \sqrt{p} $. Observe that the number of such pairs is strictly greater than $ p $ which allows us to use the pigeon-hole principle: There are at least two such pairs $ \left( x_1,y_1 \right), \left( x_2,y_2 \right) \in \mathbb{N} \times \mathbb{N} $ such that 
	$$ x_1-s_py_1 \equiv x_2 - s_py_2 \pmod{p} $$
holds. Now we define $ x \coloneqq x_1-x_2  \in \mathbb{Z} $ and $ y \coloneqq y_1-y_2  \in \mathbb{Z} $. Observe that $ \vert x \vert , \vert y \vert < \sqrt{p} $ and $ \left( x,y \right) \neq \left( 0,0 \right) $ because the pairs $ \left( x_1,y_1 \right) $ and $ \left( x_2,y_2\right) $ are different from each other. Therefore we get that 
	$$ 0 < \vert x^2 + 3xy + y^2 \vert < 5p $$ 
(remember that $ \N \left( x,y \right) = 0 $ if and only if $ x=0 $ and $ y=0 $ by \Cref{Nz_lemma}).

Moreover, we can also show that $ p \mid x^2 + 3xy + y^2 $. Indeed, we have
	$$ x \equiv x_1-x_2 \equiv s_py_1 -s_py_2 \equiv s_py  \pmod{p} $$
and therefore 
	$$ x^2+3xy+y^2 \equiv y^2 \left( s_p^2+3s_p+1 \right) \equiv 0 \pmod{p} $$
holds. We conclude that $ p \mid x^2 + 3xy + y^2 $. Combined with $ 0 < \vert x^2 + 3xy + y^2 \vert < 5p $ we deduce 
	$$ x^2 + 3xy + y^2 \in \left\{\pm p, \pm 2p, \pm 3p, \pm 4p \right\} .$$

In $ \mathbb{Z}[i_3] $ we find units $ \varepsilon \in  \mathbb{Z}[i_z] $ such that $ \N \left( \varepsilon \right) = -1 $. Therefore we can assume that $ x^2 + 3xy + y^2 \in \left\{ p,  2p, 3p, 4p \right\} $ because if not, we can consider the real and imaginary part of $ \varepsilon \left( x+yi_z \right) $.

Since there are no $ x,y \in \mathbb{N} $ such that $ x^2 + 3xy + y^2 \equiv 2 \pmod{5} $ or $ x^2 + 3xy + y^2 \equiv 3 \pmod{5} $ and $ 2p \equiv \pm 2 \pmod{5} $, $ 3p \equiv \pm 3 \pmod{5} $ we can assume 
	$$ x^2 + 3xy + y^2 \in \left\{ p, 4p \right\} .$$

If $ x^2 + 3xy + y^2 = 4p $, then we have $ x^2 + 3xy + y^2 \equiv 0 \pmod{4} $. However, if $ x^2 + 3xy + y^2 \equiv 0 \pmod{4} $ holds, then we necessarily have that $ 2 \mid x $ and $ 2 \mid y $. In this case we can set $ x' \coloneqq \tfrac{x}{2} $ and $ y' \coloneqq \tfrac{y}{2} $ and we have $ x'^2 + 3x'y' + y'^2 = p $.

Hence, we always find $ x,y \in \mathbb{N} $ such that $ x^2+3xy+y^2 = p $ if $ p \equiv 5n \pm 1 $ is a positive prime. Then $ \varepsilon \left( x+yi \right) $ has norm $ -p $ and so its real and imaginary part satisfy the equation $ x^2 + 3xy + y^2 = -p $.

If a prime $ p \in \mathbb{Z} $ is not of the above form, then it is irreducible (and so regular) in $ \mathbb{Z}[i_z] $ by \Cref{lemma_ir_and_reg_primes}. Hence, $ p \in \mathbb{Z}[i_z] $ is a prime element because $ \mathbb{Z}[i_z] $ is a unique factorization domain.
\end{proof}

With \Cref{theo_main} and \Cref{theo_char_primes_3} we can conclude the following:

\begin{corollary}
	Let $ M = 5^r \left( \prod_{l=1}^{n} {p_l}^{k_l} \right) \in \mathbb{N} \setminus \{0,1 \} $ be factorized, $ n \in \mathbb{N} $, $ k_j \in \mathbb{N} \setminus \{ 0 \} $, $ n_j \in \mathbb{Z} $ and either $ p_j = 5n_j + 1 \in \mathbb{Z} $ or $ p_j = 5n_j - 1 \in \mathbb{Z} $ be pairwise different primes for $ 1 \leq j \leq n $ where $ r \in \{0,1 \} $. Then there are $ \lceil 2^{n-1} \rceil $ positive, primitive solutions to $ x^2 + 3xy + y^2 = M $. Otherwise, i.e. if $ M $ is divisible by at least one prime not in the above form or $ r > 1 $, then there is no primitive solution.
\end{corollary}

\end{subsection}
\end{section}


	








\begin{section}{Attachment}

The next few statements are proved without using the fact that the irreducible factors of irregular elements are prime in the corresponding $ z $-rings. Since the following methods are very basic and it was a surprise to me that it was possible to proceed with them I decided to put them here instead of erasing them even if we did not use them for the previous part.

\begin{lemma} \label{lem_represent_p_-p}
	Let $ p \in \mathbb{Z} $ be a prime and assume that the Diophantine equation 
	$$ x^2+zxy + y^2 = p $$
can be solved for $ x,y \in \mathbb{Z} $. Then $ x^2+zxy + y^2 = -p $ is solvable if and only if $ z \in \left\{ -3,3 \right\} $.
\end{lemma}

\begin{proof}
	Assume that $ a,b,c,d \in \mathbb{Z} $ with $ a^2+zab+b^2 = p $ and $ c^2+zcd+d^2 = -p $. Therefore we get
	$$ \left( ab+cd \right) z = -\left(a^2+b^2+c^2+d^2\right) .$$
Inserting this in the first equation multiplied by $ \left( ab+cd \right) $ we have
	$$ a^2  \left( ab+cd\right)-ab\left(a^2+b^2+c^2+d^2\right) + b^2\left( ab+cd\right) = p\left( ab+cd\right) $$
which is equivalent to 
	$$\left( ad-bc \right) \left( ac-bd \right) = p \left( ab+cd \right) .$$
Hence, either $ p \mid ad-bc $ or $ p \mid ac-bd $.
Now we have
	$$ \left( a+bi_z \right)_p \left( c+di_z  \right)_{-p} = ac-bd + \left( ad+bc+zbd \right)i_z $$
and 	
	$$ \left( a+bi_z  \right)_p \left( d+ci_z  \right)_{-p} = ad-bc + \left( ac+bd+zbc\right)i_z	 .$$	
Observe that the norm of the left-hand side of both equations is equal to $ -p^2 $ and one of the real parts of them on the right-hand side must by divisible by $ p $. Hence, also the imaginary part has to be divisible by $ p $ by \Cref{lem_trinity_divisability}.

Thus, without loss of generality, we can assume that
	$$ \frac{ac-bd}{p} + \frac{ad+bc+zbd}{p}i_z \in \mathbb{Z}[i_z] $$
and its norm must be $ -1 $, so we conclude that $ z \in \left\{ -3,3 \right\} $ by \Cref{coro_z=3_sol_-1}.

In case $ z \in \left\{ -3,3 \right\} $ and $ a^2+zab+b^2 = p $ we can find a unit $ \varepsilon \in \mathbb{Z}[i_z] $ with $ \N \left( \varepsilon \right) = -1 $. For example, set $ \varepsilon := 1-i_z\in \mathbb{Z}[i_3] $ or $ \varepsilon := 1+i_z\in \mathbb{Z}[i_{-3}] $ depending whether $ z = 3 $ or $ z = -3 $. Then the element $ \varepsilon \left(a+bi_z \right) \in \mathbb{Z}[i_z] $ has norm $ -p $ and so its real and imaginary parts solve the Diophantine equation $ x^2+zxy + y^2 = -p $.
\end{proof}

\begin{proposition} \label{prop_mod_sol}
	Let $ z \in \mathbb{N} $, $ k \in \mathbb{N} \setminus \{ 0 \} $ and $ p \in \mathbb{Z}[i_z] $ irregular, but not special. Then there is at most one unique positive, primitive solution to 
	$$ x^2 + zxy + y^2 = p^k .$$
\end{proposition}

\begin{proof}
	 Assume that we have two positive solutions $ \{ a,b \} $ and $ \{ c,d \} $ to the Diophantine equation $ x^2 + zxy + y^2 = p^k $, i.e. we have 
	 $$ a^2 + zab + b^2 = p^k = c^2 + zcd + d^2 = p^k .$$
The aim is to show that $ \{ a,b \} = \{ c,d \} $. For this we will transform the equations. By the above equations we get
	$$ zab = p^k-a^2-b^2 $$
and
	$$ z \left( ab -cd \right) = c^2+ d^2-a^2-b^2 .$$
By multiplying $ \left( ab -cd \right) $ and $ ab $ to the above equations, respectively, we deduce
	$$ \left( p^k-a^2-b^2 \right) \left( ab-cd \right) = zab \left( ab-cd \right) = ab \left( c^2+d^2-a^2-b^2\right) .$$
The first and the last part of the equation is finally equivalent to the identity
	$$ p^k \left( ab-cd \right) = \left( ac-bd \right) \left( bc-ad \right) .$$

By the above identity we get that $ p \mid ac-bd $ or $ p \mid bc-ad $. For $ k > 1 $ it could also happen that $ p \mid ac-bd $ and $ p \mid bc-ad $. We will show that this is never the case. Assume $ p \mid ac-bd $ and $ p \mid bc-ad $, then we have that $ ac \equiv bd \pmod{p} $ and $ bc \equiv ad \pmod{p} $ and so we get
	$$ a^2d \equiv abc \equiv b^2d \pmod{p} .$$
Since our solutions are primitive, we have that $ p \nmid d $ and so $ a^2 \equiv b^2 \pmod{p} $ holds. Moreover, we have that $ p \mid a^2-b^2 = \left( a+ b \right) \left( a- b \right) $ and so either $ p \mid a + b $ or $ p \mid a - b $. Hence, either $ \left(a+b \right)^2 \equiv 0 \pmod{p} $ or $ \left(a-b \right)^2 \equiv 0 \pmod{p} $ what we will denote by 
	$$ \left( a \pm b \right)^2 \equiv 0 \pmod{p} $$ 
to consider both cases simultaneously. Thus,
	$$ a^2 \pm 2ab + b^2 \equiv 0 \equiv a^2 + zab + b^2 \pmod{p} $$
holds true which implies 
	$$ \left( z \mp 2 \right)ab \equiv 0 \pmod{p} .$$
However, since our solution is primitive, we have that $ p \nmid a $ and $ p \nmid b $. Moreover, by \Cref{lemma_associate}, $ p \nmid  z \mp 2 $ because $ p $ is not special by our assumption. Thus, we get a contradiction.

Therefore, without loss of generality (or by exchanging $ a $ and $ b $), we can assume that $ p^k \mid ac-bd $. Since $ 0 < a,b,c,d < \sqrt{p^k} $ we also have that $ \vert ac-bd \vert < p^k $ and hence 
	$$ ac-bd = 0 $$
which shows that $ ab-cd = 0 $ by the above identity.

Now we show that the solutions are equal if $ ab-cd = 0 $. Consider
	$$ a^2 + zab + b^2 = c^2 + zcd + d^2, $$
subtract $ zab = zcd $ and multiply on both sides by $ b^2 $. We get
	$$ \left(ab\right)^2 + b^4 = b^2 \left( c^2+ d^2 \right) .$$
If we replace $ ab $ by $ cd $ we obtain
	$$ b^4 - \left( c^2+ d^2 \right)b^2 + c^2d^2 = \left(b^2-c^2 \right) \left(b^2-d^2 \right) = 0 .$$
Hence, we conclude that either $ b = c $ or $ b = d $ because the solutions are positive which implies $ \{ a,b \} = \{ c,d \} $. 
\end{proof}

The next proposition is a generalization of \Cref{prop_mod_sol} and a weaker version of \Cref{theo_main}.

\begin{proposition}
	Let $ z \in \mathbb{N} $ and $ n,k_l \in \mathbb{N} \setminus \{ 0 \} $, $ p_l \in \mathbb{Z}[i_z] $ irregular and non-special for all $ l = 1, \dots, n $. Then there are at most $ 2^{n-1} $ positive, primitive solutions to 
	$$ x^2 + zxy + y^2 = \prod_{l=1}^{n} {p_l}^{k_l}. $$
\end{proposition}

\begin{proof}
	Let $ M \coloneqq \prod_{l=1}^{n} {p_l}^{k_l} $ and assume that we have two positive, primitive solutions $ \{ a,b\} $ and $ \{ c,d\} $ to the Diophantine equation $ x^2+zxy+y^2 = M $. As in the proof of \Cref{prop_mod_sol} the following identity must hold:
	$$ M \left( ab-cd \right) = \left( ac-bd \right) \left( bc-ad \right) $$
	As before we can show that the solutions have to be equal if there exists $ p_l $ such that $ p_l \mid ac-bd $ and $ p_l \mid bc-ad $. On the other hand, if $ M \mid ac-bd $ or $ M \mid bc-ad $, then the two solutions must also be equal (this follows by the same arguments used in the proof of \Cref{prop_mod_sol}).
	
By interchanging $ a,b,c,d $ if necessary, we can always assume that $ p_1^{k_1} \mid ac-bd $. However, for all the other primes we have $ 2^{n-1} $ choices whether $ {p_l}^{k_l} \mid ac-bd $ or $ {p_l}^{k_l} \mid bc-ad $ for each $ l \in \{ 2, \dots, n \} $. This means if we fix $ \{ a,b\} $ as solution to $ x^2+zxy+y^2 = M $, we can compare it with other solutions. The only thing we have to prove now is that two solutions $ \{ c_1,d_1\} $ and $ \{ c_2,d_2\} $ of $ x^2+zxy+y^2 = M $ are identical if for each $ l \in \{ 1,2, \dots, n\} $ either $ {p_l}^{k_l} \mid ac_j-bd_j $ or $ {p_l}^{k_l} \mid bc_j-cd_j $ for $ j=1,2 $.

Assume that $ {p_l}^{k_l} \mid ac_j-bd_j $ for some $ l $, then we have $ ac_j \equiv bd_j \pmod{p} $ and so we get
	$$ a \left( d_1c_2-c_1d_2 \right) \equiv b \left(d_1 d_2 - d_1d_2 \right) \equiv 0 \pmod{p} .$$
Since $ p \nmid a $ we get that $ {p_l}^{k_l} \mid d_1c_2-c_1d_2 $. 
On the other hand, if $ {p_l}^{k_l} \mid bc_j-ad_j $ for some $ l $, then we have $ bc_j \equiv ad_j \pmod{p} $ and so we get
	$$ a \left( d_1c_2-c_1d_2 \right) \equiv b \left(c_1 c_2 - c_1c_2 \right) \equiv 0 \pmod{p} $$
and we also get 	$ {p_l}^{k_l} \mid d_1c_2-c_1d_2 $ because $ p \nmid a $.

Therefore we have that $ M \mid d_1c_2-c_1d_2 $ because the above step holds for all $ l \in \{ 1,2, \dots, n\} $. By the same arguments as in the proof of \Cref{prop_mod_sol} we get that $ \{ c_1,d_1 \} = \{ c_2,d_2 \} $ which shows that the Diophantine equation $ x^2 + zxy + y^2 = \prod_{l=1}^{n} {p_l}^{k_l} $ cannot have more than $ 2^{n-1} $ positive, primitive solutions.
\end{proof}


\end{section}



\bibliography{Geometry_and_Diophantine_equations_Maximal_rat_circ_point_sets_Chris_Busenhart}
\bibliographystyle{plain}

\end{document}